\documentclass[10pt,reqno]{amsart}

\usepackage{color}
\usepackage[dvipsnames]{xcolor}

\usepackage[english]{babel}

\usepackage{bbm,amsmath,amsthm,epsfig,latexsym,marvosym,mathrsfs}
\usepackage{amsfonts}
\usepackage{amsmath,stackengine}
\usepackage{amssymb,a4wide}
\usepackage{esint,xfrac}

\usepackage[colorlinks=true,linkcolor=blue,citecolor=red]{hyperref}

\definecolor{myc}{RGB}{36,107,2}

\def\Z{\mathbb{Z}}
\def\R{\mathbb{R}}
\def\N{\mathbb{N}}

\def\F{\mathcal{F}}

\def\L{\mathcal{L}}

\def\d{\,\mathrm{d}}

\def\ca{\mathbbmss{1}}

\def\wt{\stackrel{*}{\rightharpoonup}}

\def\spt{\mathrm{spt}\,}
\def\e{\varepsilon}

\def\p{\mathbf{p}}

\def\v{\mathbf{v}}
\def\h{\mathbf{h}}
\def\p{\mathbf{p}}

\newtheorem{theorem}{Theorem}[section]
\newtheorem{corollary}[theorem]{Corollary}
\newtheorem{proposition}[theorem]{Proposition}
\newtheorem{lemma}[theorem]{Lemma}
\newtheorem{remark}[theorem]{Remark}

\theoremstyle{definition}
\newtheorem{definition}[theorem]{Definition}

\numberwithin{equation}{section}
\numberwithin{figure}{section}

\definecolor{mygreen}{RGB}{21,118,40} 

\author{Andrea Braides}
\address{SISSA, via Bonomea 265, Trieste}
\email{braides@mat.uniroma2.it}

\author{Marco Caroccia}
\address{Politecnico di Milano,  Piazza Leonardo da Vinci,  32,  20133 Milano (MI)}
\email{marco.caroccia@polimi.it}

\makeindex
\title[Dirichlet energy on Poisson point clouds]
{Asymptotic behavior of the Dirichlet energy\\ on Poisson point clouds}
\thanks{Preprint SISSA 06/2022/MATE}

\begin{document}

\begin{abstract} We prove that quadratic pair interactions for functions defined on planar Poisson clouds and taking into account pairs of sites of distance up to a certain (large-enough) threshold can be almost surely approximated by the multiple of the Dirichlet energy by a deterministic constant. This is achieved by scaling the Poisson cloud and the corresponding energies and computing a compact discrete-to-continuum limit. In order to avoid the effect of exceptional regions of the Poisson cloud, with an accumulation of sites or with `disconnected' sites, a suitable `coarse-grained' notion of convergence of functions defined on scaled Poisson clouds must be given.

 \end{abstract} 

\maketitle
 
\tableofcontents

\section{Introduction}
The object of this paper is an analysis of the asymptotic behaviour of quadratic energies on Poisson random sets.  Loosely speaking  such sets are characterized by the property that the number of their points contained in a given set has a Poisson probability distribution, and that the random variables related to disjoint sets are independent. Even more loosely, in average the number of points contained in a set is proportional to the Lebesgue measure of the set. We denote by $\eta$ such a random set.

In order to define some almost-sure properties of $\eta$ we use a discrete-to-continuum approach that has been fruitfully used to derive continuum theories from microscopic interactions (see \cite{alicandro2004general}). A simple interpretation of this method is as a finite-difference approximation. If  $\eta$ is a deterministic periodic locally finite discrete set in $\mathbb R^d$, then we can consider real-valued functions $u:\eta\to\mathbb R$ and quadratic interaction potentials between points on $\eta$. The corresponding Dirichlet-type energy is 
\begin{equation}\label{eqq1}
\sum_{\langle x,y\rangle} (u(x)-u(y))^2,
\end{equation}
where $\langle x,y\rangle$ indicates summation over nearest-neighbouring pairs $(x,y)$ in $\eta$. We can then introduce a small parameter $\e$ and scale both the environment and the energies accordingly; namely, considering $u:\e\,\eta\to\mathbb R$ and 
\begin{equation}\label{eqq2}
\sum_{\langle x,y\rangle} \e^{d-2}(u(x)-u(y))^2,
\end{equation}
now summing over nearest-neighbouring pairs $(x,y)$ in $\e\,\eta$. By letting $\e\to 0$ we obtain a limit continuum energy, of the form
\begin{equation}\label{eqq3}
\int_{\mathbb R^d} {\bf A}\nabla u\cdot\nabla u\,{\rm d}x,
\end{equation}
where the matrix $\bf A$ carries information about the microstructure of the original set $\eta$. Note that in order to perform this passage to the limit we have to embed our energies in a common environment identifying functions on $\eta$ with suitable interpolations. The limit is meant in the sense of $\Gamma$-convergence, which implies that minimum problems for the limiting energies are approximations of the discrete ones, and can also be performed `locally',  by considering interactions only for $x\in \e\eta\cap U$ for a fixed open set $U$. 

In order to define analogs of \eqref{eqq1} and \eqref{eqq2} for the realization of a Poisson cloud $\eta$ we face a choice regarding what to consider as `interacting sites', whether nearest neighbours in the sense of Voronoi cells or points `close' in the sense of the ambient space. For the random set $\eta$ these two choices are not equivalent since nearest-neighbouring points on $\eta$ may be indeed arbitrarily distant in the ambient space, and conversely a very small distance between points of $\eta$ does not ensure that they are nearest neighbours in $\eta$. We choose the second option, which also seems closer to applications;  namely, we introduce an {\em interaction radius} $\lambda>0$ and consider the energies 
\begin{equation}\label{fed}
\F_{\e}(u)=\sum_{\substack{x,y\in Q\cap\e\eta, \\ |x-y|<\e\lambda}} \e^{d-2}(u(x)-u(y))^2,
\end{equation}
defined for $u:\e\,\eta\cap Q \to\mathbb R$, where $Q$ is the unit coordinate cube centered in $0$ (for ease of notation we treat only this case, which anyhow, up to scaling and localization, implies the result for any bounded Lipschitz open set in $\mathbb R^d$). Note that, if a $\Gamma$-limit of such energies does exist then, thanks to the invariance properties by rotations of $\eta$, it must be
a multiple of the Dirichlet integral (i.e., $\bf A$ in \eqref{eqq3} is equal to a multiple of the identity matrix), which we expect to be almost surely deterministic.

\begin{figure}[h!]
 \includegraphics[scale=0.50]{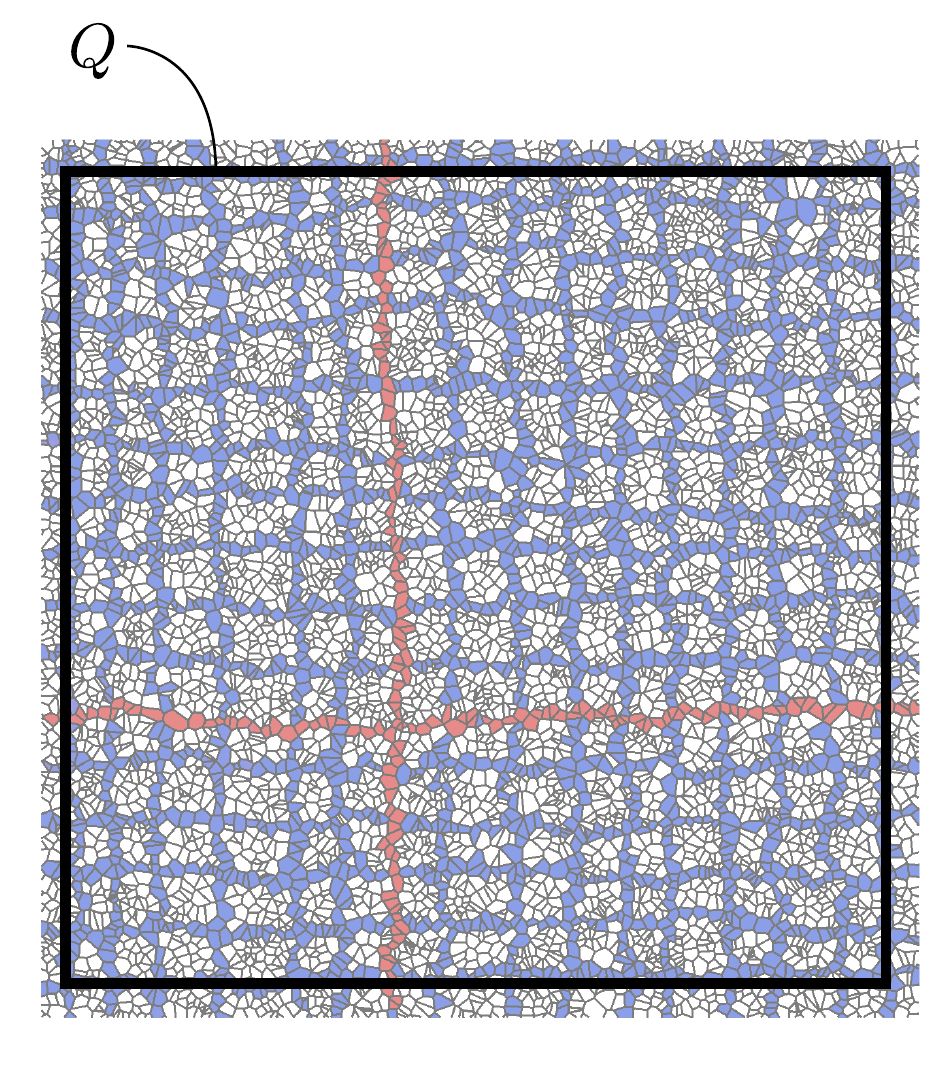}
\caption{A representation of a regular grid in $Q$ with two `paths' highlighted.}
\label{pic:grid}
\end{figure}

The main issue in proving the convergence of $\F_{\e}$ consists in providing a suitable notion of convergence for discrete functions $u_\e$ to a continuum parameter $u$, for which a compactness theorem can be proved under an assumption of boundedness of the energies. While for periodic $\eta$ we can use piecewise-constant interpolations on Voronoi cells (or, equivalently, piecewise-affine interpolations on the related Delaunay triangulation), for a Poisson point process we cannot control the behaviour of such interpolation due to the presence of arbitrarily large and arbitrarily small Voronoi cells. Nevertheless, we can prove that the union of `regular' Voronoi cells with suitably controlled dimensions form an infinite connected set in which we find `paths' of cells such that also cells at distance $\lambda$ are regular. This is done by exploiting a Bernoulli site-percolation argument. In the planar case $d=2$ the complement of this set of Voronoi is composed of isolated sets with controlled dimensions, so that the ambient space can be thought of as a ``perforated domain'', in which we do not have a control of the discrete functions only in isolated `holes' of controlled size. This allows to define a suitable convergence by choosing a subset $G_\e$ of $\e\eta$ composed of paths mentioned above whose union  $\mathcal{V}_{\e}(G_{\e})$ has the geometry of a square grid (see Fig.~\ref{pic:grid}) 
 We thus use these grids to define a suitable convergence notion: given a sequence of function $u_{\e}: \e\,\eta \rightarrow \R$ we say that $u$ is the $L^2$-limit of $u_{\e}$ if 
\begin{equation}\label{cone-f}
\int_{\mathcal{V}_{\e}(G_{\e})} |\hat{u}_{\e}-u|^2\d x \rightarrow 0;
\end{equation}
namely, if the $L^2$-distance between the piecewise-constant extensions $\hat{u}_{\e}$ of $u_{\e}$ and $u$, restricted to the Voronoi cells of the grid $G_{\e}$ vanishes as $\e\rightarrow 0$. 
Regular grids allow also to give a meaningful notion of boundary-value problems; in particular we can consider affine boundary condition as in the cell problems
	\begin{equation}\label{eqn:cellProb-1}
	\mathrm{m}(\xi ;TQ):=\inf\left\{ \sum_{x\in \eta \cap (TQ) } \sum_{y\in \eta\cap  B_\lambda(x)} |v(x)-v(y)|^2 \ \left| \begin{array}{c}
	v:\eta\rightarrow \R\\
 \ v(x)=\xi \cdot x  \ \ \text{\rm for all $x\in \eta$ such that}  \\ 
 \mathrm{dist}(x,\partial (TQ) )\leq 2\lambda
\end{array}	\right. \right\}.
	\end{equation}
Using subadditive ergodic theorems we then can prove that, if $\xi\neq0$, the constant $\Xi$ given by 
	\[
	\Xi:=\lim_{T\rightarrow +\infty} \frac{\mathrm{m}(\xi;T Q)}{T^2 |\xi|^2}
	\]
exists and is deterministic. Moreover, by the invariance properties of $\eta$ it does not depend on $\xi$. This allows to state and prove the main result of the paper, which is the almost sure $\Gamma$-convergence in the planar case $d=2$ of functionals \eqref{fed}  to 
\begin{equation}\label{fed0}
\F (u)=\Xi\int_Q |\nabla u|^2{\rm d}x
\end{equation}
with respect to the convergence in \eqref{cone-f}. Note that more in general we may consider energies of the form 
\begin{equation}\label{fed-a}
\F_{\e}^a(u)=\sum_{x,y\in Q\cap\e\eta} \e^{d-2} a\Bigl({x-y\over\e}\Bigr)(u(x)-u(y))^2,
\end{equation}
with $a$ positive and with compact support, recovering the case in \eqref{fed} as a special case when $a(\xi)$ is the characteristic function of the ball centered in $0$ and radius $\lambda$. If $a$ is radially symmetric the same limit result holds with obvious modifications in the statements.

The convergence theorem can be compared with various results in the literature. Our result is inspired by the recent paper \cite{braides2020homogenization}, where perimeter energies on Poisson random sets are considered. In that context a simpler compactness result can be obtained with respect to the convergence is measure of sets, by using a covering lemma that ensures that the energy cannot concentrate on non-regular Voronoi cells. In the present context, this would correspond to an extension theorem for Sobolev functions from regular sets, which seems hard to obtain due to the random geometry of clusters of non-regular Voronoi cells. Furthermore, we can compare our approach to that in \cite{BBL,ACG}, where a notion of stochastic lattice $\eta$ is given for which energies of the form \eqref{fed-a} can be considered. Differently from Poisson random sets, stochastic lattices are more regular, in that all Voronoi cells have controlled dimension and hence are regular in the terminology above, a condition that seems a considerable restriction in terms of applications. 
The regularity of the lattice implies that functionals $\F_{\e}^a$ are coercive with respect to the $L^2$ convergence of piecewise-constant interpolations on Voronoi cells. Conversely, in general the limits of functionals $\F_{\e}^a$, which exist under ergodicity and stationarity assumptions, are not isotropic even if $a$ is radially symmetric, except for specially constructed examples \cite{Ruf}. More general random distributions of sites have been considered within problems in Machine Learning by D.~Slepcev {\em et al.} \cite{caroccia2020mumford,GT1,GT2} (see also the references therein). In their approach the convergence is given in terms of suitable interpolations of discrete functions using Optimal-Transport techniques. The presence of non-regular Voronoi cells is mitigated by considering kernels $a^\e$ with increasing support as $\e\to 0$, which also allow to obtain isotropy in the limit (see also \cite{BraSol19} for variational limits using a coarse-graining approach).  Energies \eqref{fed-a} have a continuum approximation in terms of a convolution double integral, for which random homogenization has been considered in \cite{braides2019homogenization} (see also \cite{Ansini2020}). We note that the existence of regular paths can be proved in any dimension $d\ge 2$, but if $d>2$ the geometry of regular grids can be thought as a set of ``fibers'' rather than a perforated domain. We believe that the same asymptotic result holds but with an even more complex argument. Finally, we mention, following a remark by D.~Slepcev in a private communication, that our results may also have practical implications for the study of the graph Laplacian. Namely, one can show that  if one drops all of the eigenmodes of the graph Laplacian on a low degree random geometric graph where the eigenvector has large $L^\infty$ norm (in other words concentrates at few nodes) one still recovers the continuum spectrum.
\smallskip

We briefly outline the plan of the paper. In Section \ref{sct:2} we introduce the necessary definitions and notation for Poisson point clouds. This allows to give a definition of Dirichlet energy on a Poisson cloud $\eta$ and to prove some asymptotic properties as the Poisson set is scaled by a small parameter $\eta$ and correspondingly the energies $\mathcal F_\e$ as in \eqref{fed}. We introduce the notation for Voronoi cells and define grids of paths of regular cells ${\mathcal G}_{\e,t}$ depending on an additional parameter $t>0$. This allows to define the convergence of piecewise-constant functions $u_\e$ on Voronoi cells to a continuum function $u$ as the successive convergence of averages of $u_\e$ computed on ${\mathcal G}_{\e, t}$ at the ``mesoscopic scale'' $t$ to piecewise-constant functions $u^t$ defined on a square grid and then of such $u^t$ to a limit $u$ as $t\to+\infty$ (Definition \ref{def:conv}).  This is proved to be equivalent to the $L^2$ convergence on grids as in \eqref{cone-f}, and actually independent of the choice of the family of grids. In Section \ref{sct:3} we state the main results of the paper.  Theorem \ref{thm:Compactness} states the pre-compactness the sense of the convergence above of sequences with equi-bounded Dirichlet energy; Theorem \ref{thm:Gamma} is an almost-sure homogenization result characterizing the $\Gamma$-limit of Dirichlet energies as a deterministic quantity as in \eqref{fed0}. Section \ref{sct:comp} is devoted to the proof of the Compactness Theorem, based on the geometric properties of the grids that allow the use of Poincar\'e inequalities. In Section \ref{sct:gamma} we prove the Homogenization Theorem. The lower bound is obtained by using the Fonseca and M\"uller blow-up method, which is possible thanks to the use of cut-off functions that are locally constant close to non-regular Voronoi cells. The construction of recovery sequences is also made possible thanks to these `regular' cut-off functions. Finally Section \ref{sct:app} (the Appendix) contains the proof of the existence of regular grids.

%

\section{Notation and preliminaries}\label{sct:2}
In this section we introduce the main ingredients required to perform our analysis. For the sake of simplicity, and since our analysis will take place in this context, we always consider the ambient space dimension to be two-dimensional, even though some of the definitions and results can be extended to general space dimension.
\subsection{General notation}
We let $Q:=\left(-\sfrac{1}{2},\sfrac{1}{2}\right)^{2}$ denote the unit coordinate square centered in $0$. We will also write $Q_r=r Q$, $Q_r(x)=x+rQ$.  The same notation applies to $B_r(x)$,  being $B$ the unit ball of $\R^2$.  For a Radon measure $\nu\in \mathcal{M}^+(\R^2)$ the space $L^2(A; \nu)$ is defined as the space of all measurable functions $u:\spt(\nu)\cap A\rightarrow \R$ such that 
	\[
	\int_A |u(x)|^2 \d \nu(x)<+\infty.
	\]
When $\nu=\L^2$, the Lebesgue measure, we simply write $L^2(A)$.  We denote by $\mathrm{Bor}(A)$ the collection of all Borel subsets of $A$. The notation $|E|$ stands for \textit{the Lebesgue measure of $E$}. For a set $A$ and  for $t>0$ we define
	\[
	(A)_t:= \{x\in \R^2 \ | \ \mathrm{dist}(x,A)\leq t\}.
	\]
If no confusion arises, the notation $\{ x_i\in X_i\}_{i\in I}$ is used for a family $x_i$ indexed by $I$, such that $x_i\in X_i$

\subsection{Poisson point clouds}
Some basic properties of the stationary stochastic point process called \textit{Poisson point cloud} are here recalled.  A complete treatment of this subject can be found in \cite{daley2003introduction,schneider2008stochastic}.  In order to formally introduce this notion, we consider the family $\mathbf{N}_{s}$ of {\em simple measures};  \textit{i.e.}, 
	\begin{align*}
		\mathbf{N}_{s}&:=\biggl\{\left. \sum_{i\in I} \delta_{x_i}\in \mathcal{M}^+(\R^2) \ \right| \ \{x_i\}_{i\in I} \in \R^2, \ x_i\neq x_j \ \text{for all $i,j\in I$,  and $I$ a subset of $\N$} \biggr\}
	\end{align*}
Here and in the sequel, $\delta_x$ is the \textit{Dirac delta} at $x$.  For any Borel set $E\in \mathrm{Bor}(\R^2)$ and $k\in \N$ we define the subset $\mathbf{A}_{E,k}:=\{\mu\in \mathbf{N}_{s} \ | \ \mu(E)=k\} $ of $\mathbf{N}_{s}$ and consider the $\sigma$-Algebra	$\mathcal{N}$ generated by
$\{\mathbf{A}_{E,k} \ | \ E\in \mathrm{Bor}(\R^2), \ k\in \N\}$.
\begin{definition}\label{def:pp}
A \textit{Poisson point process on $\R^2$ with intensity $\gamma$} is a random element $\eta$ on $ (\mathbf{N}_s,\mathcal{N})$; that is, a map from a probability space $(\Omega,\mathcal{F},\mathbb{P})$ onto  $(\mathbf{N}_s,\mathcal{N})$, such that
	\begin{itemize}
	\item[1)] $\mathbb{P}( \eta\in \mathbf{A}_{E,k})=\displaystyle \frac{(\gamma|E|)^k}{k!}e^{-\gamma|E|}$;
	\item[2)] denoted by $\eta(B):(\Omega,\mathcal{F},\mathbb{P})\rightarrow \N$ the random variable induced by $\eta$ when fixing $B$, (namely $\eta(B):=\eta(\omega)(B)$ for $\omega\in \Omega$) then, for any $B_1,\ldots,B_m$ pairwise disjoint Borel sets we have that $\eta(B_1),\ldots,\eta(B_m)$ are independent.
	\end{itemize}
\end{definition}

For any Poisson point process $\eta$ we can observe that its probability distribution $\mathbb{P}_{\eta}$ on $(\mathbf{N}_s,\mathcal{N})$ satisfies
	\[
	\mathbb{P}(\eta(A)=k):=\mathbb{P}_{\eta}(\mathbf{A}_{E,k})=\mathbb{P}(\eta\in \mathbf{A}_{E,k}):=\mathbb{P}(\{\omega\in \Omega \ | \ \eta(\omega)\in \mathbf{A}_{E,k}\})= \frac{(\gamma|E|)^k}{k!}e^{-\gamma|E|}.
	\]
We will often make use of the notation $x\in \eta$,  or $x\in \eta(\omega)$ by meaning that $x\in \spt(\eta(\omega))$ for some realization $\omega\in \Omega$. Accordingly, $x\in \e\, \eta(\omega)$ will stand for $x\in \e\,\spt (\eta(\omega))$.\\

 Definition \ref{def:pp} implies (see \cite[Proposition 8.3]{last2017lectures}) in particular that the Poisson point process on $\R^2$ with intensity $\gamma$ is \textit{stationary}:  if we define $\tau_x \eta(A):= \eta(A+x)$, then $\tau_x \eta$ is equal in distribution to $\eta$ for any $x\in \R^d$.  This implies in particular that $\mathbb{P}(\eta(\R^2)<+\infty)=0$ (see \cite[Proposition 8.4]{last2017lectures}).  In the sequel,  whenever we speak of a Poisson point process we always mean \textit{a Poisson point process on $\R^2$ with intensity $\gamma$}. 
\subsection{Dirichlet energy on point clouds}
Let $\eta$ be a Poisson point process.  Without loss of generality, we fix the intensity to be $\gamma=1$ and we carry out our analysis on the unit square $Q$. This is not a restriction since we may localize our energies on regular subsets of $Q$ where the analysis applies unchanged, while we can deal with arbitrary bounded regular open sets by rescaling them to subsets of $Q$.

  Let $\lambda>0$ be a fixed parameter (the {\em interaction radius}) and for $u\in L^2(Q;\eta_{\e})$,  for a subset $A\subset Q$ define
	\[
	\F_{\e}(u;A):=\sum_{x\in \eta_{\e}\cap A} \sum_{y\in \eta_{\e} \cap B_{\lambda \e }(x)} |u(x)-u(y)|^2,
	\]  
where we have set
	\[
	\eta_{\e} (A):=\eta(\e^{-1} A), \ \ \ \spt(\eta_{\e} )=\e\,\spt(\eta).
	\]
	
{\begin{remark}\label{rmk:subadditivity}\rm The definition of $\F_\e$ takes into account the values of $u$ in a $\lambda\e$-neighbourhood of $A$. As a consequence, the energy $\F_\e$ is subadditive on essentially disjoint sets,  namely
	\[
	\F_{\e} (u; A\cup B)\leq \F_{\e}(u;A)+\F_{\e} (u;B) \ \ \ \text{for all $A,B\subset\R^2$,  $|A\cap B|=0$}.
	\]
Indeed, the energy of a function $u$ on some open set $A$, takes into account also the contribution of all those points 
	\[
	\mathcal{B}A:=\{ y\in (\R^2\setminus A)\cap \eta_{\e}  \ | \ \text{there exists $x\in A\cap \eta_{\e}$ such that $|x-y|\leq \lambda\e$}\}.
	\]
Note that $\F_\e$ is not local, since we may not have $\F_{\e}(u;A)=\F_{\e}(v;A)$ if $u=v$ on $A$ due to
the interaction around boundary points.  However,  if $u=v$ on $A$ we can infer that
	\[
	\F_{\e}(u;A\setminus (\partial A)_{\lambda\e} )=\F_{\e}(v;A\setminus (\partial A)_{\lambda\e} ).
	\]
\end{remark}}
We now give some estimates on the asymptotic behavior of $\F_\e$ and $\eta_\e$ by means of a kind of (spatial) Mean Ergodic Theorem.  We show that,  almost surely,  the average number of points in some open set $A$,  lying at a distance less then $\lambda \e$ can be bounded from above by the Lebesgue measure of $A$ (the proof of Proposition \ref{propo:count} follow the lines of the proof of \cite[Theorem 8.14]{last2017lectures}).
\begin{proposition}\label{propo:count}
There exists a constant $C$,  depending on $\lambda$ only,  such that the following property holds almost surely:
	\begin{equation}\label{eqn:estOnNumb}
  \limsup_{\e\rightarrow 0} \sum_{x\in \eta_{\e} \cap A}\e^2 \eta_{\e}(B_{\lambda\e}(x)) \leq C |A| \ 
	\end{equation}
 for any $A\subseteq Q$ with Lipschitz boundary.
\end{proposition}
\begin{proof}
{We consider a sequence $\{\e_n\}_{n\in \N}$ such that $\e_n\rightarrow 0$ and
\begin{equation}\label{stisosa}
	\frac{1}{C}\leq \frac{\e_n}{\e_{n+1}}\leq C,
\end{equation}
where $C>1$ is a fixed universal constant. } For an open set $A$ we define the following objects
	\begin{align*}
	\mathcal{I}(A):=\{i\in \lambda\Z^2 \ | \  Q_{\lambda}(i)\cap A\neq \emptyset\},\quad 
	\mathcal{Q}(A):=\{Q_\lambda(i) \ | \ i\in \mathcal{I}(A)\},\quad
	N(A):=\#(\mathcal{I}(A)),
	\end{align*}
and consider, for any $i\in \lambda\Z^2 $ the random variable $X_i:=\eta(Q_{\lambda}(i))^2$. Note that
	\[
	\mathbb{E}(X_i)=e^{-\lambda^2}\sum_{k=1}^{+\infty} k^2 \frac{(\lambda)^{2k}}{k!}=\tau_\lambda<+\infty.
	\]
Let $R:=(0,a)\times (0,b)$ for $a,b\in \R$.  Then, we can relabel each square $Q_{\lambda}(i) \in \mathcal{Q}(\e_n^{-1} R)$ in such a way that
	\[
	\frac{1}{N(\e_n^{1-} R)}\sum_{i\in \mathcal{I}(\e_n^{1-}R)} \eta(Q_\lambda(i) )^2=\frac{1}{N(\e_n^{-1} R)}\sum_{k=1}^{N(\e_n^{-1}R)} X_{i_k}.
	\]
If we now invoke the law of large numbers (see for instance \cite[Theorem B.11]{last2017lectures}), we have that
	\[
	\lim_{R\to+\infty}\frac{1}{N(\e_n^{-1} R)}\sum_{i\in I(\e_n^{-1}R)} \eta(Q_\lambda(i) )^2=\tau_\lambda=\mathbb{E}(X_1)
	\]
almost-surely.  Since $
	\e_n^2 N(\e_n^{-1} R) \rightarrow \frac{ |R|}{\lambda^2}
$
and
$
	\eta_{\e_n}(Q_{\lambda\e_n}(\e_n i))= \eta(Q_\lambda(i) ) 
$,
we have that
	\begin{align*}
	\sum_{i\in \mathcal{I}(\e_n^{-1} R)} \e_n^2\ \eta_{\e_n}(Q_{\lambda\e_n }(\e_n i))^2= \frac{(\e_n^2 N(\e_n^{-1}R))}{N(\e_n^{-1}R)}\sum_{i\in \mathcal{I}(\e_n^{-1} R) } \eta(Q_{\lambda}(i)) \rightarrow \frac{|R|}{\lambda^2} \tau_\lambda
	\end{align*}
for almost all $\omega\in \Omega$. Let now
	\begin{align*}
	\mu_{\e}(R)&:=\sum_{i\in \mathcal{I}(\e^{-1}R)} \e^2 \eta_{\e}(Q_{\lambda\e} (\e i) )^2, \ \ \ \mu_n(R):=\mu_{\e_n}(R),\\
\mathcal{R}_0&:=\{R=[0,p]\times[0,q] \ | \ p,q\in \mathbb{Q}, \ p,q,\leq 2\}\\
\Omega_0&:=\left\{\omega\in \Omega \ \Bigl| \ \mu_n(R)\rightarrow \frac{|R|}{\lambda^2} \tau_\lambda \ \text{for all $R\in \mathcal{R}_0$}\right\}.
	\end{align*}
Since $\mathcal{R}_0$ is a countable family of rectangles we have that $\mathbb{P}(\Omega_0)=1$. Let now $R:=[p,p']\times [q,q']\subset Q$ with $p,p',q,q'\in \mathbb{Q}$ and define
	\begin{align*}
	R_1&:=[0,p'] \times [0,q], \qquad R_2:=[0,p]\times [0,q']\\
	R_3&:=[0,p]\times [0,q], \qquad\ 
	R_4:=[0,p']\times [0,q'].
	\end{align*}
so that $
	R= R_4\setminus (R_1\cup R_2), \ |R|=|R_4|-|R_1|-|R_2|+|R_3|$.
Moreover
	$$
	\mu_n(R):= \mu_n(R_4) - \mu_n(R_1) - \mu_n(R_2)+ \mu_n(R_3)   + s_n,
	$$
	with
	$$s_n \leq  C\sum_{j=1}^4\sum_{i\in \mathcal{I}(\e_n^{-1} \partial R_j)} \e_n^2 \eta_{\e_n} (Q_{\lambda\e_n}(\e_n i))^2 .
$$
Since $s_n\rightarrow 0$ we immediately have
$
	\mu_n(R)\rightarrow \frac{|R|}{\lambda^2} \tau_\lambda
$.	
In particular, having defined
	\begin{align*}
	\mathcal{R}:=\{[p,p']\times[q,q'] \subset Q \ | \ p,p',q, q'\in \mathbb{Q}\},
	\end{align*}
	we have
$
		\mu_n(R)\rightarrow  \frac{|R|}{\lambda^2} \tau_\lambda \ \ \text{for all $\omega\in \Omega_0$  and $R\in \mathcal R$}
$.

{Let $R\in \mathcal{R}$,  let $\omega\in \Omega_0$ be a realization and let $\{\tilde{\e}_k=\tilde{\e}_k(\omega)\}_{k\in \N}$ be a sequence along which 
	\[
	\limsup_{\e\rightarrow 0}\sum_{i\in \mathcal{I}(\e^{-1} R)}  \e^2  \eta(\omega)(Q_{\lambda}(i))^2=\lim_{k\rightarrow +\infty}  \sum_{i\in \mathcal{I}(\tilde{\e}_k^{-1} R)}	\tilde{\e}_k^2 \eta(\omega)(Q_{\lambda}(i))^2.
	\]
Consider $\e_{n_k} \leq \tilde{\e}_k\leq \e_{n_k+1}$.  By \eqref{stisosa},  we can find $\tilde{R}\in    \mathcal{R}_0$ with $|\tilde{R}|\leq C|R|$ (for a universal constant $C$) and such that
$
	  \tilde{\e}_k^{-1}R\subset 	\e_{n_k}^{-1}\tilde{R} \ \ \ \text{for all $k\in \N$}$.
%
Note that $\mathcal{I}(\tilde{\e}_k^{-1}R)\subset \mathcal{I}(\e_{n_k}^{-1}\tilde{R})$ and thus
	\begin{align*}
	\sum_{i \in \mathcal{I}( \tilde{\e}_k^{-1}R)}  \tilde{\e}_k^2 \eta (\omega) (Q_{\lambda}(i ))^2	\leq\sum_{i \in \mathcal{I}(\e_{n_k}^{-1}\tilde{R})}  \e_{n_k+1}^2 \eta(\omega) (Q_{\lambda}( i ))^2
	\leq C\sum_{i \in \mathcal{I}(\e_{n_k}^{-1}\tilde{R})}  \e_{n_k}^2 \eta(\omega) (Q_{\lambda}( i ))^2.
	\end{align*}
By taking the limit and exploiting that $\tilde{R}\in \mathcal{R}_0$, $\omega\in \Omega'$, we achieve
	\[
 \limsup_{\e\rightarrow 0}\sum_{i\in \mathcal{I}(\e^{-1} R)}  \e^2  \eta(\omega)(Q_{\lambda}(i))\leq C |R|,
	\]
where $C$ depends on $\lambda$ only.  In particular we have that 
	\[
	 \limsup_{\e\rightarrow 0} \mu_{\e}(R)= \limsup_{\e\rightarrow 0} \sum_{i\in \mathcal{I}(\e^{-1} R)} \e^2 \eta(Q_{\lambda}(i))\leq C |R|
	\]
 almost surely. For any open set $A\subset Q$ and for any $\delta>0$ we can find a finite covering of disjoint rectangles $\{R_k\}_{k=1}^{N_k}\subset \mathcal{R}$ with
$
	\sum_{k=1}^{N_k}|R_k|\leq |A|+\delta
$.
Then, since $\mu_{\e} $ is sub-additive on disjoint sets, we conclude that
	\begin{align*}
	\mu_{\e} (A)\leq\sum_{k=1}^{N_k} \mu_{\e} (R_k) \Rightarrow  \limsup_{\e\rightarrow 0}\mu_{\e} (A)\leq C \sum_{k=1}^{N_k} |R_k|
	\end{align*}
almost surely. Now, since
	\begin{align*}
 \sum_{x\in A\cap \eta_{\e} } 	\e^2 \eta_{\e}(B_{\lambda\e}(x))&\leq C \sum_{i\in \mathcal{I}(\e^{-1} A)} \e^2 \eta_{\e} (Q_{\lambda \e} (\e i))^2 =C \mu_{\e}(A)
	\end{align*}
for a universal constant independent of $\lambda,n,\omega$ we obtain the claim for all $\omega\in \Omega_0$, which has probability $1$.}
\end{proof}
	\begin{corollary}\label{cor:est}
Fix $\lambda>0$.  There exists a constant $C$ depending on $\lambda$ only such that almost surely it holds
\begin{align}
\limsup_{\e \rightarrow 0}  \F_{\e} (u;A )&\leq C \int_A |\nabla u|^2\d x \label{eqn:estOnGrad}\\
\lim_{\e \rightarrow0}\e^2\sum_{x\in \eta_{\e} \cap A} u(x)^2\eta_{\e}(B_{\lambda\e}(x)) &\leq C  \int_A |u|^2\d x\label{eqn:estOnL2}
\end{align}
for any $u\in C^1(\overline{Q})$ and for any $A\subset \subset Q$ with Lipschitz boundary.
	\end{corollary}
	\begin{proof}
Fix $A\subset Q$. Let $\mathcal{Q}_m:=\{Q^m_i\}_{i\in \N}$ be a dyadic division of $Q$ in squares of size $\frac{1}{m}$. Since $u\in C^1$ for any fixed $\delta$ we can find
$K\ge1$ such that
	\begin{align*}
	\Bigl| \max_{x\in Q^m_i}\{u(x)\}- \min_{x\in Q^m_i}\{u(x)\}\Bigr|\leq \delta \ \text{for all $Q^m_i\in \mathcal{Q}_m$ such that $Q^m_i\cap A\neq \emptyset$}\\
	\Bigl| \max_{x\in (Q^m_i)_{\lambda\e_n}}\{|\nabla u(x)|\}- \min_{x\in (Q^m_i)_{\lambda\e_n}}\{|\nabla u(x)|\}\Bigr|\leq \delta \ \text{for all $Q^m_i\in \mathcal{Q}_m$ such that $Q^m_i\cap A\neq \emptyset$}
	\end{align*}
	whenever $m,n\geq K$. 
Moreover $\displaystyle (A)_{\sfrac{2}{m}}\supset \bigcup_{Q^m_i\cap A\neq \emptyset} Q^m_i$ and
	\begin{align*}
	 \int_{(A)_{\sfrac{2}{m}}} u(x)^2\d x\leq \int_A u(x)^2\d x+ \delta, \quad
	  \int_{(A)_{\sfrac{2}{m}}} |\nabla u(x)|^2\d x\leq \int_A |\nabla u(x)|^2\d x+ \delta.
	\end{align*}
Then
	\begin{align*}
	\e^2 \sum_{x\in \eta_{\e}\cap A} \eta_{\e} (B_{\lambda\e}(x)) u(x)^2\leq \sum_{Q^m_i\cap A\neq \emptyset} \max_{x\in Q^m_i}\{ u(x)^2\} \sum_{x\in \eta_{\e}\cap Q^m_i} \e^2\eta_{\e}(B_{\lambda\e}(x)).
	\end{align*}
In particular by invoking Proposition \ref{propo:count}, almost surely  we have
	\begin{align*}
	&\hskip-1cm\limsup_{\e\rightarrow 0 } 	\e ^2 \sum_{x\in \eta_{\e}\cap A} \eta_{\e}(B_{\lambda\e}(x)) u(x)^2\\& \leq C \sum_{Q^m_i\cap A\neq \emptyset} \max_{x\in Q^m_i}\{ u(x)^2\} |Q^m_i| =C \sum_{Q^m_i\cap A\neq \emptyset}\int_{Q^m_i} \max_{x\in Q^m_i}\{ u(x)^2\} \d x\\
	&\leq C \sum_{Q^m_i\cap A\neq \emptyset}\int_{Q^m_i} (u(x)^2+\delta^2)\d x= C \int_{(A)_{\sfrac{2}{m}}} u(x)^2\d x+C \delta^2 |Q|\\
	&\leq  C \int_{A}  u(x)^2\d x+C \delta |Q|
	\end{align*}
for a constant that depends on $\lambda$ only.  Since $\delta$ is arbitrary we get \eqref{eqn:estOnL2}.  Also
	\begin{align*}
	\F_{\e} (u;A)&=\sum_{x\in \eta_{\e} \cap A}\sum_{y\in \eta_{\e} \cap B_{\lambda\e}(x)} |u(x)-u(y)|^2\\
	& \leq \sum_{Q^m_i\cap A\neq \emptyset} \sum_{x\in \eta_{\e} \cap Q^m_i}\sum_{y\in \eta_{\e} \cap B_{\lambda\e}(x)} |u(x)-u(y)|^2\\
	& \leq \sum_{Q^m_i\cap A\neq \emptyset} \sum_{x\in \eta_{\e}\cap Q^m_i}\sum_{y\in \eta_{\e} \cap B_{\lambda\e}(x)} |u(x)-u(y)|^2\\
	& \leq \lambda^2\sum_{Q^m_i\cap A\neq \emptyset} \max_{x\in (Q^m_i)_{\lambda\e} } |\nabla u(x)|^2 \sum_{x\in \eta_{\e} \cap Q^m_i}\e^2 \eta_{\e}(B_{\lambda\e}(x)).
	\end{align*}
Then Proposition \ref{propo:count} almost surely yields
	\begin{align*}
	\limsup_{\e \rightarrow 0}	\F_{\e}(u;A)&\leq C\sum_{Q^m_i\cap A\neq \emptyset} \int_{Q^m_i}\max_{x\in (Q^m_i)_{\lambda\e} } \{|\nabla u(x)|^2 \}\d x\\
	&\leq C \int_{A}  |\nabla u(x)|^2  \d x+C\delta^2|Q|
	\end{align*}
for a constant that depends on $\lambda$ only.   The arbitrariness of $\delta$ yields \eqref{eqn:estOnGrad}.
	\end{proof}
  
	\subsection{Voronoi cells and paths}
We now consider $\eta:\Omega\rightarrow \mathbf{N}_s$ a Poisson point process.
For a given realization $\omega$ we identify the \textit{Voronoi cell} of $x\in \eta$ as
	\[
 C(x,\eta):=\{y\in \R^d \ | \ |y-x|\leq |z-x|  \text{ for all $z\in \R^d$} \}.
	\] 
Note that 	
$C(x, \eta_{\e})=\e\, C\Bigl({x\over\e},\eta\Bigr)$.

We say that $x,y$ are \textit{nearest neighbors} if $C(x,\eta)$ shares a common edge with $C(y,\eta)$. In this case we write $\langle x,y\rangle$. Given $x\in \R^2$ we define 
	\[
	\pi_{\eta}(x):=\text{argmin}\{x\in \eta \ | \ |x-y| \}\,,
	\]
where, in case of multiple choices we consider the lexico-graphical order. We say that 
\[
\mathbf{p}(x,y):=\{x_{i_1},\ldots,x_{i_{M}}\}
\] 
is a \textit{path in $\eta$ connecting $x$ to $y$} if $x_{i_m}\in \eta$ for $m=1,\ldots,M$, $x_{i_1}=\pi_{\eta}(x), x_{i_M}=\pi_{\eta}(y)$ and $\langle x_{i_m},x_{i_{m+1}}\rangle$ for $m=1,\ldots M-1$.  Moreover, we let $\ell(\mathbf{p}):=\#(\mathbf{p})$ denote the \textit{length of a path $\mathbf{p}$}. This notion induces a natural metric on $\eta$
	\[
	\tau_{\eta}(x,y):=\min\{ \ell(\mathbf{p}(x,y)) \ | \ \text{$\mathbf{p} (x,y)$ is a path in $\eta$ connecting $x$ to $y$}\}.
	\]

We say that a sub-cluster $S\subset \eta$ is connected if for every $x,y\in S$ there is a path $\mathbf{p}(x,y):=\{x_{i_1},\ldots x_{i_M}\in S\}$ connecting $x$ to $y$. 

\subsection{Piecewise-constant extensions}\label{sbsct:pwCextension}
{
Let $\mathbf{q}\subseteq \eta_{\e} $ 
be a family of points in the point cloud.  We define 
	\[
	\mathcal{V}_{\e} (\mathbf{q}):=\bigcup_{x\in \mathbf{q}} C(x,\eta_{\e}).
	\]
For $u:\eta_{\e}\cap Q\rightarrow \R$ we define the \textit{piecewise-constant extension}
	\[
	\hat{u}: Q\rightarrow \R,  \ \ \ \hat{u}(x):= \sum_{y\in \eta_{\e}\cap Q} u(y)\ca_{C(y,\eta_{\e})}(x).
	\]}
\subsection{Geometric structure of Poisson point processes}\label{sbsct:Grd}
Now we state and prove some statistical properties of Poisson point processes that we find useful in the treatment of the Dirichlet energy on point clouds. For $t\in \R_+$ and a Borel set  $A\subset \R^2$ we define
$$
\mathcal{I}_{t}(A):=\left\{ J \in t\Z^2: \  Q_t(J)\cap A\neq \emptyset \right\},\ 
\mathcal{Q}_{t}(A):=\{Q_J^t:=Q_{t}(J), \ J\in\mathcal{I}_{t}(A) \},\ 
k_t(A):=\sqrt{\#(\mathcal{I}_{t})}.
$$
When it is clear from the context that $A=Q$ we sometimes write $\kappa_t$ in place of $\kappa_t(Q)$.
For $J\in \mathcal{I}_t(Q)$ we have $J=t(i,j) $, $ i,j\in \Z$. Therefore, we set $Q^t_{J}=Q^t_{i,j}\in \mathcal{Q}_t$ the square placed on the $i$-th row and $j$-th column.\smallskip

For $i=1,\ldots,k_t$, $j=1,\ldots, k_t$ we define the \textit{vertical} and \textit{horizontal} rectangles as
	\begin{align*}	R_i^{\textit{h}}(t):=\bigcup_{j=1}^{k_t} Q_{i,j}^t, \ \ \ R_j^{\textit{v}}(t):=\bigcup_{i=1}^{k_t} Q_{i,j}^t.
	\end{align*}
\subsubsection{Selecting a sub-cluster: Percolation Theory}
For a Poisson point process $\eta$ we consider the sub-cluster
	\[
	\eta^{\alpha}(\lambda):=\left\{x\in \eta \ \left| \ \mathrm{in}(C(x,\eta))>\alpha, \ \mathrm{diam}(C(x,\eta))<\alpha^{-1},\    \eta(B_\lambda (x))\leq  \alpha^{-1} \lambda^2 \right. \right\}
	\]
where the {\em in-radius} of a set $\mathrm{in}(A)$ is defined as
	\[
	\mathrm{in}(A):= \sup\{s>0 \ | \ \text{there exists $B_s(x)\subset A$}\}
	\]
	
\begin{definition}\label{def:pathConn}
Fix $\alpha$ and $\lambda$.  Let $\e, t>0 $ be fixed.  We say that a family of vertical and horizontal paths 
	\[
	G_{\e, t}=\{\mathbf{h}^i_m, \mathbf{v}^j_m, \ i,j\in \{1,\ldots,k_t\}, m\in \{1,\ldots, M_{\e,t} \} \}
	\]
is a \text{\em regular $t$-grid with $\Upsilon$ bounds for $\eta_{\e}$}, if the following properties are satisfied 
	\begin{itemize}
	\item[a)] all the paths are in $\e \eta^{\alpha}(\lambda)$;
	\item[b)] for any $m=1,\ldots, M_{\e,t}$, $\mathbf{h}^i_{m}$ connects the two opposite sides of $R_i^h(t)$ of size $t$ and is strictly contained in $R_i^h(t)$; 
		\item[c)] for any $m=1,\ldots, M_{\e,t}$, $\mathbf{v}^j_{m}$ connects the two opposite sides of $R_j^v(t)$ of size $t$ and is strictly contained in $R_j^v(t)$; 
	\item[d)] the following bounds hold for any $i,j\in \{1,\ldots,k_t\}$, $m \in\{1,\ldots,M_{\e,t}\}$, $k_t\in \N$;
	\begin{equation}\label{eqn:unifBnd}
	\begin{array}{c}
	\displaystyle\frac{t}{\Upsilon  \e}\leq \ell(\mathbf{h}^i_m\cap Q_{i,j}^t)\leq \frac{\Upsilon t }{\e}, \ \ \ 
	\displaystyle	\frac{t}{\Upsilon \e}\leq \ell(\mathbf{v}^j_m\cap Q_{i,j}^t)\leq \frac{\Upsilon t}{\e}\\
	\text{}\\
		\displaystyle	\frac{t}{\Upsilon \e }\leq M_{\e,t}\leq \frac{\Upsilon t}{\e}
\end{array}
	\end{equation}
	
	\item[e)] $\mathrm{dist}(\mathbf{h}^i_m,\mathbf{h}^{i'}_s)\geq 3\lambda \e$ and $\mathrm{dist}(\mathbf{v}^j_m,\mathbf{v}^{j'}_s)\geq 3\lambda \e$, for all $i, i',j,  j'\in \{1,\ldots,k_t\}$, $m,s\in\{1,\ldots,M_{\e,t}\}$ (with $m\neq s$ for $i=i'$ or $j=j'$);
	\item[f)] If $x\in (\mathbf{h}^i_m)_{3\lambda\e}\cap \eta_{\e}$, $(x\in (\mathbf{v}^j_m)_{3\lambda\e}\cap \eta_{\e})$ it holds $\eta_{\e}(B_{\lambda\e}(x))\leq \frac{1}{\alpha} \lambda^2$;
	\item[g)] If $x,y\in \mathbf{h}^i_m$, ($x,y\in \mathbf{v}^j_m$) neighboring points then $|x-y|\leq \lambda\e$;
	\end{itemize}
For the family of $(\e,t)$ regular grids with $\Upsilon$ bounds we use the notation
$
	\mathcal{G}_{t}(\Upsilon;\eta_{\e})$.\end{definition}
	
\begin{figure}
 \includegraphics[scale=0.62]{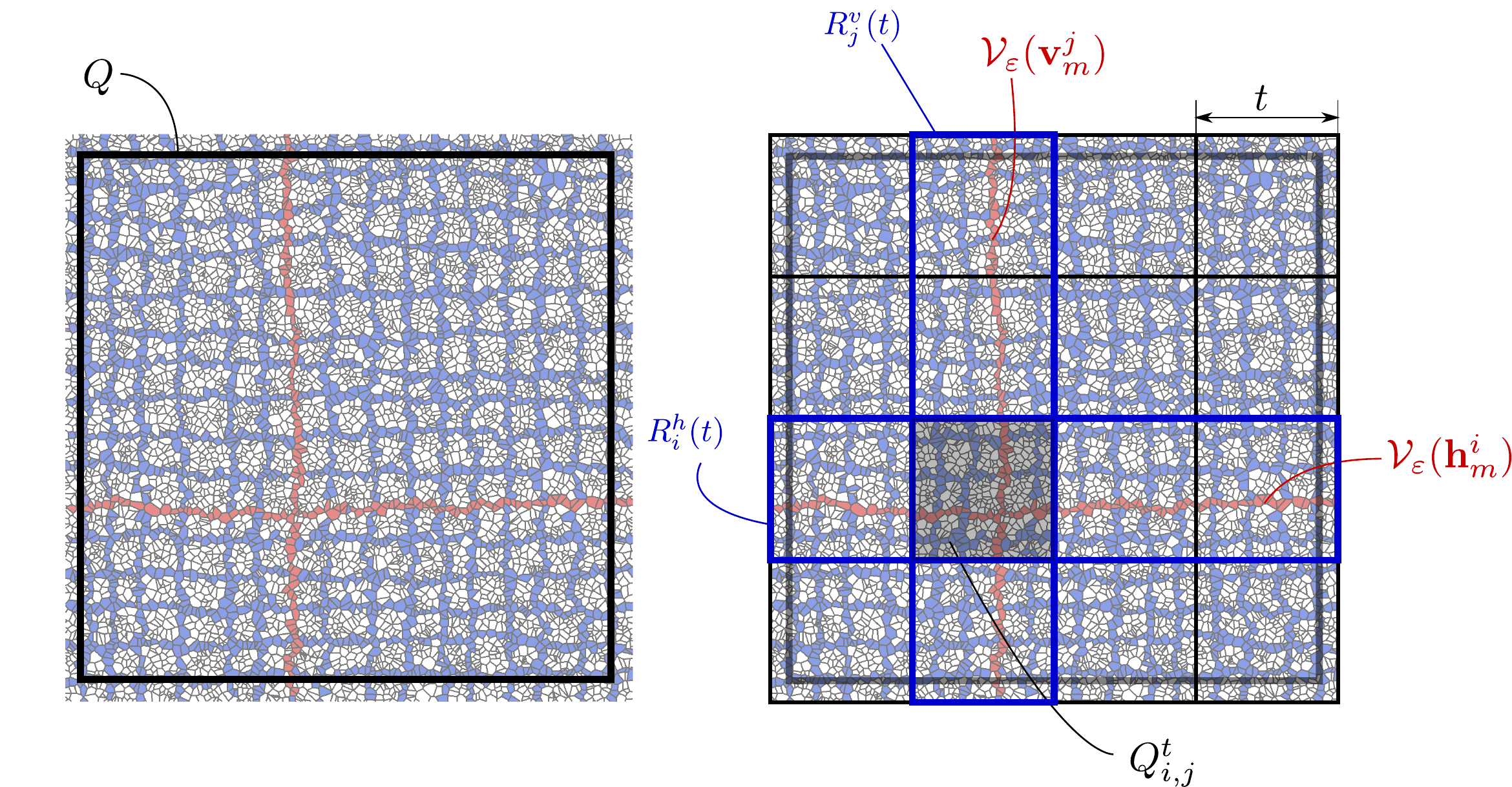}
\caption{A graphic visualization of Definition \ref{def:pathConn}.}
\label{pic:illustration}
\end{figure}
We invite the reader to confront Definition \ref{def:pathConn} with the situation depicted in Figure \ref{pic:illustration},  which has the only purpose of being illustrative.  For any $t>0$ we consider the $k_t^2 \approx 1/t^2$ squares covering $Q$. Then we consider horizontal and vertical rectangles made by union of these squares and labeled suitably (from left to right for the vertical,  and from bottom to top for the horizontal).  We define a grid to be regular if inside each rectangles we can find a certain number of paths of points with properties (a)--(g) of Definition \ref{def:pathConn}.  These paths are further labeled inside each rectangles (from left to right in the vertical rectangles,  and from bottom to top in the horizontal rectangles).  In Figure \ref{pic:illustration} are depicted the sets $\mathcal{V}_{\e}(\mathbf{v}_m^j)$ and  $\mathcal{V}_{\e}(\mathbf{h}_m^i)$ composed of the Voronoi cells of the points of the path.  Theorem \ref{thm:pathConn} ensures the existence of such grids with uniform bounds,  by exploiting a Bernoulli site-percolation argument (see Section \ref{sct:app}).\smallskip

The following result ensures that we can find a universal constant $\Upsilon$ such that $\mathcal{G}_{t}(\Upsilon;\eta_{\e})\neq \emptyset$ for a suitable choice of the parameters $\e, t,\alpha,\lambda$.   It comes as a re-adaptation of a percolation result in \cite{braides2020homogenization} coupled with a technical geometric construction.  We postpone the proof to Appendix \ref{sct:app}.
\begin{theorem}\label{thm:pathConn}
There exists $\alpha_0, \lambda_0$ with the following properties. Provided 
	\begin{equation}\label{bnd:par}
	\alpha\leq \alpha_0,  \ \lambda >\max\left\{\frac{2}{\alpha},\lambda_0\right\},
	\end{equation} 
we can find a constant $\Upsilon>0$, depending on $\alpha$ only,  and,  for any $t\in \R_+$ and for almost all realizations $\omega$, a constant $\e_0(\omega,t)>0$, such that, if $\e \leq \e_0$ then $
\mathcal{G}_{t}(\Upsilon;\eta_{\e})\neq \emptyset$.
\end{theorem}
In order to lighten the notation we now choose $\alpha,\lambda$ satisfying \eqref{bnd:par} and we consider them to be fixed for the rest of the paper.  
%
We are covering $Q$ with squares of size $t$ and in horizontal, and vertical rectangles given by union of squares from the subdivision. We are considering the paths lying inside the rectangles and satisfying geometric properties (a)--(g). In particular, we find sometimes  convenient to work on $Q_{t k_t}\supset Q$, which represent a slightly bigger square and can be divided in exactly $k^2_t$ squares of size $t$.
\subsection{A notion of convergence for functions on Poisson point clouds}
For a fixed $t$ we introduce the set of \textit{simple function on $Q$} as the space
	\[
	X_t:=\biggl\{w\in L^2(Q_{t k_t} ) \ \biggl| \ w=\sum_{i,j=1}^{k_t^2} c_{i,j} \ca_{Q_{i,j}^t}\biggr\}.
	\]
Fix now $(\e,t)$ and consider a grid $G_{\e,t}\in \mathcal{G}_{t}(\Upsilon;\eta_{\e})$. For every $i,j=1,\ldots\ k_t$ we denote by
	\begin{equation}\label{eqn:media}
	(u)^{G_{\e,t}}_{i,j}:=\frac{1}{ \eta_{\e} (G_{\e,t}\cap Q_{i,j}^t)} \sum_{x\in G_{\e,t}\cap Q_{i,j}^t }u(x),
	\end{equation}
and we consider the operator $T^{G}: L^1(Q;\eta_{\e})\rightarrow X_t$ to be
\begin{equation}\label{eqn:operatore}
	T^{G_{\e,t}}(u)(x):=\sum_{i,j=1}^{k_t} (u)^{G_{\e,t}}_{i,j} \ca_{Q_{i,j}^t}(x).
\end{equation}

	\begin{definition}\label{def:conv}
 A sequence of function $u_{\e}: \eta_{\e}\cap Q \rightarrow \R$, $u_{\e}\in L^2(Q;\eta_{\e})$ is said to {\em converge to} $u:Q\rightarrow \R$,  and we simply write $u_{\e}\rightarrow u$, if there exists $\Upsilon \in \R_+$ such that, for any $t\in \R$ and for any sequence of grids $\{G_{\e,t}\in \mathcal{G}_{t}(\Upsilon;\eta_{\e})\}_{\e>0}$, it holds 
	\begin{align*}
	T^{G_{\e,t}}(u_{\e}) \longrightarrow   u^t& \ \ \text{in $L^2(Q_{t k_t})$}  \text{ as $\e\rightarrow 0$,}
	\end{align*}
where $\{u^t\in X_t\}_{t\in \R_+}$ satisfies
	\[
	u^t\rightarrow u  \ \ \text{in $L^2(Q)$} \text{ as $t\rightarrow +\infty$}.
	\]
	\end{definition}

\begin{remark}\rm
We note that in order to have a meaningful notion of convergence, it is necessary to prove that the convergence is well defined; that is, it is independent of the choice of the grids. This will be a consequence of Lemma \ref{lem:comparisongrid}. Also note that this convergence implies (and, along sequences with equibounded energy, is in fact equivalent to) the $L^2$ convergence of the piecewise constant extension $\hat{u}_{\e} $ restricted to the (Voronoi cells of the) grids
	(see Propositions \ref{prop:L2ongrid} and \ref{prop:equivalence})
\end{remark}
	\begin{remark}[Convergence up to subsequences] \label{rmk:subconv}\rm
Note that with this notion of convergence,  a sequence $u_{\e}$ converge to $u$ up to subsequences if there exists $\{\e_n\}_{n\in \N}$ such that for any sequence of regular grids $\{\{G_{\e_n,t}\}_{n\in \N}\in \mathcal{G}_{t}(\Upsilon;\eta_{\e_n} )\}_{t\in \R_+}$ it holds
	\begin{align*}
	T^{G_{\e_n,t}} (u_{\e_n}) \rightarrow  u^{t} \ \ \ \text{as $n\rightarrow +\infty$ in $L^2(Q_{t k_{t}} )$}, \qquad
	u^{t}  \longrightarrow  u \ \ \ \text{as $t\rightarrow 0$ in $L^2(Q)$.}
	\end{align*}
	\end{remark}

\section{The main results}\label{sct:3}
We have now introduced all the basic notation and we are thus ready to state our main results, which regard the asymptotic behavior of $\F_{\e}$.  The first one is a compactness result.

\begin{theorem}[Compactness Theorem] \label{thm:Compactness}
Let $\Omega\supset Q$ and given $\lambda_0>0$ as in Theorem \ref{thm:pathConn},  if $\lambda>\lambda_0$ the following holds.  If $\{u_{\e}\in L^2(\Omega;\eta_{\e})\}_{\e>0} $ is a sequence satisfying
\begin{equation}\label{eqn:CMPbnds}
	\sup_{\e >0} \biggl\{\sum_{x\in \eta_{\e} \cap A } \sum_{y\in \eta_{\e}\cap B_{\lambda\e}(x)} |u_{\e} (x)-u_{\e} (y)|^2+ \sum_{x\in \eta_{\e} \cap  Q} \e^2 u_{\e} ^2(x)\biggr\}<+\infty,
	\end{equation}
where $Q \subset A \subset \Omega$ is any open set strictly containing $Q$, then there exists $u\in W^{1,2}(Q)$ such that $u_{\e}$ converge to $u$,  up to subsequences,  in the sense of Definition \ref{def:conv}.
\end{theorem}

\begin{remark}\rm
Note that in requiring to a sequence to have equibounded energy on $Q$ we need to take into account Remark \ref{rmk:subadditivity} and the fact that the energy is carried also on $\mathcal{B}Q$.  For this reason we state the compactness Theorem in terms of $u_{\e}\in L^2(\Omega;\eta_{\e})$,  with a uniform bound for the energy assumed on a set $A$,  provided $Q\subset A\subset \Omega$.  In other words we are asking to the functions $u_{\e}$ to be defined on a slightly bigger open set than just $Q$. 
\end{remark}
With this compactness theorem in mind,  that will be proved in Section \ref{sct:comp}, the following $\Gamma$ convergence result is then meaningful.

\begin{theorem}[$\Gamma$-convergence] \label{thm:Gamma}
There exists a deterministic constant $\Xi$ such that,  for almost all realizations $\omega$,  the energy $\F_{\e}(\cdot;Q)$ $\Gamma$-converges to 
	\[
	\F(u;Q)=\Xi\int_Q |\nabla u(x)|^2\d x
	\]
in the topology induced by the convergence of Definition \ref{def:conv}.  Namely
\begin{itemize}
\item[ ]$(\liminf)$ if $u_{\e} \rightarrow u$ in the sense of Definition \ref{def:conv} then
	\[
	\liminf_{\e \rightarrow 0} \F_{\e} (u_{\e} ;Q)\geq \Xi \int_Q |\nabla u(x)|^2\ d x;
	\]
	\item[ ]$(\limsup)$  for any $u\in W^{1,2}(Q)$ there exists a sequence $\{u_{\e}\in L^2(Q;\eta_{\e} )\}_{\e>0}$ such that $u_{\e} \rightarrow u$ in the sense of Definition \ref{def:conv} and
		\[
		\limsup_{\e \rightarrow 0} \F_\e(u_{\e};Q)\leq  \Xi \int_Q |\nabla u(x)|^2\ d x.
		\]
\end{itemize}
The constant $\Xi$ is identified by the relation
	\[
	\Xi:=\lim_{T\rightarrow +\infty} \frac{\mathrm{m}(\xi;Q_T)}{T^2 |\xi|^2}
	\]
where, for any open set $A$, $\mathrm{m}(\cdot;A)$ denotes the {\em cell problem}
	\begin{equation}\label{eqn:cellProb}
	\mathrm{m}(\xi ;A):=\inf\left\{ \sum_{x\in \eta \cap A} \sum_{y\in \eta\cap  B_\lambda(x)} |v(x)-v(y)|^2 \ \left| \begin{array}{c}
	v:\eta\rightarrow \R\\
 \ v(x)=\xi \cdot x  \ \ \text{\rm for all $x\in \eta$ such that}  \\ 
 \mathrm{dist}(x,\partial A)\leq 2\lambda
\end{array}	\right. \right\}.
	\end{equation}
\end{theorem}
This theorem will be proved in Section \ref{sct:gamma}.  More precisely,  in Section \ref{sbsct:cellP} we will prove the relation between $\Xi$ and cell problem \eqref{eqn:cellProb}, in Section \ref{sbsct:liminf}  the  lower bound and finally in Section \ref{sbsct:limsup} the upper bound. 

\section{Proof of the compactness theorem}\label{sct:comp}
This section is entirely devoted to the proof of the compactness Theorem \ref{thm:Compactness}.
We first need a few technical lemmas in order to guarantee that the convergence of Definition \ref{def:conv} is well defined and that any sequence with bounded energy is pre-compact. 
\subsection{Preliminary lemmas} 
For $u\in  L^2(Q;\eta_\e)$ we adopt the shorthand
	\[
	D_{\e} u(x):=\sum_{y\in\eta_\e\cap  B_{\lambda\e}(x)} |u(y)-u(x)|^2.
	\]
For any $\e, t$ fixed we set
	\begin{align*}	\mathrm{Hor}(i;G_{\e,t})&:=\{\mathbf{h}^i_1,\ldots, \mathbf{h}^i_{M_{\e, t}}\}\\	\mathrm{Ver}(j;G_{\e,t})&:=\{\mathbf{v}^j_1,\ldots, \mathbf{v}^j_{M_{\e,t}}\}.
\end{align*}	
Next lemma ensures that we can choose a ``skeleton" of $G_{\e,t}$ connecting two neighboring squares and carrying small energy  (see Figure \ref{fig:skeleton}). This lemma is stated and proved for horizontal neighboring squares but it holds also for vertical neighboring squares with obvious changes in the proof.   It will be used in the proof of Theorem \ref{thm:Compactness} (in particular in the proof of Lemma \ref{lem:localEstimate}). 

\begin{figure}
\includegraphics[scale=1]{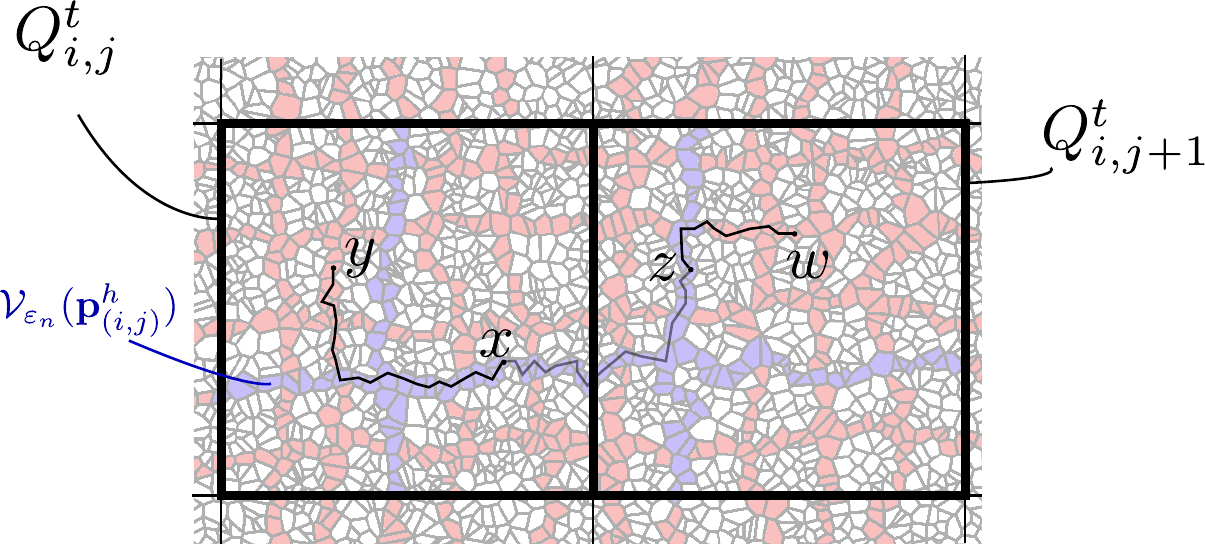}\label{fig:skeleton}
\caption{The ``skeleton" obtained in Lemma \ref{lem:ref}.  The blue part is carrying less energy than the grid (in red).  In particular this allows to link each $y\in Q^{t}_{i,j}$ to $w\in Q^{t}_{i,j+1}$ on particular paths that follows $\mathbf{p}_{(i,j)}^h$. Again, for illustrative reason we are depicting the Voronoi cells of the point clouds. }
\end{figure}

\begin{lemma}\label{lem:ref}
Fix $\e>0$.  Let $t\in \R_+$, $ G_{\e,t} \in \mathcal{G}_{t}(\Upsilon;\eta_\e)$ and $u\in L^2(Q;\eta_\e)$.  For all $Q_{i,j}^t, Q_{i,j+1}^t$ (horizontal) neighboring squares there exists two vertical paths and a horizontal path
\[
\bar{\mathbf{v}}^{j} \in \mathrm{Ver}( j; G_{\e ,t}),\quad \bar{\mathbf{v}}^{j+1} \in \mathrm{Ver}( j+1; G_{\e,t}), \quad\bar{\mathbf{h}}^{i} \in \mathrm{Hor}( i; G_{\e ,t})
\]
such that, setting
\begin{align*}
\mathbf{p}^h_{i,j}&:= (\bar{\mathbf{v}}^{j}\cap Q_{i,j}^t) \cup(\bar{\mathbf{v}}^{j
+1}\cap Q_{i,j+1}^t)\cup (\bar{\mathbf{h}}^{i}\cap (Q_{i,j}^t\cup Q_{i,j+1}^t ) ),
\end{align*}
we have
	\begin{align*}
	\sum_{x \in \p^h_{(i,j)}} D_\e u( x) &\leq  \frac{\Upsilon \e}{t} \F_{\e }(u;(Q_{i,j}^t\cup Q_{i,j+1}^t ) ).
	\end{align*}
\end{lemma}

\begin{proof}
Since both horizontal and vertical paths are all disjoint we have that
	\begin{align*}
	\sum_{m=1}^{M_{\e,t}} \sum_{x \in \mathbf{h}^i_m \cap (Q_{i,j}^t\cup Q_{i,j+1}^t) } D_\e u (x)& \leq \F_{\e}(u;(Q_{i,j}^t\cup Q_{i,j+1}^t) )\\
		\sum_{m=1}^{M_{\e,t}} \sum_{x \in\mathbf{v}^{j}_m \cap Q_{i,j}^t} D_\e u (x)+\sum_{m=1}^{M_{\e,t}} \sum_{x \in \mathbf{v}^{j+1}_m \cap Q_{i,j+1}^t} D_\e u (x)& \leq \F_{\e}(u;Q_{i,j}^t)+ \F_{\e}(u;Q_{i,j+1}^t ).
	\end{align*}
In particular there are three paths $\bar{\mathbf{h}}^i\in \mathrm{Hor}(i;G_{\e,t})$, $\bar{\mathbf{v}}^i\in \mathrm{Ver}(j;G_{\e,t}) , \bar{\mathbf{v}}^{j+1}\in \mathrm{Ver}(j+1;G_{\e,t}) $
 such that
	\begin{align*}
	 \sum_{x \in \bar{\mathbf{h}}^i \cap (Q_{i,j}^t\cup Q_{i,j+1}^t) } D_\e u (x) &\leq \frac{1}{M_{\e,t}} \F_\e(u;Q_{i,j}^t\cup Q_{i,j+1}^t)\\
	 	 \sum_{x \in \bar{\mathbf{v}}^j \cap Q_{i,j}^t} D_\e u(x)+  \sum_{x \in \bar{\mathbf{v}}^{j+1} \cap Q_{i,j+1}^t} D_\e u(x)  &\leq \frac{1}{M_{\e, t}}[\F_{\e}(u;Q_{i,j}^t)+ \F_{\e}(u;Q_{i,j+1}^t)]
	\end{align*}
Note that, by the properties of the grid we have
$ \frac{1}{M_{\e,t} }\leq \Upsilon \frac{\e}{t}$.
Therefore, we conclude.
\end{proof}
We now provide a lemma that allows to estimate the difference of $T^{G_{\e,t}}(u_{\e})$, defined as in \eqref{eqn:operatore},  between neighboring squares of the squares partition $\mathcal{Q}_t$. Once again, we limit ourselves to prove the statement for horizontal neighboring squares since the vertical case follows in the same way up to changing the notation accordingly.
\begin{lemma}\label{lem:localEstimate}
There exists a constant $C>0$ independent of $\e$ and $t$ such that for any $ G_{\e,t}\in \mathcal{G}_{t}(\Upsilon;\eta_\e)$   any $u \in L^2(Q;\eta_\e)$ and any pair $Q_{i,j}^t,Q_{i,j+1}^t$ of  neighboring squares there holds
	\[
	\Bigl|(u)^{G_{\e,t}}_{i,j} - (u)^{G_{\e, t}}_{i,j+1}\Bigr|^2\leq C \F_\e(u;Q_{i,j}^t\cup Q_{i,j+1}^t).
	\] 
\end{lemma}

\begin{proof}
Let $\bar{\mathbf{h}}^i,\bar{\mathbf{v}}^j,\bar{\mathbf{v}}^{j+1} $ be the paths given by Lemma \ref{lem:ref}. Let $\mathbf{p}:=\mathbf{p}_{i,j}^h$ and let 
	\begin{align*}
	(u)_{i,j}^{\mathbf{p}}:=\frac{1}{\eta_n(\mathbf{p}\cap Q_{i,j}^t)} \sum_{x\in \mathbf{p}\cap Q_{i,j}^t} u(x), \qquad	(u)_{i,j+1}^{\mathbf{p}}:=\frac{1}{\eta_\e (\mathbf{p}\cap Q_{i,j+1}^t)} \sum_{x\in \mathbf{p}\cap Q_{i,j+1}^t} u(x).
	\end{align*}
We will make use of the above quantities (namely the average of the function $u$ on the skeleton $\mathbf{p}$) to estimate $(u)^{G_{\e,t}}_{i,j}$. \smallskip
	
	For any $x\in \mathbf{p}\cap Q_{i,j}^t$ we now build a family of paths $\mathcal{P}_{i,j}(x)$ which links $x$ to all points $y\in  G_{\e,t} \cap Q_{i,j}^t$.  The construction will be such that  we may control the number of times a path path passes through a point $y\in  G_{\e,t} \cap Q_{i,j}^t$.  We say that $y\in  G_{\e,t} \cap Q_{i,j}^t$ is a \textit{horizontal point} if it belongs to some horizontal path $\mathbf{h} \in \mathrm{Hor}(i;G_{\e,t})$ and does not belongs to any  $\mathbf{v}\in \mathrm{Ver}(j;G_{\e,t})$. We say instead that it is \textit{vertical} if the converse happens.  We say that it is \textit{nodal} if it belongs to the intersection of horizontal and vertical paths. We briefly describe the construction: if $y$ is horizontal, say it belongs to $\mathbf{h}^i_m$ the $m$-th horizontal path, then we consider the path starting from $x$, following $\mathbf{p}$ until we meet $ \mathbf{h}^i_m$ and then following $\mathbf{h}^i_m$ until we reach $y$.  If instead $y$ is vertical and belongs to $\mathbf{v}^j_l$ we start from $x$,  follow $\mathbf{p}$ until we meet $\mathbf{v}^j_l$ and then follow $\mathbf{v}^j_l$ to reach  $y$. If it is nodal we follow any of the two possibilities. This family of paths, call it $\mathcal{P}_{i,j}(x)$, have the following property, which is crucial in what follows. For any $x\in \mathbf{p}\cap Q_{i,j}^ t$, each point $y\in (G_{\e,t}\cap  Q_{i,j}^t)\setminus \mathbf{p}$ belongs to not more than $\frac{C t}{\e}$ paths $\mathbf{t}\in \mathcal{P}_{i,j}(x)$.  Indeed, if  $y\in \mathbf{h}_m^i$ then it belongs only to the paths with initial points in $\mathbf{h}_m^i$, which do not exceed $Ct/ \e$.  Moreover, the length of each path does not exceed $Ct/\e$ as well. 
	
	 We now divide the rest of the proof in two steps.  \smallskip

\textbf{Step one:} \textit{Comparison between $(u)_{i,j}^{\mathbf{p}}$ and $(u)_{i,j}^{G_{\e,t}}$}.  For the sake of shortness set
	\[
	N_j:= \eta_\e(\mathbf{p}\cap Q_{i,j}^t), \ \ N_j':= \eta_\e(G_{\e,t}\cap Q_{i,j}^t)
	\]
Then, by using Jensen's inequality twice and property (g) of Definition \ref{def:pathConn} we have
	\begin{align*}
	\left|(u)_{i,j}^{G_{\e,t}}-(u)_{i,j}^{\mathbf{p}}\right|^2\leq&\frac{1}{N_j N_j'} \sum_{x\in \mathbf{p}\cap Q_{i,j}^t} \sum_{y\in G_{\e,t}\cap Q_{i,j}^t} |u(x)-u(y)|^2\\
	\leq&\frac{C}{N_j N_j'}\sum_{x\in \mathbf{p}\cap Q_{i,j}^t} \sum_{\mathbf{t}\in \mathcal{P}_{i,j}(x)} \ell(\mathbf{t})\sum_{y\in \mathbf{t}} D_{\e}u(y)
\end{align*}
By invoking property (d) we can further deduce
\begin{align*}
	\left|(u)_{i,j}^{G_{\e,t}}-(u)_{i,j}^{\mathbf{p}}\right|^2\leq&\frac{C t}{\e N_j N_j'}\sum_{x\in \mathbf{p}\cap Q_{i,j}^t}  \sum_{y\in G_{\e,t}\cap Q_{i,j}^t} D_{\e}u(y) \sum_{\mathbf{t}\in \mathcal{P}_{i,j}(x)} \ca_{\mathbf{t}}(y)\\
	\leq&\frac{Ct}{\e N_j N_j'}\sum_{x\in \mathbf{p}\cap Q_{i,j}^t}  \sum_{y\in (G_{\e,t}\cap Q_{i,j}^t)\setminus \mathbf{p}} D_{\e}u_n(y) \sum_{\mathbf{t}\in \mathcal{P}_{i,j}(x)} \ca_{\mathbf{t}}(y)\\
	&+\frac{Ct}{\e N_j N_j'}\sum_{x\in \mathbf{p}\cap Q_{i,j}^t}  \sum_{y\in  \mathbf{p}} D_{n}u_n(y) \sum_{\mathbf{t}\in \mathcal{P}_{i,j}(x)} \ca_{\mathbf{t}}(y)\\
	\leq&\frac{C t^2}{(\e)^2 N_j N_j'}\sum_{x\in \mathbf{p}\cap Q_{i,j}^t}  \sum_{y\in (G_{\e,t}\cap Q_{i,j}^t)\setminus \mathbf{p}} D_{\e}u(y) +\frac{Ct^3}{(\e)^3 N_j N_j'}\sum_{x\in \mathbf{p}\cap Q_{i,j}^t}  \sum_{y\in  \mathbf{p}} D_{\e}u(y)\\
	\leq & \frac{Ct^2}{\e^2 N_j'} \F_{\e}(u;Q_{i,j}^t)+\frac{Ct^3}{\e^3 N_j'}  \frac{\e}{t} \F_{\e}(u;Q_{i,j}^t\cup Q_{i,j+1}^t),
	\end{align*}
where the last inequality follows from the particular choice of $\mathbf{p}$ given by Lemma \ref{lem:ref}. Since the properties of the grid (property (d) of Theorem \ref{thm:pathConn}) also imply that
	\[
	N_j' \geq \frac{Ct^2}{\e^2},
	\]
we conclude that
	\begin{equation}\label{eqn:p1}
	\left|(u)_{i,j}^{G_{\e, t}}-(u)_{i,j}^{\mathbf{p}}\right|^2\leq C\F_{\e}(u;Q_{i,j}^t\cup Q_{i,j+1}^t).
	\end{equation}
The same exact computation shows also that
	\begin{equation}\label{eqn:p12}
	\left|(u)_{i,j+1}^{G_{\e,t}}-(u)_{i,j+1}^{\mathbf{p}}\right|^2\leq C\F_{\e}(u;Q_{i,j}^t\cup Q_{i,j+1}^t).
	\end{equation}
	\smallskip
	
\textbf{Step two:} \textit{Comparison between $(u)_{i,j+1}^{\mathbf{p}}$ and $(u)_{i,j}^{\mathbf{p}}$}. For $x\in \mathbf{p}\cap Q_{i,j}^t$ let $\mathcal{P}_{j+1}'(x)$ be the family of paths $\mathbf{t}$ on $\mathbf{p}$ which link each point of $\mathbf{p}\cap Q_{i,j+1}^t$ to $x$.  By considering one path for each $y\in \mathbf{p}\cap Q_{i,j+1}^t$ we can build $\mathcal{P}_{j+1}'(x)$ in a way that each $\mathbf{t}\in \mathcal{P}_{j+1}'(x)$ contains not more than $Ct/\e)$ points.  Moreover, each point $z\in \mathbf{p}$ is contained in at most $Ct/\e$ paths. With the notation introduced above, we then compute
\begin{align}
	\left|(u)_{i,j}^{\mathbf{p}}-(u)_{i,j+1}^{\mathbf{p}}\right|^2\leq& \frac{1}{N_j  N_{j+1}}\sum_{x \in\mathbf{p}\cap Q_{i,j}^t} \sum_{y\in\mathbf{p}\cap Q_{i,j+1}^t}|u(x)-u(y)|^2 \nonumber\\
	\leq& \frac{C}{N_j N_{j+1}}\sum_{x \in\mathbf{p}\cap Q_{i,j}^t} \sum_{\mathbf{t}\in \mathcal{P}'_{j+1}(x)}\ell(\mathbf{t}) \sum_{y\in\mathbf{t}} D_\e u(y) \nonumber\\
		\leq& \frac{Ct}{\e N_j N_{j+1}}\sum_{x \in\mathbf{p}\cap Q_{i,j}^t} \sum_{y\in\mathbf{p}\cap (Q_{i,j}^t\cup Q_{i,j+1}^t)} D_\e u(y) \sum_{\mathbf{t}\in \mathcal{P}'_{j+1}(x)} \ca_{\mathbf{t}}(y) \nonumber\\
		\leq& \frac{Ct^2}{\e^2 N_j N_{j+1}}\sum_{x \in\mathbf{p}\cap Q_{i,j}^t} \sum_{y\in\mathbf{p}\cap (Q_{i,j}^t\cup Q_{i,j+1}^t)} D_\e u(y) \nonumber\\
			\leq& C \sum_{x \in\mathbf{p}\cap Q_{i,j}^t} \frac{\e}{t} \F_\e(u;Q_{i,j}^t\cup Q_{i,j+1}^t) \nonumber\\
			\leq& C\F_\e(u;Q_{i,j}^t\cup Q_{i,j+1}^t).\label{eqn:p2}
\end{align}

 \textbf{Conclusion}. By means of Step one and Step two, in particular by collecting \eqref{eqn:p1}, \eqref{eqn:p12} and \eqref{eqn:p2} and by means of a triangular inequality we conclude.
\end{proof}

\subsection{Properties of the convergence for sequences with equibounded energy}
We now state and prove some useful properties of the convergence in Definition \ref{def:conv}. We start with the following lemma, which ensures that the limit of sequences with equibounded energy,  when it exists,  is unique and does not depend on the choice of the sequence of regular grids $\{G_{\e,t}\in \mathcal{G}_{t}(\Upsilon;\eta_\e)\}_{\e, t>0}$ when $\e,t\rightarrow 0$.
\begin{lemma}\label{lem:comparisongrid}
If $\{u_\e \in L^2(Q;\eta_\e)\}_{\e>0}$ is a sequence of function satisfying \eqref{eqn:CMPbnds} and $\{G_{\e,t}\in \mathcal{G}_{t}(\Upsilon;\eta_\e)\}_{\e,t>0},\{\bar{G}_{\e,t}\in \mathcal{G}_{t}(\Upsilon';\eta_\e) \}_{\e,t>0}$ (with possibly $ \Upsilon\neq \Upsilon'$) are two sequences of regular grids such that
	\[
	T^{G_{\e,t}}(u_\e)\stackrel{\e\to 0}{\longrightarrow } u^t, \ \ \ T^{\bar{G}_{\e,t}}(u_\e)\stackrel{\e\to 0}{\longrightarrow } \bar{u}^t
	\]
then
	\begin{align*}
\lim_{t\rightarrow  0} \int_Q \left|u^t(x)-\bar{u}^t(x)\right|^2\d x=0.
	\end{align*}
\end{lemma}
\begin{proof}
Fix $Q_{i,j}^t\in \mathcal{Q}_t$ and let $\mathbf{p}$, $\bar{\mathbf{p}}$ be the union of paths given by Lemma \ref{lem:ref} relative to $Q_{i,j}^t$ (and one of its neighbors, say $Q_{i,j+1}^t$ without loss of generality) and to to the grid $G_{\e,t}, \bar{G}_{\e,t}$ respectively.  Now by construction we have that $\mathbf{p}\cap Q_{i,j}^t,\bar{\mathbf{p}}\cap Q_{i,j}^t$ share at least two points.  In particular, for any $x\in\mathbf{p}\cap Q_{i,j}^t$ we can still build a family $\mathcal{P}(x)$ of paths that link each $y\in \bar{\mathbf{p}}\cap Q_{i,j}^t$ to $x$ containing only points in $\mathbf{p}\cup \bar{\mathbf{p}}\cap Q_{i,j}^t$. We can also ensure that each $\mathbf{t}\in\mathcal{P}(x)$ contains not more than $Ct/\e$ points and that any point in $z\in \mathbf{p}\cup \bar{\mathbf{p}}\cap Q_{i,j}^t$ is contained in at most $Ct/\e$ paths in $\mathcal{P}(x)$. With the notation introduced in the proof of Lemma \ref{lem:localEstimate} we now compute
	\begin{align*}
\left| (u_\e)_{i,j}^{\mathbf{p}}-  (u_\e)_{i,j}^{\bar{\mathbf{p}}} \right|^2&\leq \frac{1}{N_j \bar{N}_j} \sum_{x\in \mathbf{p}\cap Q_{i,j}^t} \sum_{y\in \bar{\mathbf{p}}\cap Q_{i,j}^t} |u_\e(x)-u_\e(y)|^2\\
 &\leq \frac{1}{N_j \bar{N}_j} \sum_{x\in \mathbf{p}\cap Q_{i,j}^t} \sum_{\mathbf{t}\in \mathcal{P}(x)} \ell(\mathbf{t})\sum_{y\in \mathbf{t}} D_\e u(x)\\
  &\leq C \F_{\e}(u_\e;Q_{i,j}^t\cup Q_{i,j+1}^t),
\end{align*}
where we have used the properties of $\mathbf{p},\bar{\mathbf{p}}$ and of $\mathbf{t}\in \mathcal{P}(x)$. Now, by recalling that \eqref{eqn:p1},\eqref{eqn:p12} are in force respectively on $G_{\e,t},\mathbf{p}$ and $\bar{G_{\e,t}},\bar{\mathbf{p}}$ by means of the same arguments used to prove Step one of Lemma \ref{lem:localEstimate}, with a triangular inequality we obtain
	\begin{align*}
	\left|(u_\e)^{G_{\e,t}}_{i,j}- (u_\e)^{\bar{G}_{\e,t}}_{i,j}\right|^2\leq C \F_\e(u_\e;Q_{i,j}^t \cup Q_{i,j+1}^t).
	\end{align*}
If we sum up over all $i,j$ and we observe that the energy of each square is counted at most a finite number of time (independent of $t,n$) we reach
	\[
	\sum_{i,j=1}^{k_t} \left|(u_\e)^{G_{\e,t}}_{i,j}- (u_\e)^{\bar{G}_{\e,t}}_{i,j}\right|^2\leq C \F_\e(u_\e;Q_{tk_t})\leq C \F_\e(u_\e;A).
	\]
By means of this last relation we have
	\[
	\int_{Q}\left|T^{G_{\e,t}}(u_\e)(x)-T^{\bar{G}_{\e,t}}(u_\e)(x)\right|^2\d x=t^2\sum_{i,j=1}^{k_t} \left|(u_\e)^{G_{\e,t}}_{i,j}- (u_\e)^{\bar{G}_{\e,t}}_{i,j}\right|^2\leq C t^2
	\]
In particular, if $T^{G_{\e,t}}(u_\e)\stackrel{ _\e}{\longrightarrow} u^t$, $T^{\bar{G}_{\e,t}}(u_\e)\stackrel{ _\e}{\longrightarrow} \bar{u}^t$ then
\[
	\int_{Q}\left|u^t(x)-\bar{u}^t(x)\right|^2\d x\leq C t^2
	\]
and we conclude.
\end{proof}

Now we proceed to state and prove Proposition \ref{prop:L2ongrid} and Proposition \ref{prop:equivalence}, which will give us a useful characterization of the convergence in Definition \ref{def:conv}. 
We will make use of the notion of piecewise-constant extension introduced in Section \ref{sbsct:pwCextension}.

\begin{proposition}[$L^2$ convergence on the grids] \label{prop:L2ongrid}
Let $\{u_{\e}\in L^2(Q;\eta_{\e})\}_{\e>0}$ be a sequence  satisfying \eqref{eqn:CMPbnds} and assume that $u_\e\rightarrow u$ in the sense of Definition \ref{def:conv}. Then
\begin{equation}\label{eqn:keybounds2} 
		\lim_{t\rightarrow 0} \limsup_{\e\rightarrow 0} \int_{\mathcal{V}_{\e}(G_{\e,t}) \cap Q} |\hat{u}_\e(x)-u(x)|^2\d x=0
		\end{equation}
for any sequence of regular grids $ \{G_{\e,t}\in \mathcal{G}_{t}(\Upsilon;\eta_\e)\}_{\e,t>0}$. 
\end{proposition}

\begin{proof}
By means of similar computations as the ones used in the proof of Lemmas \ref{lem:localEstimate} and \ref{lem:comparisongrid}, we can infer also that for any $x\in G_{\e,t}\cap Q_{i,j}^t$ (adopting the same notation)
	\begin{align}
	\left|u_\e(x) - (u_{\e})_{i,j}^{\mathbf{p}}\right|^2\leq &\frac{C}{N_j}\sum_{y\in \mathbf{p}\cap Q_{i,j}^t} |u_\e(x)-u_\e(y)|^2\nonumber
	\leq \frac{C}{N_j}\sum_{\mathbf{t}\in \mathcal{P}_{i,j}(x)} \sum_{y\in \mathbf{t}} D_\e u(y)\nonumber\\
		\leq &\frac{C}{N_j}\sum_{y\in G_{\e,t}\cap Q_{i,j}^t} D_\e u_\e(y)\sum_{\mathbf{t}\in \mathcal{P}_{i,j}(x)}  \ca_{\mathbf{t}}(y)\nonumber\\
		\leq &\frac{C t }{N_j \e }\sum_{y\in G_{\e,t}\cap Q_{i,j}^t\setminus \mathbf{p}} D_\e u_\e(y)+ \frac{C t^2 }{N_j\e^2 }\sum_{y\in \mathbf{p}\cap Q_{i,j}^t} D_\e u_\e(y)\nonumber\\
		\leq &\frac{C t }{N_j \e }\sum_{y\in G_{\e,t}\cap Q_{i,j}^t\setminus \mathbf{p}} D_\e u_\e(y)+ \frac{C   t }{N_j \e }\F_\e (u_\e;Q_{i,j}^t\cup Q_{i,j+1}^t)\nonumber\\
		\leq & C \F_\e(u_\e;Q_{i,j}^t\cup Q_{i,j+1}^t) \label{eq:p3}
	\end{align}
for a constant independent of $t,\e$. In particular, by collecting \eqref{eq:p3} and \eqref{eqn:p1} we have
	\[
	\left|u_\e(x)-(u_\e)_{i,j}^{G_{\e,t}} \right|^2\leq C\F(u_\e;Q_{i,j}^t\cup Q_{i,j+1}^t),
	\]
which, summed up over $x\in G_{\e,t}$ and $i,j\in \{1,\ldots,k_t\}$, and taking into accounting property (d) of Definition \ref{def:pathConn} (implying that $\#(G_{\e,t}\cap Q_{i,j}^t)\leq C \sfrac{t^2}{\e^2}$ ),  yields
	\begin{align}\label{eqn:avAndPw}
	\e^2 \sum_{x\in G_{\e,t}\cap Q}\left|u_\e(x)-T^{G_{\e,t}}(u_\e)(x)\right|^2\leq C t^2 \F_{\e}(u_\e;A).
	\end{align}
In particular, if $u_\e \rightarrow u$, and  $\{u^t \in X_t \}_{t\in \R_+}$ denotes the sequence of intermediate functions with respect to $G_{\e,t}$, then
\begin{align*}
	t^2\sum_{J\in \mathcal{I}_t(Q)} |(u_\e)_{J}^{G_{\e,t}}- u_J^{t}|^2 \stackrel{ _\e}{\longrightarrow} 0
\end{align*}	
and thus
	\begin{align*}
		\e^2\sum_{x \in G_{\e,t}\cap Q} |u_\e(x)- u^{t}(x)|^2 \leq & \e^2 \sum_{x \in G_{\e,t}\cap Q} \left|u_\e(x)- T^{G_{\e,t}}(u_\e)(x)\right|^2+C t^2 \sum_{J\in \mathcal{I}_t(Q)}  \left|  (u_\e)_{i,j}^{G_{\e,t}}-u_J^{t}\right|^2\\
		\leq & C t^2 \F_\e(u_\e;A)+ C t^2\sum_{J\in \mathcal{I}_t(Q)} \left| (u_\e)_{J}^{G_{\e,t}}-u_J^{t}\right|^2.
	\end{align*}
This implies that
	\begin{align*}
	\int_{\mathcal{V}_{\e}(G_{\e,t})\cap Q} |\hat{u}_\e(x)-u(x)|^2\d x &\leq\int_{\mathcal{V}_{\e}(G_{\e,t})\cap Q} |\hat{u}_\e(x)-u^t(x)|^2\d x +\int_{G_{\e,t}} |u^t(x)-u(x)|^2\d x\\
	&\leq C t^2 \left(1+ \sum_{J\in \mathcal{I}_t(Q)} \left| (u_\e)_{J}^{G_{\e,t}}-u_J^{t}\right|^2 \right)  + \int_Q |u^t(x)-u(x)|^2\d x
	\end{align*}
	and then also \eqref{eqn:keybounds2} holds for any sequence of regular grids $ \{G_{\e,t}\in \mathcal{G}_{t}(\Upsilon;\eta_\e)\}_{\e,t>0}$.
\end{proof}

\begin{proposition}\label{prop:equivalence}
Let $\e_n, t_n$ be two sequences. Set $\eta_n:=\eta_{\e_n}$. Let $\{u_n\in L^2(Q;\eta_n))\}_{n\in \N}$ be a sequence satisfying \eqref{eqn:unifBnd}.  If there exists a sequence of regular grids $\{G_n\in \mathcal{G}_{t_n}(\Upsilon;\eta_n)\}_{n\in \N}$ such that
	\begin{equation}\label{eqn:convOnL2one}
	\lim_{n\rightarrow +\infty} \int_{\mathcal{V}_{\e_n}(G_n)\cap Q} |\hat{u}_n(x)-u(x)|^2\d x=0.	
	\end{equation}
Then $u_n\rightarrow u$ in the sense of Definition \ref{def:conv}.
\end{proposition}
\begin{proof}
We first show that
	\begin{equation}\label{consequence}
	\lim_{n\rightarrow +\infty} \int_Q |T^{G_n}(u_n(x))-u(x)|^2\d x=0. 
	\end{equation}
Indeed, by means of the same computation as in the proof of Proposition \ref{prop:L2ongrid} we can achieve \eqref{eqn:avAndPw}; that is,
	\[
	\e_n^2 \sum_{x\in G_n\cap Q} |u_n(x)-T^{G_n}(u_n)(x)|^2\leq C t_n^2 \F_{\e_n}(u_n;Q).
	\]
This implies that
	\[
	\lim_{n\rightarrow +\infty} \int_{\mathcal{V}_{\e_n} (G_n)\cap Q}  |\hat{u}_n(x)-T^{G_n}(u_n)(x)|^2\d x=0,
	\]
which, together with hypothesis \eqref{eqn:convOnL2one} implies \eqref{consequence}. 
\smallskip

We now fix $t\in \R$ and we provide a suitable rearrangement of $G_n$ in order to obtain a grid $\tilde{G}_{\e_n,t}\in  \mathcal{G}_{t}(\Upsilon;\eta_n)$, provided $n$ is large enough.  
We describe the construction.
Inside each $R_i^h(t)$ we can find at least $c_n=\lfloor\sfrac{(t-2t_n)}{t_n}\rfloor$ horizontal rectangles $R_{i'}^{h}(t_n),\ldots R_{i'+c_n} ^{h}(t_n)$ strictly contained in $R_i^h(t)$ and each containing $M_n$ disjoint horizontal paths $\mathbf{h}_1^i,\ldots , \mathbf{h}_{M_n}^i$.  Note that 
\begin{equation}\label{jojo}
 \frac{t}{2}  \leq   t_n c_n   \leq 2 t
\end{equation}
In particular, by relabeling the paths of $G_{n}$ as $\{ \mathbf{h}_m^i, \ m=1,\ldots, M_n c_n \} \subset R_i^h(t)$ for each horizontal rectangle accordingly. We repeat the same argument for vertical paths and we obtain a grid $\tilde{G}_{n,t}$. Thanks to \eqref{jojo}  we have that $\tilde{G}_{n,t}\in \mathcal{G}_{t}(2\Upsilon, \eta_n)$.  

\smallskip
We now show that $T^{\tilde{G}_{n,t}}(u_n)\rightarrow u^t$ and $u^t\rightarrow u$ in $L^2(Q)$.  Let now $J\in \mathcal{I}_t(Q)$ and consider
	\[
	\mathcal{I}_n(Q_t(J)):=\{ J'\in \mathcal{I}_{t_n}(Q_t(J)) : \ Q_{t_n}(J')\subset Q_t(J)\}, \ \qquad r_n(J):=\#(\mathcal{I}_n(Q_t(J))).
	\]
Recalling the notation
	\begin{align*}
	(u_n)_{J'}^{G_n} = \frac{1}{\eta_n(G_n\cap Q_t(J') )  } \sum_{x\in G_n\cap Q_t(J') } u_n(x), \quad
		(u_n)_{J}^{\tilde{G}_{n,t} } = \frac{1}{\eta_n(\tilde{G}_{n,t} \cap Q_t(J) ) } \sum_{x\in \tilde{G}_{n,t} \cap Q_t(J) } u_n(x),
	\end{align*}
we have
		\begin{align*}
		(u_n)_{J}^{\tilde{G}_{n,t} } &= \frac{1}{\eta_n(\tilde{G}_{n,t} \cap Q_{t_n}(J) ) } \sum_{x\in \tilde{G}_{n,t} \cap Q_{t_n}(J) } u_n(x) \\
		&= \frac{1}{\eta_n(\tilde{G}_{n,t} \cap Q_t(J) ) } \sum_{J'\in 	\mathcal{I}_n(Q_t(J))} \sum_{ x\in \tilde{G}_{n,t} \cap Q_{t_n}(J') } u_n(x) + R_n(J)\\
		&= \sum_{J'\in 	\mathcal{I}_n(Q_t(J))} \frac{\eta_n(\tilde{G}_{n,t} \cap Q_{t_n}(J') ) }{\eta_n(\tilde{G}_{n,t} \cap Q_t(J) ) } (u_n)_{J'}^{G_n} + R_n(J)\\
		&= \sum_{J'\in 	\mathcal{I}_n(Q_t(J))} \kappa_n(J') (u_n)_{J'}^{G_n} + R_n(J),
	\end{align*}
	where
	\begin{align*}
	\kappa_n(J') &= \frac{\eta_n(\tilde{G}_{n,t} \cap Q_{t_n}(J') ) }{\eta_n(\tilde{G}_{n,t} \cap Q_t(J) ) }\\
	R_n(J)&= \frac{1}{\eta_n(\tilde{G}_{n,t} \cap Q_t(J) ) } \sum_{J'\notin \mathcal{I}_n(Q_t(J))} \sum_{x\in G_n\cap Q_{t_n}(J')\cap Q_t(J)} u_n(x).
	\end{align*}
Observing that
	\begin{align*}
	(u)_{Q_t(J)}=\frac{t_n^2}{t^2}\sum_{J'\in \mathcal{I}_n(Q_t(J))} (u)_{Q_{t_n}(J')}+\frac{1}{t^2}\sum_{J' \in \mathcal{I}_{t_n}(\partial Q_t(J)))}\int_{Q_{t_n}(J')\cap Q_t(J)} u(y) \d y
	\end{align*}
we have
\begin{align}
\int_{Q}& \Bigl|T^{\tilde{G}_{n,t}}(u_n)(x) - \sum_{ J\in \mathcal{I}_t(Q)} (u)_{Q_t(J)} \ca_{Q_t(J)}(x) \Bigr|^2\d x\nonumber\\
=& \sum_{ J\in \mathcal{I}_t(Q)} t^2 |(u_n)_{J}^{\tilde{G}_{n,t} }  - (u)_{Q_t(J)}|^2 \nonumber\\
=&\ C\sum_{ J\in \mathcal{I}_t(Q)} t^2 \Bigl| \sum_{J'\in 	\mathcal{I}_n(Q_t(J))} \kappa_n(J') (u_n)_{J'}^{G_n} - (u)_{Q_t(J)}\Bigr|^2+\sum_{ J\in \mathcal{I}_t(Q)} t^2 |R_n(J)|^2\nonumber\\
\leq &\ C\sum_{ J\in \mathcal{I}_t(Q)} t^2 \Bigl| \sum_{J'\in 	\mathcal{I}_n(Q_t(J))} \kappa_n(J') (u_n)_{J'}^{G_n} - \frac{t_n^2}{t^2} (u)_{Q_{t_n}(J')}\Bigr|^2 +C\sum_{ J\in \mathcal{I}_t(Q)} t^2 |R_n'(J)|^2,
\label{w1}
\end{align}
where we have set
	\[
	R_n'(J):=R_n(J)+\frac{1}{t^2}\sum_{J' \in \mathcal{I}_{t_n}(\partial Q_t(J)))}\int_{Q_{t_n}(J')\cap Q_t(J)} u(y) \d y.
	\]
We now concentrate on proving separately the estimates required.
\begin{align}
& \hskip-3cm\Bigl| \sum_{J'\in 	\mathcal{I}_n(Q_t(J))} \kappa_n(J') (u_n)_{J'}^{G_n} - \frac{t_n^2}{t^2} (u)_{Q_{t_n}(J')}\Bigr|^2 \nonumber  \\
  \leq  &  \Bigl| \sum_{J'\in 	\mathcal{I}_n(Q_t(J))} \kappa_n(J') \Bigl[(u_n)_{J'}^{G_n} -  (u)_{Q_{t_n}(J')}\Bigr]\Bigr|^2\nonumber \\
& + \sum_{J'\in 	\mathcal{I}_n(Q_t(J))}  \Bigl|  \kappa_n(J')  - \frac{t_n^2}{t^2} \Bigr|^2  |(u)_{Q_{t_n}(J')}|^2\nonumber\\
  \leq  & \sum_{J'\in 	\mathcal{I}_n(Q_t(J))} \kappa_n(J')^2 \frac{r_n(J)}{t_n^2}   \int_{Q_{t_n}(J')} |(u_n)_{J'}^{G_n} -  u(x)|^2 \d x  \nonumber \\
& + \sum_{J'\in 	\mathcal{I}_n(Q_t(J))}  \Bigl|  \kappa_n(J')  - \frac{t_n^2}{t^2} \Bigr|^2  |(u)_{Q_{t_n}(J')}|^2\label{w11}
\end{align}
Observe that
$
	r_n(J)\leq \frac{t^2}{t_n^2},  \ \ \kappa_n(J')\leq C \frac{t_n^2}{t^2}
$,
and then
\begin{align}
 &\hskip-1cm\sum_{J'\in 	\mathcal{I}_n(Q_t(J))}  \kappa_n(J')^2 \frac{r_n(J)}{t_n^2}   \int_{Q_{t_n}(J')} |(u_n)_{J'}^{G_n} -  u(x)|^2 \d x \nonumber\\
 \leq& \frac{C}{t^2}  \sum_{J'\in 	\mathcal{I}_n(Q_t(J))}     \int_{Q_{t_n}(J')} |(u_n)_{J'}^{G_n} -  u(x)|^2 \d x 
 \leq \frac{C}{t^2} \int_{Q_t(J)} |T^{G_n}(u_n)(x)-u(x)|^2\d x. \label{w111}
\end{align}
We also have
\begin{align}
\sum_{J'\in 	\mathcal{I}_n(Q_t(J))} & \Bigl|  \kappa_n(J')  - \frac{t_n^2}{t^2} \Bigr|^2  |(u)_{Q_{t_n}(J')}|^2\leq C \frac{t_n^2}{t^4} \sum_{J'\in 	\mathcal{I}_n(Q_t(J))} \int_{Q_{t_n}(J') } |u|^2\d x\nonumber\\
&\leq C \frac{t_n^2}{t^4} \int_{Q_{t}(J) } |u|^2\d x\label{w112}.
\end{align}
In particular, \eqref{w111} and \eqref{w112} yield
 \begin{equation}\label{w1f}
 \begin{split}
 \sum_{ J\in \mathcal{I}_t(Q)} t^2 \Bigl| \sum_{J'\in 	\mathcal{I}_n(Q_t(J))} \kappa_n(J') (u_n)_{J'}^{G_n} - \frac{t_n^2}{t^2} (u)_{Q_{t_n}(J')}\Bigr|^2\leq & C \int_{Q} |T^{G_n}(u_n)(x) - u(x)|^2\d x\\
 &+ \frac{Ct_n^2 }{t^2} \int_Q |u|^2\d x.
 \end{split}
 \end{equation}
Finally, we estimate  
	\begin{align}
 t^2|R_n(J)|^2&\leq t^2\Bigl(\sum_{J'\in \mathcal{I}_t(\partial Q_t(J) )}  \kappa_n(J') (u_n)^{G_n}_{J'}\Bigr)^2\nonumber\\
	&\leq t^2 \#( \mathcal{I}_{t_n}(\partial Q_t(J) )) \sum_{J'\in \mathcal{I}_{t_n}(\partial Q_t(J) )}  \kappa_n(J')^2 | (u_n)^{G_n}_{J'} |^2 \nonumber\\
	&\leq \frac{Ct_n}{t} \sum_{J'\in \mathcal{I}_{t_n}(\partial Q_t(J) )} t_n^2 | (u_n)^{G_n}_{J'} |^2\nonumber\\
	&\leq \frac{Ct_n}{t} \sum_{J'\in \mathcal{I}_{t_n}(\partial Q_t(J) )} \int_{Q_{t_n}(J')} | T^{G_n}(u_n)(x) |^2\d x.\label{w2}
	\end{align}
	Also, we have
		\begin{align}
\frac{1}{t^2}&\Bigl(\sum_{J' \in \mathcal{I}_{t_n}(\partial Q_t(J)))}\int_{Q_{t_n}(J')\cap Q_t(J)} u(y) \d y\Bigr)^2\nonumber\\
&\leq C\frac{\#( \mathcal{I}_{t_n}(\partial Q_t(J) ))}{t^2}\sum_{J' \in \mathcal{I}_{t_n}(\partial Q_t(J)))}\Bigl(\int_{Q_{t_n}(J')\cap Q_t(J)} u(y) \d y\Bigr)^2\nonumber\\
&\leq C\frac{t_n}{t}\sum_{J' \in \mathcal{I}_{t_n}(\partial Q_t(J)))}\int_{Q_{t_n}(J')\cap Q_t(J)} u(y)^2 \d y  \label{w3}
		\end{align}
By collecting \eqref{w1f},  \eqref{w2} and \eqref{w3} we then have 
	\begin{align*}
	\int_{Q}& \Bigl|T^{\tilde{G}_{n,t}}(u_n)(x) - \sum_{ J\in \mathcal{I}_t(Q)} (u)_{Q_t(J)} \ca_{Q_t(J)}(x) \Bigr|^2\d x\\
	\leq &\ C \int_{Q} |T^{G_n}(u_n)(x) - u(x)|^2\d x + \frac{Ct_n^2 }{t^2} \int_Q |u|^2\d x\\
&+\frac{C t_n}{t} \sum_{ J\in \mathcal{I}_t(Q)}  \sum_{J'\in \mathcal{I}_{t_n}(\partial Q_t(J) )} \int_{Q_{t_n}(J')} ( | T^{G_n}(u_n)(x) |^2+|u(x)|^2)\d x\\
\leq &\ C \int_{Q} |T^{G_n}(u_n)(x) - u(x)|^2\d x + \frac{Ct_n^2 }{t^2} \int_Q |u|^2\d x\\
& +\frac{C t_n}{t}  \int_{Q} (| T^{G_n}(u_n)(x) |^2+|u(x)|^2)\d x.
	\end{align*}
Hence,  by taking the limit as $n\rightarrow +\infty$ we obtain $T^{\tilde{G}_{n,t}}(u_n) \rightarrow u^t$ in $L^2(Q)$ where
	\[
	u^t:=\sum_{ J\in \mathcal{I}_t(Q)} (u)_{Q_t(J)} \ca_{Q_t(J)}(x).
	\]
Note that $u^t\rightarrow u$ in $L^2(Q)$.  In particular we exhibit a sequence of regular grids along which convergence in the sense of Definition \ref{def:conv} holds.  Lemma \ref{lem:comparisongrid} ensures then that it occurs on any sequence of regular grids, achieving thus the proof. 
\end{proof} 

\begin{remark}\rm
Note that Propositions \ref{prop:L2ongrid} and \ref{prop:equivalence} imply also that the stronger convergence on the piecewise-constant functions implies the convergence in the sense of Definition \ref{def:conv},  provided the sequence has equibounded energy.  Indeed if $u_\e$ is a sequence satisfying \eqref{eqn:unifBnd} and
	\[
	\lim_{\e\rightarrow 0} \int_{\mathcal{V}_{\e}(\eta_{\e})\cap Q}  |\hat{u}_\e(x)-u(x)|^2\d x=0
	\]
then, in particular it converges when restricted to any sequence of regular grids, allowing us to invoke Proposition \ref{prop:equivalence} and thus to conclude that $u_{\e}\rightarrow u$ in the sense of Definition \ref{def:conv}.
 This has some useful consequence, as in the proof of the locality of the convergence (Lemma \ref{lem:tecConv}) and the fact that we can diagonalize in the $L^2$ convergence (Lemma \ref{lem:diagonal}). \end{remark}

\begin{lemma}\label{lem:tecConv}
Let $u_{\e}\rightarrow u$ in the sense of Definition \ref{def:conv} and satisfying \eqref{eqn:CMPbnds}.  Let $A\subseteq Q$ and suppose that $u_{\e}=w$  on $A$ for some $w\in C^1(A)$ and for all $\e>0$. Then $u=w$ on $A$. 
\end{lemma}

\begin{proof}
Let $\{ G_{\e,t }\in \mathcal{G}_{t}(\Upsilon,\eta_{\e})\}_{\e,t>0}$ be a sequence of regular grids.  Then, Proposition \ref{prop:L2ongrid} gives
	\begin{align*}
	\lim_{t \rightarrow 0} \limsup_{\e \rightarrow 0 }\int_{\mathcal{V}_{\e} (G_{\e ,t })\cap A } |\hat{u}_{\e}(x)-u(x)|^2\d x=0.
\end{align*}	  
Moreover, since $u_{\e}(x)=w(x)$ on $\eta_{\e}\cap A$ and since $|x-y|\leq \lambda\e$ for $x\in G_{\e,t}$, $y\in C(x;\eta_\e)$ (property (g)  of Definition \ref{def:pathConn}) and $|C(x;\eta_\e)|\leq C\e^2$, we have 
\begin{align*}
\int_{\mathcal{V}_{\e} (G_{\e,t})\cap A} |\hat{u}_{\e}(y)-w(y)|^2\d y=&\sum_{x\in G_{\e,t} \cap A} \int_{C(x;\eta_{\e})\cap A} |w(x)-w(y)|^2\d y\\
\leq &\ C \|\nabla w\|_{\infty}^2 \sum_{x\in \eta_{\e} \cap A}  \e^4\leq C\e^2,
\end{align*}
where $C$ depends on $\alpha,\lambda$ and $|A|$.  Here we have used Proposition \ref{propo:count}.  Then,  for some subsequence $\{\e_n,t_n\}_{n\in \N}$ and by means of a triangular inequality we have
\[
\lim_{n\rightarrow +\infty} \int_{\mathcal{V}_{\e_n} (G_{\e_n,t_n})\cap A} |u(x)-w(x)|^2\d x=0.
\]
Observe now that $\ca_{\mathcal{V}_{\e_n} (G_{\e_n,t_n})}\rightharpoonup f$ weakly $L^2(A)$ and $f(x)\geq s_0>0$  due to the good properties of the grids.  In particular 
	\[
	s_0\int_{A}|u(x)-w(x)|\d x=\lim_{n\rightarrow +\infty} \int_{\mathcal{V}_{\e_n} (G_{\e_n,t_n})\cap A} |u(x)-w(x)|\d x=0.
	\]
Being $s_0>0$ we conclude $u=w$ on $A$.
\end{proof}
\begin{lemma}\label{lem:diagonal}
Let $\{u_{n,r}\in L^2(Q;\eta_{\e_n}), \ r>\e_n>0\}$ be a sequence such that
\begin{itemize}
\item[a)] For any $r>0$,  $u_{n,r}\rightarrow u_r$ as $n$ goes to $+\infty$ in the sense of Definition \ref{def:conv};
\item[b)] $\{u_r\}_{r>0} \subset  W^{1,2}(Q)$,  $u_r\rightarrow u$ as $r$ goes to $0$ in $L^2(Q)$ for some $u\in W^{1,2}(Q)$;
\item[c)] $\sup_{r>\e_n>0}\{\F_{\e_n}(u_{n,r};A)\}<+\infty$ on some open set $A\supset Q$.
\end{itemize}
Then there is a sequence $r_{n}>\!>\e_n$ such that $r_{n} \rightarrow 0$ and for which $u_{n,r_{n}}\rightarrow u$ in the sense of Definition \ref{def:conv}.
\end{lemma}
\begin{proof}
By invoking Proposition \ref{prop:L2ongrid} we have (up to a subsequence)
	\[
	\limsup_{n \rightarrow +\infty} \int_{\mathcal{V}_{\e_n}(G_n)\cap Q} |\hat{u}_{n,r}(x)-u_r(x)|^2\d x=0.
	\]
Since also $u_r\rightarrow u$ we have
	\[
	\lim_{r\rightarrow 0} \limsup_{n \rightarrow +\infty} \int_{\mathcal{V}_{\e_n}(G_n)\cap Q} |\hat{u}_{n,r}(x)-u(x)|^2\d x=0.
	\]
Hence,  we can select the sought sequence $r_n$ to satisfy
	\[
	  \limsup_{n \rightarrow +\infty} \int_{\mathcal{V}_{\e_n}(G_n)\cap Q} |\hat{u}_{n,r_n}(x)-u(x)|^2\d x=0.
	\]
Property (c) and Proposition \ref{prop:equivalence} imply now that $u_{n,r_n}\rightarrow u$ in the sense of Definition \ref{def:conv}.
\end{proof}
\subsection{Proof of the compactness theorem}
We are now ready to prove Theorem \ref{thm:Compactness}.  We rely on the following Lemma \ref{lem:alici}, which comes as a consequence of \cite[Theorem 3.1]{alicandro2004general}. Recall that
	\[
	\mathcal{I}_t(Q):=\{J=(i,j)\in t \Z^2\cap Q\}.
	\]
If $|J-J'|=t$ we write $\langle J,J'\rangle$,  meaning that the square $Q_J^t=Q_{i,j}^t$ and $Q_{J'}^t=Q_{i',j'}^t$ are neighboring squares.
\begin{lemma}\label{lem:alici}
Let $\{u^t \in X_t\}_{t\in \R_+}$ be a sequence of function such that
	\[
	\sup_{t\in \R_+}\Bigl\{\sum_{J\in \mathcal{I}_t(Q)}  \sum_{\substack{J'\in  \mathcal{I}_t(Q):\\ \langle J,J' \rangle}} |u_J^t-u^t_{J'}|^2+t^2\sum_{J\in \mathcal{I}_t(Q)} |u^t_J|^2\Bigr\} <+\infty.
	\]
Then there exists a function $u\in W^{1,2}(Q)$ and a subsequence $\{t_l\}_{l\in \N}$ such that the piecewise-constant interpolation of $u_{t_l}$ converge to $u$ in $L^2(Q)$.
\end{lemma}

\begin{proof}[Proof of Theorem \ref{thm:Compactness}]
Let $\alpha,  \lambda$ be fixed and ensuring the validity of Theorem \ref{thm:pathConn} and choose $\{G_{\e,t}\in\mathcal{G}_{t}(\Upsilon;\eta_{\e})\}_{\e,t}$ to be a sequence of regular grids. Then, by invoking Lemma \ref{lem:localEstimate} and by summing up on $i,j\in \{1,\ldots,k_t\}$, we infer that
	\begin{equation}\label{in:enbnd}
	\sum_{J\in \mathcal{I}_t(Q)} \sum_{\substack{J'\in  \mathcal{I}_t(Q):\\ \langle J,J' \rangle}} \left|(u_\e)_J^{G_{\e,t}}-(u_\e)_{J'}^{G_{n,t}}\right|^2\leq C\F_\e(u_\e;A)<+\infty,
	\end{equation}
where we have adopted the shorthand $(u_\e)_J^{G_{\e,t}}=(u_\e)_{i,j}^{G_{\e,t}}$ for $J=(i,j)$. Moreover, by Jensen's inequality and the properties of the grid we have
	\begin{align}
	t^2\sum_{J\in \mathcal{I}_t(Q)}   \left|(u_\e)_J^{G_{\e,t}}\right|^2&\leq  C \e^2\sum_{J\in \mathcal{I}_t(Q)} \sum_{y\in G_{\e,t}\cap Q_J^t} |u_\e(y)|^2\nonumber\\
		&\leq C\e^2 \sum_{J\in \mathcal{I}_t(Q)} \sum_{y\in \cap Q_J^t} |u_\e(y)|^2\leq C \e^2\sum_{y\in \cap Q} |u_\e(y)|^2 <+\infty \label{in:normbnd}.
	\end{align}
Fix now a sequence $\{t_l\}_{\in \N} \in \R_+$ going to zero and observe that, for any $J\in \mathcal{I}_{t_l}(Q)$ we have
	\begin{align*}
	\left|(u_\e)_J^{G_{\e,t_l}} \right|^2 \leq \frac{1}{\eta_\e(G_{\e,t_l}\cap Q_J^{t_l})}\sum_{x\in G_{\e,t_l}\cap Q_J^{t_l} }  |u_\e(x)|^2 \leq \frac{\e^2}{t_l^2} \sum_{x\in G_{\e,t_l}\cap Q_J^{t_l}}  |u_\e(x)|^2<+\infty. 
	\end{align*}
Therefore, for such a fixed $t_l\in \R_+$, we can find  $\{\e_m\}_{m\in \N}$ such that
	$	(u_{\e_m})_J^{G_{\e_m,t_l}}\rightarrow u^{t_l}_J
	$.
Up to a diagonal extraction argument we may find $\{\e_m\}_{m\in \N}$ such that
\[	(u_{\e_m})_J^{G_{\e_m,t_l}}\rightarrow u^{t_l}_J \ \ \text{for any $J\in \{1,\ldots, k_{t_l} \}^2$, $l \in  \N$}.
	\]
This means that we can find a sequence of functions $\{u^{t_l}\in X_{t_l}\}_{l\in \N}$ and a subsequence $\{\e_m\}_{m\in \N}$ such that
	\[
	T^{G_{\e_m,t_l}}(u_{\e_m})\stackrel{ _m} {\longrightarrow} u^{t_l} \ \ \ \text{in $L^2(Q)$}.
	\] 
If we now invoke Lemma \ref{lem:alici}, combined with estimates \eqref{in:enbnd} and \eqref{in:normbnd}, we can find a subsequence of $\{t_l\}_{l\in \N}$ and a function $u\in W^{1,2}(Q)$ such that, with a slight abuse of notation, $u^{t_{l}}\rightarrow u$ in $L^2(Q)$ along the subsequence. But then, by Proposition \ref{prop:L2ongrid} we have
	\[
	\lim_{l\rightarrow +\infty} \limsup_{m\rightarrow +\infty} \int_{\mathcal{V}_{\e_m}(G_{\e_m,t_l})\cap Q}|\hat{u}_{\e_m}(x)-u(x)|^2\d x=0.
	\]
and by invoking Proposition \ref{prop:equivalence} this means that there is a subsequence of $\{ u_{\e_m}\}_{m\in \N}$ converging to $u$ in the sense of Definition \ref{def:conv}.
\end{proof}
%

\section{Proof of the $\Gamma$-convergence Theorem \ref{thm:Gamma}}\label{sct:gamma}
In this section we prove Theorem \ref{thm:Gamma}. We state and prove some preliminary results, subordinated to the identification of the constant $\Xi$ and to the development of the technical machinery required to present the proof.

\subsection{The cell problem}\label{sbsct:cellP}
We recall the notation for the boundary-value problem 
	\begin{equation*}
	\mathrm{m}(\xi ;A):=\inf\biggl\{ \sum_{x\in \eta \cap A} \sum_{y\in \eta\cap  B_\lambda(x)} |v(x)-v(y)|^2 \ \biggl| \begin{array}{c}
	v:\eta\rightarrow \R\\
 \ v(x)=\xi \cdot x  \ \ \text{for all $x\in \eta$ such that}  \\ 
 \mathrm{dist}(x,\partial A)\leq 2\lambda
\end{array}	 \biggr\}
	\end{equation*}
The first thing we need is the following lemma on the asymptotic behavior of the ``cell problem'' $\mathrm{m}(\xi;Q_T)$ when the boundary values are fixed on a cube $Q_T$ and $T$ diverges. 

\begin{lemma}\label{lem:lim}
For any $\xi\in \R^2$ it holds
	\[
	\Xi |\xi|^2:= 
	 \lim_{T\rightarrow +\infty} \frac{\mathrm{m}(\xi ;Q_T)}{T^2}
	\]
where $\Xi$ is a constant independent of the realization.
\end{lemma} 

\begin{remark}\rm
In light of Lemma \ref{lem:lim}, the constant $\Xi$ can be identified as
	\[
	\Xi:= \lim_{T\rightarrow +\infty} \frac{\mathrm{m} (e_1 ;Q_T)}{T^2}.
	\]
\end{remark}

The proof of Lemma \ref{lem:lim} comes as a consequence of Proposition \ref{prop:ex} below.  \smallskip

We denote by $M(\vartheta)$ a clockwise rotation of an angle $\vartheta$ around the origin. 

 \begin{proposition}\label{prop:ex}
For all $\xi\in \R^2$ there exists almost surely the limit
 	\[
 	f(\xi):=\lim_{T\rightarrow +\infty} \frac{\mathrm{m}(\xi;Q_T)}{T^d}
 	\]
independent of the realization. Moreover,  there exists a function $g: \R_+\rightarrow \R_+$ such that 
	\[
	\lim_{R\rightarrow +\infty} g(R)=+\infty
	\]
and such that, for any sequence $y_T$ with $|y_T|\leq Tg(y_T)$ and any rotation $M(\vartheta)$ we have
 	\begin{align*}
 \lim_{T\rightarrow +\infty} \frac{\mathrm{m} ((\xi; M(\vartheta)Q_T(y_T))}{T^d}=  \lim_{T\rightarrow +\infty} \frac{\mathrm{m} (\xi  ;Q_T)}{T^d}=f (\xi).
 	\end{align*}
 \end{proposition}
\begin{proof}
Note that 
	\[
	\mathrm{m}(\xi;A\cup B)\leq \mathrm{m}(\xi;A)+\mathrm{m}(\xi;B)
	\]
whenever $|A\cap B|=0$.  In particular,  by arguing exactly as in the proof of \cite[Lemma 5.1]{braides2019homogenization},  invoking the uniform version of the sub-additive ergodic theorem, \cite[Theorem 1]{krengel1987uniform}, we can achieve the existence of $g$ and $f$ such that for any family of translations $\{y_T\}_{T\in \N}$ satisfying $|y_T|\leq T g(|y_T|)$ it holds
	\[
	f(\xi)=\lim_{T\rightarrow +\infty} \frac{\mathrm{m}(\xi ;M(\vartheta) Q_T(y_T))}{T^d}.
	\]
as desired.
\end{proof}
We now focus on proving Lemma \ref{lem:lim}. During the proof we find it convenient to explicit the dependence of $\F_{1}$ and  $\mathrm{m}_{\delta}$ on $\eta$ and $\lambda$. In particular, if $\eta'$ is another Poisson point process, $\lambda'\in \R$ we write
	\[
	\F_1(v;A,\eta'\lambda'):=\sum_{x\in \eta'\cap A} \sum_{y\in \eta'\cap B_{\lambda'(x)}} |v(x)-v(y)|^2,
	\]
and
	\begin{equation*}
	\mathrm{m}(\xi; A,\eta',\lambda'):=\inf\left\{ \F_1(v;A,\eta'\lambda') \ \left| \begin{array}{c}
	v:\eta'\rightarrow \R\\
 \ v(x)=\xi\cdot x  \ \ \text{for all $x\in \eta'$ such that}  \\ 
 \mathrm{dist}(x,\partial A)\leq 2\lambda'
\end{array}	\right. \right\}.
	\end{equation*}
We also refer, whenever needed, to $f_{\eta'}(\xi,\lambda')$ as the limit in $T$ of $\frac{\mathrm{m}(\xi; A,\eta',\lambda')}{T^2}$ (which exists because the argument in Proposition \ref{prop:ex} applies to a generic Poisson point process $\eta'$). 
\begin{proof}[Proof of Lemma \ref{lem:lim}]
We argue as follows. We prove the following two relations on $f$:
\begin{align}
f (M(\vartheta)\xi)&=f(\xi) \ \ \ \text{for all $\vartheta\in [0,2\pi)$} \label{a}\\
f(r\xi)&=r^2f(\xi)  \ \ \ \text{for all $r\in \R_+$} \label{d}.
\end{align} 
Equations \eqref{a} and \eqref{d} tells us that, setting
$
	\Xi:=f(e_1)
$
then
$
	f(\xi)=\Xi |\xi|^2
$.
We proceed then to the proof of \eqref{a} and \eqref{d} separately.
\smallskip

\textbf{Step one:} \textit{invariance by rotation}. Observe that
\begin{align*}
	\F_1(v;Q_T,\eta,\lambda)&=\sum_{x\in \eta\cap Q_T} \sum_{y\in \eta\cap B_{\lambda}(x)} |v(x)-v(y)|^2\\
	&=\sum_{z\in M(\vartheta) \eta\cap (M(\vartheta) Q_T)} \sum_{w\in M(\vartheta) \eta\cap B_{\lambda}(z)} |v(M(-\vartheta) z)-v(M(-\vartheta) w)|^2\\
	&=\F_1(v(M(-\vartheta)\cdot ), M(\vartheta) Q_T, M(\vartheta) \eta, \lambda).
\end{align*}
If $v=\xi \cdot x$ on $(\partial Q_T)_{2\lambda}$, then $v(M(-\vartheta)x)=\xi \cdot( M(-\vartheta)x)$ for $x\in  (\partial (M(\vartheta)Q_T))_{2\lambda}$. In particular
	\begin{align*}
	\mathrm{m}  (M(-\vartheta)\xi;M(-\vartheta)Q_T, \eta,\lambda)=\mathrm{m}(\xi; Q_T, M(\vartheta)\eta,\lambda)
\end{align*}	 
By dividing by $T^2$ and taking the limit, recalling that the limit exists (Proposition \ref{prop:ex}) , we get
	\begin{align}
	f_{M(\vartheta)\eta}(\xi,\lambda)&=\lim_{T\rightarrow +\infty} \frac{\mathrm{m} (\xi; Q_T, M(\vartheta)\eta,\lambda)}{T^2}\nonumber\\
	&=\lim_{T\rightarrow +\infty} \frac{\mathrm{m} (M(-\vartheta)\xi; M(-\vartheta)Q_T, \eta,\lambda)}{T^2}=f_{\eta}(M(-\vartheta)\xi,\lambda).\label{a1}
	\end{align}
Noting that
$M(\vartheta)\eta=\eta$ in distribution,
we conclude that
$\mathrm{m}(\xi;Q,M(\vartheta) \eta,\lambda)=	\mathrm{m}(\xi;Q_T, \eta,\lambda)$ in distribution.
This equality in distribution implies that
$	f_{\eta}(\xi,\lambda)=f_{M(\vartheta) \eta}(\xi,\lambda)$ in distribution.
Since $f_{\eta}$ and  $f_{M(\vartheta)\eta}$ are independent of the realizations, as stated in Proposition \ref{prop:ex}, we conclude that
	\begin{equation}\label{a2}
	f_{\eta}(\xi,\lambda)=f_{M(\vartheta)\eta}(\xi,\lambda).
	\end{equation}
Thus, by combining \eqref{a1}, \eqref{a2} we get \eqref{a}.
\smallskip

\textbf{Step two:} \textit{positive homogeneity of degree two}. Fix $r\in \R_+$ and observe that
	\begin{align*}
	\F_1(v;Q_T,\eta,\lambda) =&\sum_{x\in \eta\cap Q_T}\sum_{y\in B_{\lambda}(x)} |v(x)-v(y)|^2\\
	=&\sum_{z\in \eta/r\cap Q_{T/r}}\sum_{w \in B_{\lambda/r}(z)} |v(r z)-v(r w)|^2\\
	=&r^2 \F_1(\sfrac{v(r\cdot)}{r}; Q_{T/r},\eta/r,\lambda/r).
	\end{align*}
If $v(x)=\xi \cdot x$ on $x\in (\partial Q_T)_{2\lambda}$ then $\frac{v(r x)}{r}=\xi \cdot x$ on $x \in (\partial Q_{T/r})_{2\lambda/r}$. In particular, we also get that
$
	\mathrm{m} (\xi ;Q_T,\eta,\lambda)=	r^2\mathrm{m}  (\xi ;Q_{T/r},\eta/r,\lambda/r)
$.
By dividing by $T^2$ and taking the limit, using the existence result in Proposition \ref{prop:ex}, we get (setting $\tilde{T}=T/r$)
	\begin{align}
	f_{\eta}(\xi,\lambda)&=\lim_{T\rightarrow +\infty} \frac{\mathrm{m} (\xi ;Q_T,\eta,\lambda)}{T^2}=\lim_{T\rightarrow +\infty} \frac{r^2\mathrm{m} (\xi ;Q_{T/r},\eta/r,\lambda/r)}{T^2}\nonumber\\
	&=\lim_{\tilde{T}\rightarrow +\infty} \frac{\mathrm{m} (\xi ;Q_{\tilde{T}},\eta/r,\lambda/r)}{\tilde{T}^2}=f_{\eta/r}(\xi,\lambda/r).\label{b}
	\end{align}
Analogously we have
	\[
	\F_1(v;Q_T,\eta,\lambda)=r^4\F_1(\sfrac{v(r\cdot )}{r^2};Q_{T/r},\eta/r,\lambda/r).
	\]
Thus, if $v(x)=(r\xi\cdot x)$ on $(\partial Q_{T})_{2\lambda}$ then $\frac{v(rx)}{r^2}=\xi\cdot x$ on $(\partial Q_{T/r})_{2\lambda/r}$. As a consequence we also have the equality
	\begin{align*}
	\mathrm{m} (r \xi  ;Q_T,\eta,\lambda)=	r^4\mathrm{m}  (\xi;Q_{T/r},\eta/r,\lambda/r),
	\end{align*}
which, dividing by $T$ and sending to $T\rightarrow +\infty$,  still by Proposition \ref{prop:ex} yields
	\begin{align}
	f_{\eta}(r \xi,\lambda)&=\lim_{T\rightarrow}\frac{ \mathrm{m}(r \xi ;Q_T,\eta,\lambda)}{T^2}=r^4\lim_{T\rightarrow +\infty} \frac{\mathrm{m}  (\xi ;Q_{T/r},\eta/r,\lambda/r)}{T^2}\nonumber\\
	&=r^2\lim_{\tilde{T}=T/r \rightarrow +\infty} \frac{\mathrm{m} (\xi ;Q_{\tilde{T}},\eta/r,\lambda/r)}{\tilde{T}}\nonumber\\
	&=r^2 f_{\eta/r}(\xi,\lambda/r)\label{c}.
	\end{align}
By combining \eqref{b},\eqref{c} we thus get $
	f_{\eta}(r \xi,\lambda)=r^2 f_{\eta}(\xi,\lambda)$; that is, \eqref{d}.
\end{proof}

\begin{figure}
\includegraphics[scale=0.6]{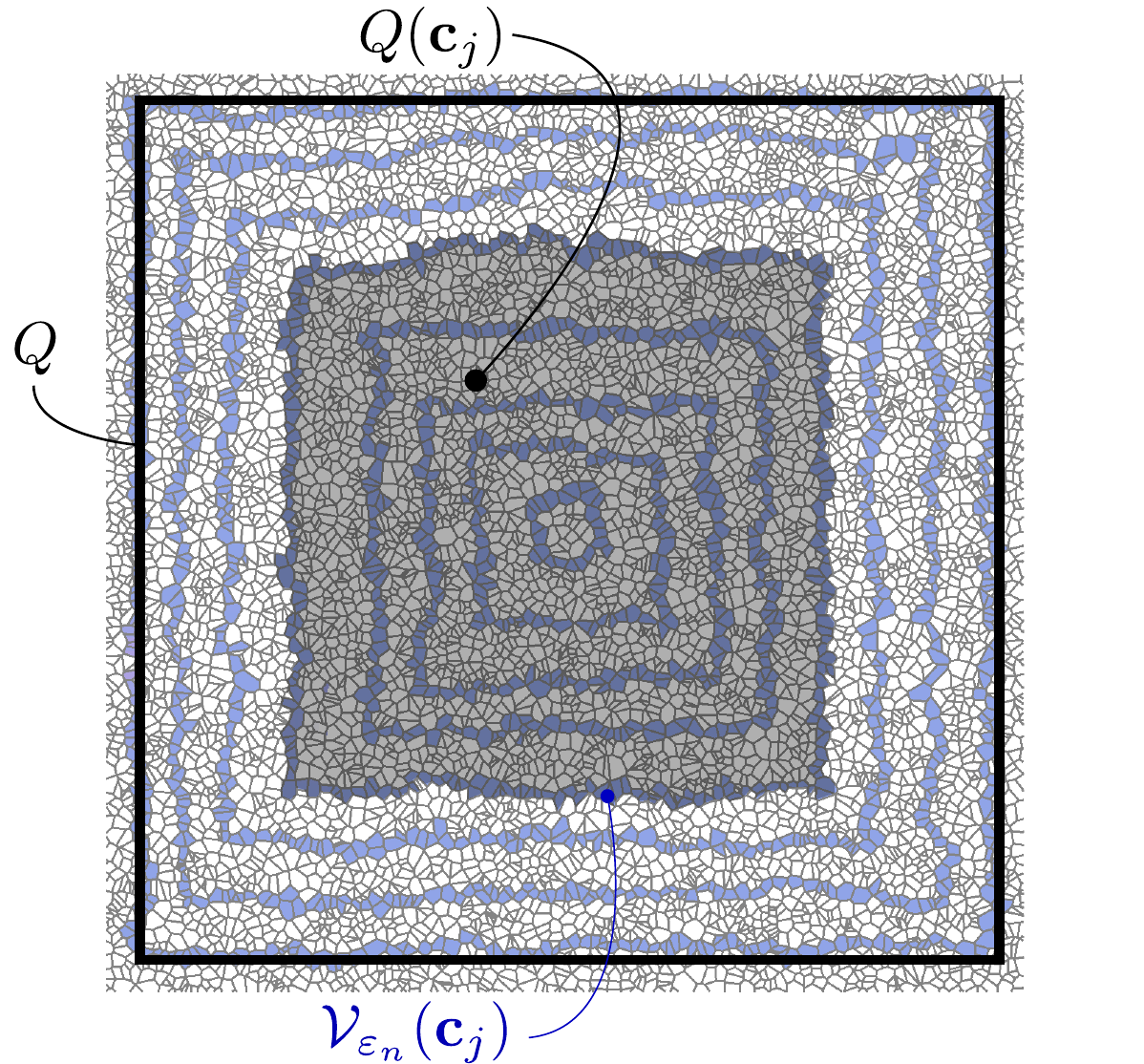}\\
\includegraphics[scale=0.5]{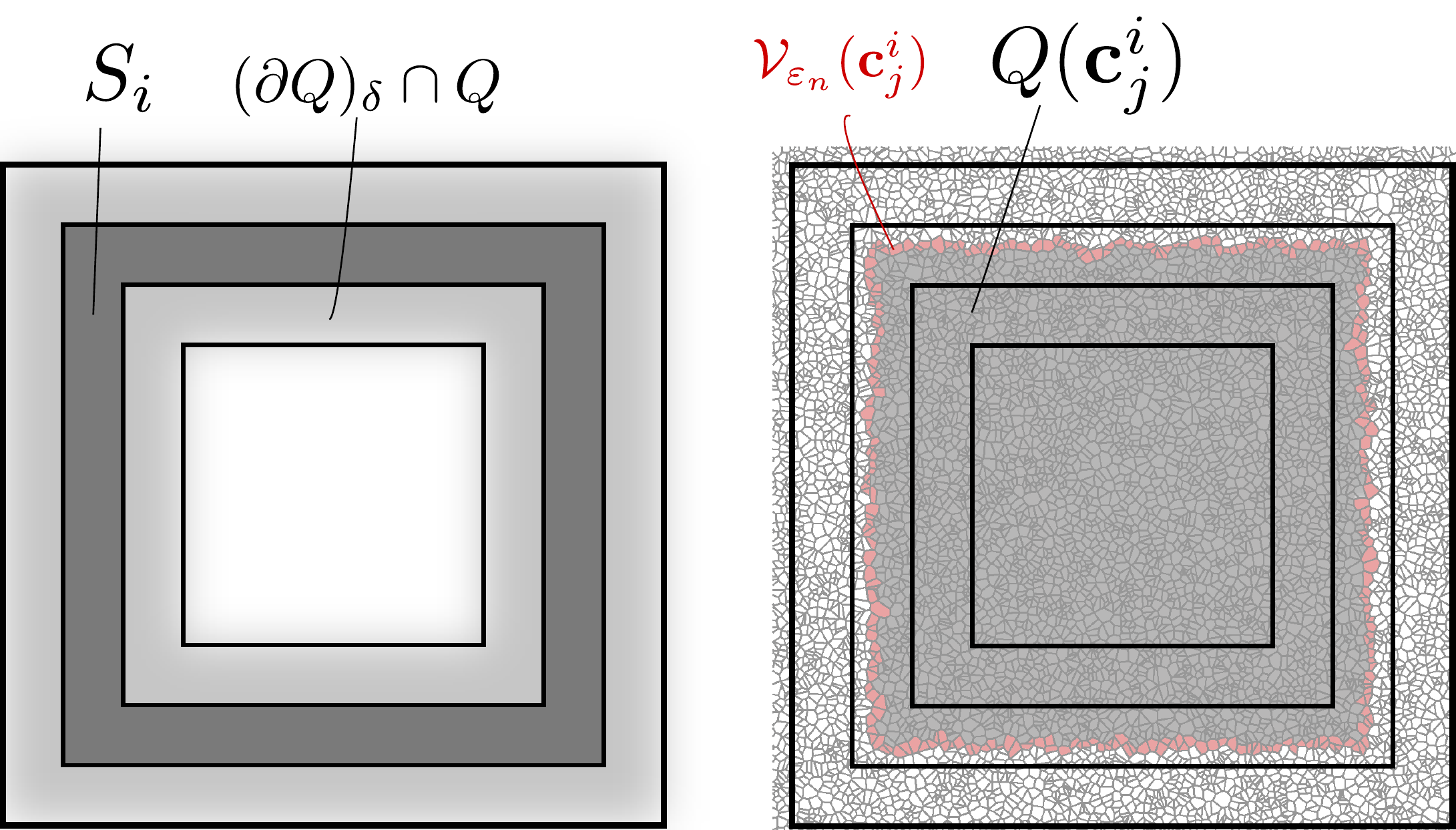}
\caption{In the proof of Proposition \ref{prop:bndDATA} we refer to the notation depicted here. .}\label{fig:bdFix}
\end{figure}
\subsection{A boundary-value fixing argument}

We now concentrate on a key ingredient of these types of results; that is, the possibility of  modifying  boundary values.   Before proceeding, we introduce the notation by referring to Figure \ref{fig:bdFix}.   We fix $\delta>0,N>0$ and we divide $(\partial Q_1)_{\delta}\cap Q$ (depicted in soft grey on the left in Figure \ref{fig:bdFix}) in $N$ sectors $S_i$ of size $\delta/N$ (see one of them in dark grey on the left).  Given $G_{\e_n,t_n}\in \mathcal{G}_{t_n}(\Upsilon;\eta_n)$ inside each sector we can find $\mathbf{c}_1^i,\ldots, \mathbf{c}_{K_n}^i$ disjoint ``annuli'' (one of them is depicted in red on the right) composed of portions of paths from the grid and all contained in $S_i$.  Still the good properties of the grid allows us to estimate $K_n\approx \delta/\e_n$.  We consider this annuli labelled increasingly from the outer one, and we call $Q(\mathbf{c}^i_j)$ the portion of the square bounded by the Voronoi cells of $\mathbf{c}^i_j$ and including them (on the right: the union of the region depicted in dark grey and all the red regions).  Clearly $Q(\mathbf{c}^i_{j'})\subset Q(\mathbf{c}^i_j)$ for $j'\geq j$.  We use this geometry to build cut off functions $h_i$ that we will use to change the boundary data of a sequence of functions $u_n$ converging to $u$ in the sense of Definition \ref{def:conv}.

\begin{proposition}\label{prop:bndDATA}
Let $U\subset Q$, be an open set with Lipschitz boundary.  Let $\{x_n\}_{n\in \N}$ be a sequence of points and pick $\{G_{\e_n,t_n}\in \mathcal{G}_{t_n}(\Upsilon;\eta_n-\e_n x_n)\}_{n\in \N}$ a sequence of grids and a sequence of functions $\{u_n\in L^2(Q;\e_n(\eta-x_n) ) \}_{n\in \N}$ such that
	\[
	  \lim_{n\rightarrow +\infty} \int_{\mathcal{V}_{\e_n}(G_{\e_n,t_n})\cap U} |\hat{u}_{n}(x)-u(x)|^2\d x =0
	\]
for some $u\in C^1(U)$. Suppose that the sequence also satisfies
	\[
	\sup_{n\in \N} \{\F_n(u_n;A,\eta_{\e_n}-x_n\e_n),\lambda)\}=\zeta <+\infty
	\]
for a supset $A\supset Q$. Then, for any $\delta>0$  there exists a sequence $\{v_n\}_{n\in \N}$ such that 
\begin{equation*}
\begin{array}{c}
	\displaystyle v_n=u \ \ \ \ \text{on $(\eta_{\e_n}-x_n\e_n) \cap (\partial U)_{\delta} \cap U   $},\\
	\text{}\\
\displaystyle \lim_{n\rightarrow +\infty} \int_{\mathcal{V}_{\e_n}(G_{\e_n,t_n})\cap U} |\hat{v}_n(x)-u(x)|^2\d x =0
\end{array}
\end{equation*}
and 
   \begin{equation}\label{eqn:finalBNDfixing}
	\liminf_{n\rightarrow +\infty} \F_n(v_n;U,\eta_{\e_n}-x_n\e_n,\lambda)  \leq \liminf_{n\rightarrow +\infty} \F_n(u_n;U,\e_n(\eta-x_n),\lambda) + CP(U) \|\nabla u\|_{\infty}^2\delta
   \end{equation}
	where $C=C(\alpha,\lambda)$ depends on $\alpha$ and $\lambda$ only and $P(U)$ denotes the perimeter of $U$. 
\end{proposition}
\begin{proof}
For the sake of simplicity we will prove the result only in the case $U=Q$ since the general case results only in a heavier notation.  So we assume $U=Q$ and we fix $\delta>0$ and $N>0$.  Consider
	\[
	S_{i}:=Q_{1-\frac{(i-1)}{N}\delta}\setminus Q_{1-\frac{i}{
N}\delta}, \ \ \ i=1,\ldots, N.
	\] 
For any fixed $n>0$, if $G_{\e_n,t_n}\in  \mathcal{G}_{t_n}(\Upsilon;\eta_n-\e_n x_n)$ is a regular grid, by joining the paths of the grid suitably,  as in the proof of Theorem \ref{thm:pathConn} (see Appendix)
%
%
for any $S_i$ we can find (and eventually relabel) annuli $\mathbf{c}^i_1,\ldots,\mathbf{c}^i_{K_n}$ where
	\[
	\frac{\delta }{\Upsilon \e_n} \leq K_n\leq \frac{\Upsilon\delta  }{\e_n}
	\] 
for a $\Upsilon$ uniform in $n$ (see Figure \ref{fig:bdFix}). We moreover observe,  due to the properties of the grid, that
	\begin{align}
 \mathrm{dist}(\mathbf{c}_j^i, \mathbf{c}_{j'}^{i'} )\geq 3\lambda \e_n \ \ \ \ &   \text{for all $i,j,i'j'$}; \label{prop1}\\
\mathrm{dist}(\mathbf{c}_j^i,\partial S_i)  \geq 3\lambda \e_n \ \ \ \ &   \text{for all $i,j$}.\label{prop2}
	\end{align} 
Moreover, we have
	\begin{align}\label{prop3}
	\text{if $x\in (\mathbf{c}_j^i)_{3\lambda\e_n}\cap \tilde{\eta}_n$ then $\tilde{\eta}_n(B_{\lambda\e_n}(x))\leq \frac{1}{\alpha} \lambda^2$}. 
	\end{align}
	
	\begin{figure}
\includegraphics[scale=.7]{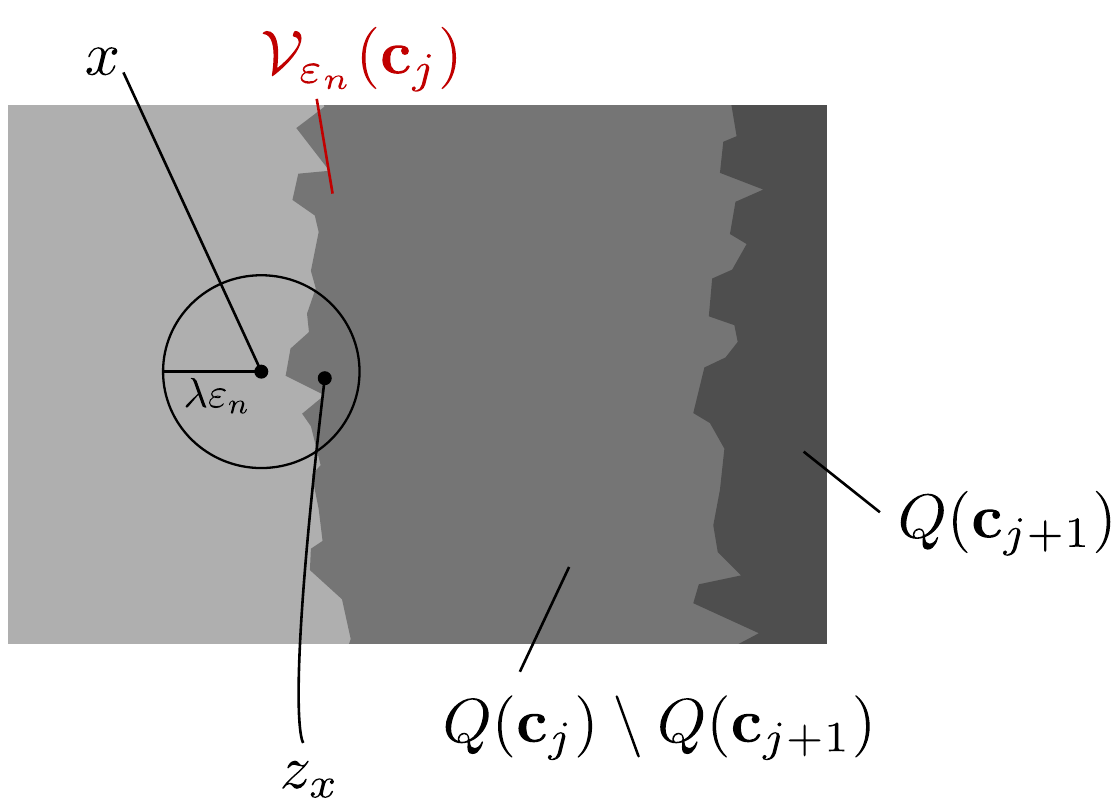}
\caption{The situation in the proof of Proposition \ref{prop:bndDATA}.  The difference $h_n(x)-h_n(y)$ is different from $0$ only on pairs $(x,y)$ satisfying the relation \eqref{ghig}. }\label{fig:Pointcrown}
\end{figure}
Both properties \eqref{prop1} and \eqref{prop2} derive from property (e) of Definition \ref{def:pathConn} and from the fact that $\mathbf{c}_j^i$ is made of paths of $G_{\e_n,t_n}$.  Property \eqref{prop3} is instead consequence of Property (f) of  Definition \ref{def:pathConn}.  Let $Q(\mathbf{c}^i_j)$ denote the portion of the square bounded by $\mathbf{c}^i_j$ and containing the Voronoi cells of points in $\mathbf{c}_j^i$ (we refer to Figures \ref{fig:bdFix} and \ref{fig:Pointcrown}).  For $i=1,\ldots, N$ we set
\begin{equation}
h^i_{n}(x):=
\left\{
\begin{array}{ll}
0 \ \ \ \ \ &  \text{if $x\in Q(\mathbf{c}^i_1)^c$} \\
\frac{s}{K_n} \ \ \ \ \ &  \text{if $x\in Q(\mathbf{c}^i_s)\setminus Q(\mathbf{c}^i_{s-1})$}\\
1 \ \ \ \ \ &  \text{if $x\in Q(\mathbf{c}_{K_n})$}.
\end{array}
\right.
\end{equation}
Finally we adopt the shorthand $\tilde{\eta}_n:=\e_n(\eta-x_n)$. Note that
\begin{align*}
	\sum_{i=1}^N \sum_{x\in S_i\cap\tilde{\eta}_n}& \sum_{y\in B_{\lambda\e_n}(x)\cap \tilde{\eta}_n} (u_n(x)-u_n(y))^2
\leq 2 \F_n(u_n;Q \setminus Q_{1-\delta},\tilde{\eta}_n,\lambda),
	\end{align*}
where the constant $2$ arises since the interaction around $\partial S_i$ are counted twice when summed up over $i$.  We can thus pick $i=1,\ldots, N$ for which it holds
	\begin{align*}
	\sum_{x\in S_i\cap\tilde{\eta}_n}& \sum_{y\in B_{\lambda\e_n}(x)\cap \tilde{\eta}_n} (u_n(x)-u_n(y))^2+(u(x)-u(y))^2\\
	\leq& \frac{2}{N} (\F_n(u_n;Q\setminus Q_{1-\delta},\tilde{\eta}_n,\lambda)+\F_n(u;Q\setminus Q_{1-\delta},\tilde{\eta}_n,\lambda))
	\end{align*}

Then, we set $h_{n}=h^i_{n}$, $\mathbf{c}_j=\mathbf{c}^i_j$ and
\begin{align*}
v_{n}(x):=(1-h_{n}(x))u(x)+h_{n}(x)u_n(x)
\end{align*}
Set also $G_n:=G_{\e_n,t_n}$.  We immediately conclude that
 	\begin{align*}
 	\e_n^2\sum_{x\in G_{n}\cap Q} (v_n(x)-u(x))^2&=\e_n^2 \sum_{x\in G_{n} \cap S_i}  (v_n(x)-u(x))^2(1-h_n(x))^2\\
 	&\leq  \e_n^2  \sum_{x\in G_{n} \cap Q}  (u_n(x)-u(x))^2,
 	\end{align*}
giving $v_n\rightarrow u$. We also note that
\begin{align*}
v_{n}(x)-v_{n}(y)=&(1-h_n(y))(u(y)-u(x))+h_{n}(y) (u_n(x)-u_n(y))\\
&+(h_{n}(x)-h_{n}(y))(u_n(x)-u(x)).
\end{align*}
Thanks to the structure of $v_{n}$ we thus have
	\begin{align*}
	\sum_{x\in Q\cap\tilde{\eta}_n} \sum_{y\in B_{\lambda\e_n}(x)\cap \tilde{\eta}_n} (v_{n}(x)-v_{n}(y))^2 \leq& \sum_{x\in (Q\setminus S_i)\cap\tilde{\eta}_n} \sum_{y\in B_{\lambda\e_n}(x)\cap \tilde{\eta}_n} (v_{n}(x)-v_{n}(y))^2 \\
	&+ \sum_{x\in S_i\cap\tilde{\eta}_n} \sum_{y\in B_{\lambda\e_n}(x)\cap \tilde{\eta}_n} (v_{n}(x)-v_{n}(y))^2  \\
&\leq 	
	\F_n(u_n;Q,\tilde{\eta}_n,\lambda)+\F_n(u;Q\setminus Q_{1-\delta},\tilde{\eta}_n,\lambda)\\
&+\sum_{x\in S_i\cap\tilde{\eta}_n} \sum_{y\in B_{\lambda\e_n}(x)\cap \tilde{\eta}_n} (v_{n}(x)-v_{n}(y))^2,
	\end{align*}
	where the second inequality exploits property \eqref{prop2} and the fact that $v_n$ agrees with $u$ and $u_n$ on a slightly bigger sets than the two connected components of $Q\setminus S_i$.  The choice of $i$ now allows to estimate
	\begin{align*}
	\sum_{x\in S_i\cap \tilde{\eta}_n}\sum_{y\in B_{\lambda\e_n}(x)\cap \tilde{\eta}_n} (v_{n}(x)-v_{n}(y))^2 \leq& \frac{C}{N}( \F_n(u_n;Q\setminus Q_{1-\delta},\tilde{\eta}_n,\lambda)+\F_n(u;Q\setminus Q_{1-\delta},\tilde{\eta}_n,\lambda))\\
	&+C\sum_{x\in S_i\cap\tilde{\eta}_n}(u_n(x)-u(x))^2\sum_{y\in B_{\lambda\e_n}(x)\cap\tilde{\eta}_n} (h_{n}(x)-h_{n}(y))^2,
	\end{align*}
where $C$, here and in the rest of the proof,  stands for a constant depending on $\alpha,\lambda$ only and that may vary from line to line.  We now exploit property \eqref{prop1}: the annuli paths lie at a certain distance between each other and therefore we have that
	\begin{align}
	(h_{n}(x)-h_{n}(y))^2\ca_{B_{\lambda\e_n}(x)}(y)&\leq C \e_n^2 \ \ \ \begin{array}{c}
	\text{if  $x\notin Q(\mathbf{c}_j)$,  $y\in \tilde{\eta}_n\cap  Q(\mathbf{c}_j)$  for some $j$ }\\
	\text{and $|x-y|\leq \lambda\e_n$}
	\end{array}
	\label{ghig}\\
	(h_{n}(x)-h_{n}(y))^2\ca_{B_{\lambda\e_n}(x)}(y) &=0 \ \ \ \ \ \ \ \text{otherwise}.
	\end{align}
In particular, we get
\begin{align*}
\sum_{x\in S_i\cap \tilde{\eta}_n}&(u_n(x)-u(x))^2\sum_{y\in B_{\lambda\e_n}(x)\cap \tilde{\eta}_n} (h_{n}(x)-h_{n}(y))^2\\
=& \sum_{\substack{(x,y)\\ \text{satisfies \eqref{ghig}} }} (u_n(x)-u(x))^2 (h_{n}(x)-h_{n}(y))^2\\
\leq &C\frac{\e_n^2}{\delta^2} \sum_{j=1}^{K_n} \sum_{\substack{x\in\tilde{\eta}_n: \\
x\in  Q(\mathbf{c}_j)^c\cap (\mathbf{c}_j)_{\lambda\e_n}}}  \sum_{y\in B_{\lambda\e_n}(x)\cap Q(\mathbf{c}_j) } (u_n(x)-u(x))^2\\
\leq &C\frac{\e_n^2}{\delta^2} \sum_{j=1}^{K_n} \sum_{\substack{x\in\tilde{\eta}_n: \\
x\in  Q(\mathbf{c}_j)^c\cap (\mathbf{c}_j)_{\lambda\e_n}} }  (u_n(x)-u(x))^2,
\end{align*}
where the last equality follows from property \eqref{prop3} of the annuli.
For $x\in \tilde{\eta}_n\cap  Q(\mathbf{c}_j)^c\cap (\mathbf{c}_j)_{\lambda\e_n}$, let $z_x\in \mathbf{c}_j$ be such that $|x-z_x|\leq \lambda \e_n $ (see Figure \ref{fig:Pointcrown}). 
 Then 
	\begin{align*}
 \sum_{\substack{x\in\tilde{\eta}_n: \\
x\in  Q(\mathbf{c}_j)^c\cap (\mathbf{c}_j)_{\lambda\e_n}} }  (u_n(x)-u(x))^2\leq &\ C\Bigl(  \sum_{x\in Q(\mathbf{c}_j)^c\cap (\mathbf{c}_j)_{\lambda\e_n}}  (u_n(x)-u_n(z_x))^2+\hskip-.8cm  \sum_{\substack{x\in\tilde{\eta}_n: \\
x\in  Q(\mathbf{c}_j)^c\cap (\mathbf{c}_j)_{\lambda\e_n}} }  (u(x)-u(z_x))^2\\
 &+   \sum_{\substack{x\in\tilde{\eta}_n: \\
x\in  Q(\mathbf{c}_j)^c\cap (\mathbf{c}_j)_{\lambda\e_n}} }(u_n(z_x)-u(z_x))^2\Bigr)\\
 \leq & C  \sum_{x\in \mathbf{c}_j}  \sum_{y\in B_{\lambda\e_n}(x)\cap \tilde{\eta}_n} (u_n(x)-u_n(y))^2+(u(x)-u(y))^2\\
  &+C \sum_{x\in \mathbf{c}_j }  (u_n(x)-u(x))^2 ,
	\end{align*}
where the last inequality follow from property \eqref{prop3} and the fact that $|x-z_x|\leq \lambda\e_n$.  Thus, by summing up over $j=1,\ldots ,K_n$ we obtain
	\begin{align*}
	 \sum_{j=1}^{K_n} \sum_{\substack{x\in\tilde{\eta}_n: \\
x\in  Q(\mathbf{c}_j)^c\cap (\mathbf{c}_j)_{\lambda\e_n}} }  (u_n(x)-u(x))^2\leq &\ C \F_n(u_n;Q ,\tilde{\eta}_n,\lambda) + C \F_n(u;Q ,\tilde{\eta}_n,\lambda)\\
  &+C \sum_{x\in G_n }  (u_n(x)-u(x))^2 
\end{align*}	 Then we conclude that
\begin{align*}
	\F_n(v_n; Q,\tilde{\eta}_n,\lambda) \leq&\ \F_n(u_n; Q,\tilde{\eta}_n,\lambda) + \F_n(u; Q\setminus Q_{1-\delta} ,\tilde{\eta}_n,\lambda)\\
	&+ \frac{C}{N}  (\F_n(u_n; Q\setminus Q_{1-\delta} ,\tilde{\eta}_n,\lambda)+  \F_n(u; Q\setminus Q_{1-\delta} ,\tilde{\eta}_n,\lambda) )\\
	&+\frac{C}{\delta^2}\left[\e_n^2 \F_n(u_n;S_i,\tilde{\eta}_n,\lambda) +\e_n^2 \F_n(u;S_i,\tilde{\eta}_n,\lambda) + \e_n^2 \sum_{x\in G_n }  (u_n(x)-u(x))^2  \right],
\end{align*}	 
where the constant $C$ is independent of $n,N,\delta$.  We now use Proposition \ref{propo:count} (by observing that $\eta-x_n$ has the same distribution than $\eta$) and consider the limit in $n$ and achieve 
\begin{align*}
\liminf_{n\rightarrow +\infty} \F_n(v_{n};Q,\tilde{\eta}_n,\lambda)\leq& \liminf_{n\rightarrow +\infty} \F_n(u_n;Q,\tilde{\eta}_n,\lambda) + C P(Q) \|\nabla u\|_{\infty}^2 \delta  \\
&+\frac{C}{N}( 1+ P(Q)\|\nabla u\|_{\infty}^2 \delta).
\end{align*}
A further limit in $N$ yields \eqref{eqn:finalBNDfixing}.
\end{proof}

\begin{proposition}[Blow-up]\label{prop:blw}
Let $u_n\rightarrow u$, $u\in W^{1,2}(Q)$ and pick $x_0\in Q$ a Lebesgue point of $\nabla u$. Fix a sequence of regular grids $\{G_{\e_n,t} \in \mathcal{G}_{t}(\Upsilon;\eta_{\e_n})\}_{t\in \R_+}$ such that
	\[
\lim_{t\rightarrow 0} \limsup_{n\rightarrow+\infty} \int_{G_{\e_n,t}\cap Q} |u_n(x)-u(x)|^2\d x= 0.
	\]
Then, for any $\rho>0$ it holds
\[
G^{\rho}_{\e_n,t}(x_0):=\frac{G_{\e_n,t}\cap Q_{\rho}-x_0}{\rho }\in \mathcal{G}_{t/\rho} \left(\Upsilon;  \eta_{\e_n}^{\rho,x_0}  \right),
\]
where
\[
\eta_{\e_n}^{\rho,x_0}:= \frac{\eta_{\e_n} - x_0}{\rho}.
\]
Moreover, for any $\delta>0$ we can choose two sequences $t_n$, $\rho_n\rightarrow 0$ such that
	\begin{equation}\label{eqn:BL1}
	\frac{ t_n}{\rho_n}\rightarrow 0, \ \ \ \frac{\e_n}{\rho_n}\rightarrow 0, \ \ \frac{|x_0|}{\rho_n}\leq g\left(\frac{|x_0|}{\e_n}\right),
	\end{equation}
where $g$ is the function given by Proposition \ref{prop:ex}.
Finally, setting 
	\[
	G_n(x_0):= \frac{G_{\e_n,t_n}\cap Q_{\rho_n}-x_0}{\rho_n}, 
	\]
 we have
	\begin{equation}\label{eqn:BL2}
	\lim_{n\rightarrow +\infty} \int_{\mathcal{V}_{\e_n}(G_{n}(x_0))\cap Q} |\hat{u}^{\rho_n,x_0}_{n}(x)-\nabla u(x_0)\cdot x|^2\d x=0,
	\end{equation}
	where
	 \begin{align*}
 u^{\rho, x_0}_{n}(x):=\frac{u_n(x_0+\rho x)-u(x_0)}{\rho} \ \ \ \ \text{for $x\in \eta^{\rho,x_0}_{\e_n}\cap Q $}
 \end{align*}
 
\end{proposition}
 \begin{proof}
We start by proving that $G_{\e_n,t}^{\rho}(x_0)\in\mathcal{G}_{t/\rho} \left(\Upsilon;  \eta_{\e_n}^{\rho,x_0}  \right)$.  Indeed, property (b) and (c) follow immediately by construction.  By scaling we also get property  (e), (f) and (g) from their validity on $G_{\e_n,t}$.  Since $G_{\e_n,t}\in \e_n\eta^{\alpha}(\lambda)$ we also obtain property (a): $G_{\e_n,t}^{\rho}(x_0)\subset \left(\frac{\eta^{\alpha}(\lambda)-x_0}{\e_n}\right)$. Property (d) is immediate since, by replacing $t$ with $t/\rho$ and $\e_n$ with $\e_n/\rho$ the bounds \eqref{eqn:unifBnd} given by $\Upsilon$ are still in force. \smallskip

The second part of the statement comes just from a diagonalization argument in $n, t, \rho$ and by exploiting that $x_0$ is a Lebesgue point of $\nabla u$. 
 \end{proof}

\subsection{Proof of the lower bound} \label{sbsct:liminf}
We follow the blow-up method by Fonseca and M\"uller \cite{FM} (see also \cite{BMS} for its adaptation to homogenization).
Let $u_\e\rightarrow u$ in the sense of Definition \ref{def:conv}.  Without loss of generality we can assume that 
	\[
	\liminf_{\e\rightarrow 0} \F_{\e}(u_{\e}; Q )<+\infty
	\]
Fix $x_0\in Q$ a Lebesgue point of $\nabla u$ and $u$, and a subsequence $\e_n\rightarrow 0$ achieving the $\liminf$. Define
	\[
	\mu_{n}(A):=\F_{\e_n}(u_{\e_n};A).
	\]
Then $\mu_n\wt \mu$ (up to a subsequence) for some measure $\mu$, and $u_n:=u_{\e_n}\rightarrow u$ in the sense of Definition \ref{def:conv}.  Moreover $u\in W^{1,2}(Q)$ due to Theorem \ref{thm:Compactness}. We would like to show that
	\[
	\frac{\d \mu}{\d \L^2}(x)\geq\Xi |\nabla u(x)|^2
	\]
for $\L^2$-almost every $x\in Q$.  This would imply that
	\[
	\liminf_{\e\rightarrow 0} \F_{\e}(u_{\e};Q)=\lim_{n\rightarrow +\infty} \F_{\e_n}(u_n;Q)\geq \Xi \int_Q |\nabla u(x)|^2\d x.
	\]
	
Note that
	\begin{align*}
	\frac{\d \mu}{\d \L^2}(x_0)=\lim_{\rho\rightarrow 0} \frac{\mu(Q_{\rho}(x_0))}{\rho^2}=\lim_{\rho\rightarrow 0}\lim_{n\rightarrow +\infty} \frac{\mu_n(Q_{\rho}(x_0))}{\rho^2}.
	\end{align*}
Since $u_n\rightarrow u$ then, for any $t\in \R_+$ and any grid $G_{\e_n,t}$ there exists a $u^t\in X_t$ such that $T^{G_{\e_n,t}}(u_n)\rightarrow u^t$. With fixed $\delta$,  by invoking Proposition \ref{prop:blw} we can find two subsequences $\rho_n,t_n$  such that \eqref{eqn:BL1} and \eqref{eqn:BL2} holds. 
By relabeling $\tilde{\e}_n:=\e_n/\rho_n$,  $\tilde{t}_n=t_n/\rho_n$, $x_n:=x_0/\e_n$ and by invoking Proposition \ref{prop:bndDATA} we can find $\{v_n\in L^2(Q,\tilde{\e}_n(\eta-x_n))\}_{n\in \N}$ such that
	\[
	v_n(x)=\nabla u(x_0)\cdot x \ \ \ \ \text{on $Q_1\setminus Q_{1-\delta}$}
	\]
and 
\[
\liminf_{n\rightarrow +\infty} \F_{\tilde{\e}_n} (v_n;Q, \tilde{\e}_n(\eta-x_n),\lambda)\leq\liminf_{n\rightarrow +\infty} \F_{\tilde{\e}_n} (u_n^{\rho_n,x_0} ;Q, \tilde{\e}_n(\eta-x_n),\lambda) + C|\nabla u(x_0)|^2 \delta.
\]
Observe now that
	\[
	\frac{\d \mu}{\d \L^2}(x)=\lim_{n\rightarrow +\infty} \frac{\F_n(u_n;Q_{\rho_n}(x_0);\e_n \eta,\lambda)}{\rho_n^2}.
	\]
Moreover the following holds
	\begin{align*}
	\F_n(u_n;Q_{\rho_n}(x_0);\e_n \eta,\lambda)&= \sum_{x\in Q_{\rho_n}(x_0) \cap \e_n\eta} \sum_{y\in B_{\lambda\e_n}(x) \cap \e_n\eta} |u_n(x)-u_n(y)|^2\\
	&= \sum_{x\in Q_{\rho_n}  \cap (\e_n\eta-x_0)} \sum_{y \in B_{\lambda\e_n}(x) \cap (\e_n\eta-x_0)} |u_n(x_0+x)-u_n(x_0+y)|^2\\
&= \sum_{x \in Q \cap \frac{\e_n}{\rho_n}(\eta-x_n)} \sum_{y \in B_{\lambda\frac{\e_n}{\rho_n} }(x) \cap \frac{\e_n}{\rho_n}(\eta-x_n)} |u_n(x_0+\rho_n x)-u_n(x_0+\rho_n y)|^2\\
&=\rho_n^2 \sum_{x \in Q \cap \tilde{\e}_n (\eta-x_n)} \sum_{y \in B_{\lambda \tilde{\e}_n }(x) \cap \tilde{\e}_n (\eta-x_n)} \left|u_n^{\rho_n,x_0} (x)-u_n^{\rho_n,x_0} (y)\right|^2\\
&=\rho_n^2 \F_{\tilde{\e}_n} ( u_n^{\rho_n,x_0} ; Q, \tilde{\e}_n(\eta-x_n),\lambda).
	\end{align*}
In particular, we have  
	\begin{align*}
	\frac{\d \mu}{\d \L^2} (x) =&\lim_{n\rightarrow +\infty} \frac{\F_n(u_n;Q_{\rho_n}(x_0);\e_n \eta,\lambda)}{\rho_n^2}\\
	&\geq  -C\delta|\xi|^2+ \liminf_{n\rightarrow +\infty}   \F_{\tilde{\e}_n} (v_n;Q,\tilde{\e}_n(\eta-x_n),\lambda),
	\end{align*}
and finally
\begin{align*}
\F_{\tilde{\e}_n} (v_n;Q,\tilde{\e}_n(\eta-x_n),\lambda)&=\sum_{x \in Q \cap \tilde{\e}_n (\eta-x_n)} \sum_{y \in B_{\lambda \tilde{\e}_n }(x) \cap \tilde{\e}_n (\eta-x_n)} \left|v_n(x)-v_n (y)\right|^2\\
&=\tilde{\e}_n^2 \sum_{ x \in Q_{1/\tilde{\e_n}}(x_n) \cap \eta } \sum_{y \in B_{\lambda  }(x) \cap \eta } \left|\frac{v_n\left(\tilde{\e}_n (x-x_n)\right)}{\tilde{\e}_n}-\frac{v_n (\tilde{\e}_n (y-x_n))}{\tilde{\e}_n}\right|^2.
\end{align*}
Since 
$v_n(x)=\nabla u(x_0)\cdot x $ for $x\in (Q\setminus Q_{1-\delta})\cap \tilde{\e}_n(\eta-x_n)$, 
we get \[
\frac{v_n\left( \tilde{\e}_n(x-x_n) \right)}{\tilde{\e}_n}+\nabla u(x_0)\cdot x_n =\nabla u(x_0)\cdot \left(\frac{y}{\tilde{\e}_n}+  x_n\right)=\nabla u(x_0)\cdot x
\]
for $x\in \left(Q_{1/\tilde{\e}_n }(x_n) \setminus Q_{1/\tilde{\e}_n(1-\delta)}(x_n) \right) \cap \eta$.
In particular, we have
	\[
	\F_{\tilde{\e}_n}(v_n;Q,\tilde{\e}_n(\eta-x_n),\lambda) \geq \tilde{\e}_n^2 \mathrm{m}(\nabla u(x_0)\cdot x;Q_{1/\tilde{\e}_n}(x_n) )  \geq\tilde{\e}_n^2 \mathrm{m}(\nabla u(x_0)\cdot x;Q_{1/\tilde{\e}_n}(x_n) ) 
	\]
Considering $T_n=\frac{1}{\tilde{\e}_n}$ and observing that
	\[
	|x_n|=\frac{|x_0|}{\rho_n \tilde{e}_n}\leq T_n \frac{|x_0|}{\rho_n}\leq T_n g_{\delta}(|x_n|),
	\]
we conclude  that
 \[
 \lim_{n\rightarrow +\infty} \tilde{\e}_n^2 \mathrm{m}(\nabla u(x_0)\cdot x;Q_{1/\tilde{\e}_n}(x_n) ) = \lim_{n\rightarrow +\infty} \frac{ \mathrm{m}(\nabla u(x_0)\cdot x;Q_{T_n}(x_n) )}{T_n}= f(\nabla u(x_0))=\Xi |\nabla u(x_0)|^2
 \]
by Proposition \ref{prop:ex}.
Hence,
 	\begin{align*}
 	\frac{\d \mu}{\d \L^2}(x_0)\geq -C \delta |\nabla u(x_0)|^2 +\Xi |\nabla u(x_0)|^2.
 	\end{align*}
 By considering the limit as $\delta\rightarrow 0 $ we get
 	\[
 	\frac{\d \mu}{\d \L^2}(x_0)\geq \Xi |\nabla u(x_0)|^2.
 	\]
	as desired.
 \qed
\subsection{Proof of the upper bound} \label{sbsct:limsup}

We prove the statement in several steps, in order to clarify the diagonalization process that we use.  The strategy will be to approximate a generic function $u\in W^{1,2}(Q)$ with a sequence of functions $\{v_k\}_{k\in \N}$ which are  piecewise affine on simplexes and then show how to recover the energy of each $v_k$.  Then we will exploit a diagonalization procedure (by means of Lemma \ref{lem:diagonal}).  The recovery sequence for piecewise-affine maps is constructed in Step two below. The major technical point consists in handling the interaction at the common boundary between two simplexes $S,S'$.   To deal with this issue we build, for a generic simplex $S$ and for an affine function $v$ on $S$, an almost recovery sequence, which agrees with $v$ in an internal neighborhood of $\partial S$. This is done in the Step one below.

\smallskip
\textbf{Step one:} \textit{We prove that for any triangle $S\subseteq Q$, for $u(x)=\xi \cdot x$ on $S\subseteq Q$ and for any fixed $\delta>0$ there exists a sequence $\{u_{\e_n,\delta}\in L^2(S;\eta_n)\}_{n\in \N}$ such that 
	\begin{align}
	u_{n,\delta} \rightarrow \xi\cdot x & \ \ \text{in the sense of Definition \ref{def:conv}}\label{eqn:convonSimplex}\\
	u_{n,\delta}=&\ \xi\cdot x \ \ \ \text{on $(\partial S)_{\delta}\cap S$}.\label{eqn:localityonSimplex}\\
	\lim_{n\rightarrow +\infty} \F_{\e_n}(u_{\e_{n,\delta}};S\setminus (\partial S)_{\lambda\e_n} )\leq &\ \Xi |S| |\xi|^2+C\delta \label{eqn:energyonSimplex}
	\end{align}
with $C$ depending on $S$ and $|\xi|$ only.  }
\begin{figure}
\centering
\includegraphics[scale=0.5]{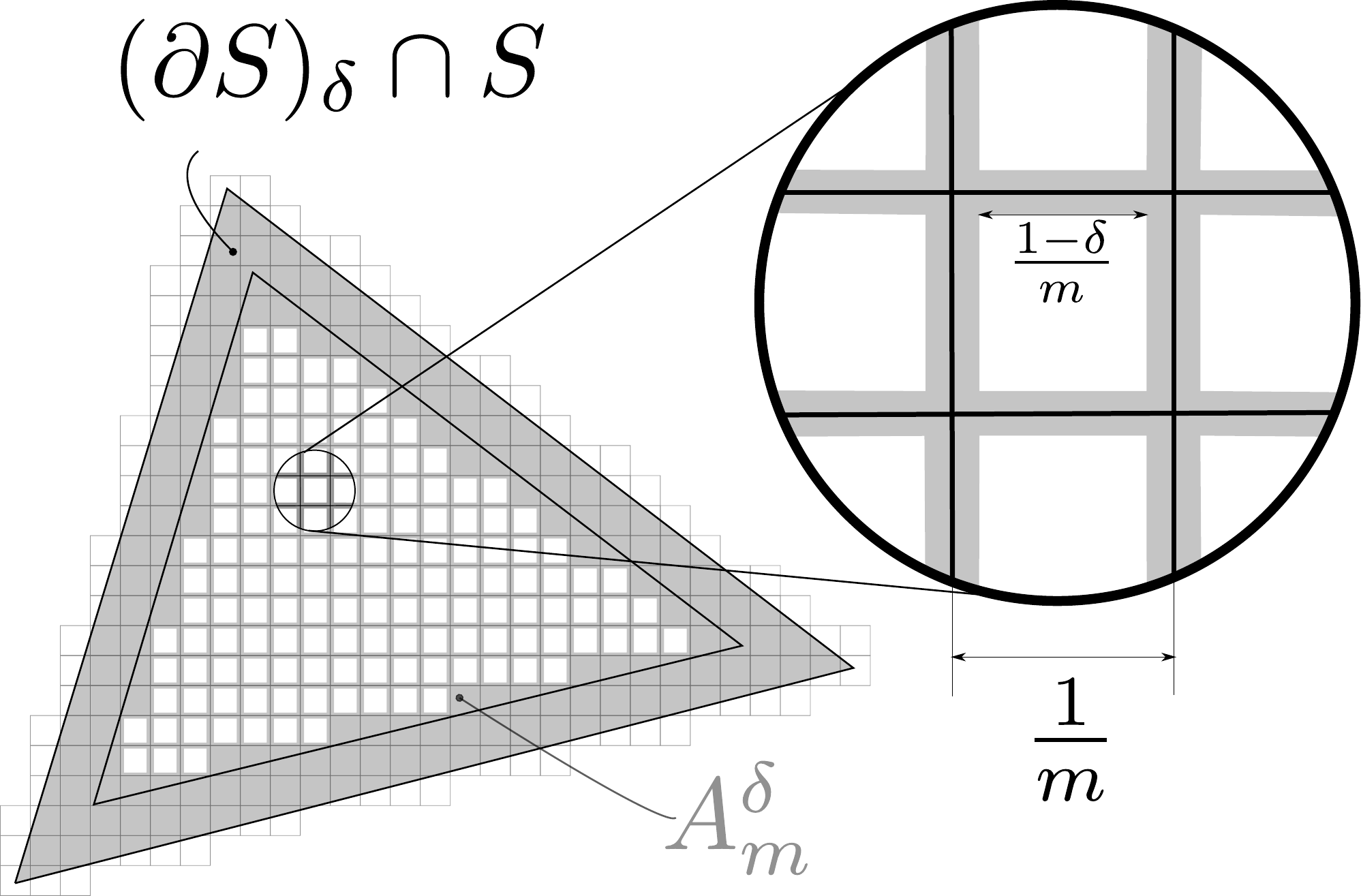}
\caption{construction of recovery sequences on a triangular domain.} \label{LimsupS.pdf}
\end{figure}

The construction is illustrated in Fig.~\ref{LimsupS.pdf}: we fix $\delta>0$, $m\in \N$ and we pick a grid of squares of size $1/m$.   For a square $Q_{1/m}(J)$ that intersect in a non trivial way $(\partial S)_{\delta}\cap S$ we define $u_{\e,m,\delta}=\xi\cdot x$ on $Q_{1/m}(J)\cap S$. If instead $Q_{1/m}(J)$ is well contained in $S$ we consider $u_{\e,m,\delta}$ to agree with the quasi minimum $v_{\e}$ of the cell problem $\mathrm{m}$ on a the  subsquare $Q_{(1-\delta)/m}(J)$ and $u_{\e,m,\delta}=\xi \cdot x$ on $Q_{1/m}(J)\setminus Q_{(1-\delta)/m}(J)$.  In this way $u_{\e,m,\delta}=\xi \cdot x$ on $A_{m}^{\delta}$ and this construction ensures convergence and the $\limsup$ upper bound for the function $u(x)=\xi\cdot x$ on $S$.

We now formalize this argument. Fix $\delta>0$,  $m\in \N$ and introduce the sub-class of indexes 
	\begin{align*}
	\mathcal{I}_{1,m}(S)&:=\{J\in (\sfrac{1}{m}\Z^2)\cap Q \ | \ Q_{\sfrac{1}{m}}(J)\subset S, \  Q_{\sfrac{1}{m} }(J)\cap (\partial S)_{\delta}=\emptyset \} \\
		\mathcal{I}_{2,m}(S)&:=\{J\in (\sfrac{1}{m}\Z^2)\cap Q \ | \ Q_{\sfrac{1}{m}}(J)\cap (\partial S)_{\delta} \cap S\neq \emptyset\}.
	\end{align*} 
For $J\in \mathcal{I}_{1,m}(S)$ consider $u_{\e}^J\in L^2( \sfrac{1}{\e}Q_{\sfrac{(1-\delta)}{m}}(J); \eta)$ such that
	\begin{align*}
	u_{\e,m,\delta}^J(x) &=\xi\cdot x \  \  \ \text{for all $x\in \eta\cap \sfrac{1}{\e} Q_{\sfrac{(1-\delta)}{m}}(J) $ : $\mathrm{dist}(x,\partial  (\sfrac{1}{\e}  Q_{\sfrac{(1-\delta)}{m}}(J) ) \leq 2\lambda $}  
\end{align*}
and
\begin{align*}
	\sum_{x\in \eta \cap \sfrac{1}{\e} Q_{\sfrac{(1-\delta)}{m}}(J)} \sum_{y\in \eta\cap B_{\lambda}(x)} |u_{\e,m,\delta}^J(x)-u_{\e,m,\delta}^J(y)|^2&\leq \mathrm{m}(\xi;\sfrac{1}{\e}  Q_{\sfrac{(1-\delta)}{m}}(J))+1.
	\end{align*}
For $J\in \mathcal{I}_{2,m}(S)$ consider just 
	\[
	u_{\e,m,\delta}^J(x) :=\xi\cdot x  \ \ \ \text{for all $x\in \eta \cap (\sfrac{1}{\e}Q_{\sfrac{1}{m}}(J) \cap \sfrac{1}{\e}S)$}.
	\]
	In particular, setting 
\begin{equation}
		v_{\e,m,\delta}^J(x):=\left\{ \begin{array}{ll }
		\e u_{\e,m,\delta}^J(x/\e)  & \text{for $x\in \eta_{\e} \cap Q_{\sfrac{(1-\delta)}{m}}(J) \cap S$;}\\
		\xi \cdot x &  \text{for $x\in \eta_{\e} \cap [ Q_{\sfrac{1}{m}}(J)\setminus Q_{\sfrac{(1-\delta)}{m}}(J) ]  \cap S$;}
				\end{array}
				\right.
\end{equation}
	we see that, for $J\in \mathcal{I}_{1,m}(S)$ we have
		\begin{align*}
		v_{\e,m,\delta}^J(x)=&\xi\cdot x \ \ \ \text{for all $x\in \eta_{\e}\cap Q_{\sfrac{1}{m}}(J) $ : $\mathrm{dist}(x,\partial Q_{\sfrac{1}{m}}(J))\leq 2\lambda\e_n+\frac{\delta}{m}$}
		\end{align*}
and, by applying Corollary \ref{cor:est},
\begin{align*}
		\sum_{z \in \eta_{\e} \cap Q_{\sfrac{1}{m}}(J)} \sum_{y \in \eta_{\e} \cap  B_{\lambda \e }(z)} |v_{\e,m,\delta}^J(z)-v_{\e,m,\delta}^J\left(y \right)|^2&\leq  \e^2\mathrm{m} (\xi;\sfrac{1}{\e} Q_{\sfrac{(1-\delta)}{m}}(J))+\e^2 + C |\xi|^2 \frac{\delta}{m^2},
		\end{align*}
	while $v_{\e,m,\delta}^J(x)=\xi \cdot x $ for all $x\in \eta_{\e}\cap Q_{\sfrac{1}{m}}(J) \cap S$  if $J\in \mathcal{I}_{2,m}(S)$.
Hence, we define
	\[
	v_{\e,m,\delta}(x):=\sum_{J\in \mathcal{I}_{1,m}(S)\cup  \mathcal{I}_{2,m}(S)} \ca_{Q_{\sfrac{1}{m}}(J) \cap S} (x) v_{\e,m,\delta}^J(x).
	\]
In this way,  by applying again Corollary \ref{cor:est} 
	\begin{align*}
	\F_{\e}(v_{\e,m,\delta};S\setminus (\partial S)_{\lambda\e_n} )\leq & \sum_{J\in \mathcal{I}_{1,m}(S)} \F_{\e}(v_{\e,m,\delta};Q_{\sfrac{1}{m}}(J))+\sum_{J\in \mathcal{I}_{2,m}(S)} \F_{\e}(v_{\e,m,\delta};Q_{\sfrac{1}{m}}(J)\setminus (\partial S)_{\lambda\e_n} )\\
	\leq & C \sum_{J\in \mathcal{I}_{2,m}(S)} |\xi|^2 |Q_{\sfrac{1}{m}}(J)\cap S| +  \sum_{J\in \mathcal{I}_{1,m}(S)} \F_{\e}(v_{\e,m,\delta};Q_{\sfrac{1}{m}}(J)) 	\\
		\leq &  C P(S)  |\xi|^2\delta    +  \sum_{J\in \mathcal{I}_{1,m}(S)} \F_{\e}(v_{\e,m,\delta};Q_{\sfrac{1}{m}}(J))\\
			\leq & C P(S)  |\xi|^2\delta +  \sum_{J\in \mathcal{I}_{1,m}(S)}  \left(\e^2\mathrm{m} (\xi;\sfrac{1}{\e} Q_{\sfrac{(1-\delta)}{m}}(J))+\e^2 + C |\xi|^2 \frac{\delta}{m^2}\right)
	\end{align*}
The existence of the limit of $\mathrm{m}(\cdot ;Q_T)$ given by Lemma \ref{lem:lim} yields the existence of $\e_0$ such that,  for all $\e<\e_0$ we have also
	\[
	\e^2\mathrm{m} (\xi;\sfrac{1}{\e} Q_{\sfrac{(1-\delta)}{m}}(J)  \leq \	\Xi |\xi|^2 |Q_{\sfrac{(1-\delta)}{m}}(J)| + \frac{1}{m^3} \ \ \ \ \text{for all $J\in \mathcal{I}_{1,m}(S)$}.
	\]
Therefore
	\begin{align}
	\F_{\e}(v_{\e,m,\delta};S\setminus (\partial S)_{\lambda\e_n} )&\ \leq    C P(S)  |\xi|^2\delta   +  \sum_{J\in \mathcal{I}_{1,m}(S)}  \left(\e^2\mathrm{m} (\xi;\sfrac{1}{\e} Q_{\sfrac{(1-\delta)}{m}}(J))+\e^2 + C |\xi|^2 \frac{\delta}{m^2}\right) \nonumber\\
	&\ \leq   \Xi |\xi|^2  \sum_{J\in \mathcal{I}_{1,m}(S)}  |Q_{\sfrac{(1-\delta)}{m}}(J)|   +  \left(\frac{1}{m^3} +\e^2\right) \#(\mathcal{I}_{1,m}(S) ) + C|\xi|^2 \delta  \nonumber\\
	&\ \leq \Xi \int_S |\xi|^2 + \left(\frac{1}{m} +\e^2m^2\right)|S| +  C|\xi|^2 \delta.\label{eqn:finalEnergy}
	\end{align}
Moreover
	\[
	\sup_{\sfrac{1}{m}>\e>0}\{\F_{\e}(v_{\e,m,\delta};S)\}<+\infty.
	\]
Thanks to the compactness Theorem \ref{thm:Compactness} we can thus conclude that $v_{\e_n,m,\delta}\rightarrow u_{m,\delta}$ in the sense of Definition \ref{def:conv}.  
Observe also that $v_{\e,m,\delta}(x)=\xi\cdot x $ for all $x\in A^{\delta}_m$,
where
\[
A^{\delta}_{m}:=  \bigcup_{J\in \mathcal{I}_{1,m}(Q)} \left\{ x \in S \ : \mathrm{dist}(x,\partial Q_{\sfrac{1}{m}}(J) ) \leq\sfrac{\delta}{m} \right\}  \cup \Bigl(\bigcup_{J\in \mathcal{I}_{2,m}(S)} Q_{\sfrac{1}{m}}(J)\cap S\Bigr).
\] 
This in particular implies that $u_{m,\delta}=\xi\cdot x$ on $A_m^{\delta}$.  Since $u_{m,\delta}\in W^{1,2}(S)$ and
	\[
	\Xi \int_{S}|\nabla u_{m,\delta}|^2 \d x\leq \liminf_{n\rightarrow +\infty} \F_{\e_n}(v_{\e_n,m,\delta};S)<\Xi|S|\xi|^2+ C\left( \frac{1}{m}+ \delta\right)
	\]
we also have that, up to sub-sequences $u_{{m,\delta}}\rightarrow \xi \cdot x$ in $L^2(S)$.  Hence,  by applying Lemma \ref{lem:diagonal} we can find $m_{\e}\rightarrow +\infty$, with  $\e\, m_{\e}\rightarrow 0$ and for which $u_{\e,\delta}:=v_{\e,m_{\e},\delta}\rightarrow \xi\cdot x$ and also \eqref{eqn:energyonSimplex} and \eqref{eqn:localityonSimplex} hold.
\smallskip

\textbf{Step two:} \textit{we prove the existence of a recovery sequence for the piecewise-affine function
	\[
	v=\sum_{S\in \mathcal{S} } v_S(x)
	\]
with $\mathcal{S}$ a finite family of essentially disjoint triangles partitioning $Q$ and $v_S(x)=v_{S'}(x)$ for $x\in \partial S\cap \partial S'$.} Note that, being each $v_S$ affine it can be represented as $v_S(x)=\xi_S\cdot x +b_S$ on $S$.  Fix $\delta>0$ and for any $S\in \mathcal{S}$ let $\{u_{n,\delta}^S\in L^2(S;\eta_n) \}_{n\in \N}$ be the functions constructed in Step one and satisfying \eqref{eqn:energyonSimplex},\eqref{eqn:localityonSimplex}, and \eqref{eqn:convonSimplex} relatively to $\xi_S$.  Set
	\[
	v_{\e,\delta}(x):= \sum_{S\in \mathcal{S}} \ca_{S}(x) (u_{n,\delta}^S(x)+b_S).
	\] 
We have that $v_{\e,\delta}\rightarrow v$ in the sense of Definition \ref{def:conv}.  Moreover,
\begin{align*}
\F_{\e}(v_{\e,\delta};Q)&\leq \sum_{S\in \mathcal{S}} \F_{\e}(v_{\e,\delta};S\setminus (S)_{\lambda\e_n} )+  \sum_{S\in \mathcal{S}} \F_{\e}(v_{\e,\delta}; (S)_{\lambda\e_n} ).
\end{align*}
Then we just observe that 
\begin{align*}
\sum_{S\in \mathcal{S}} \F_{\e}(v_{\e,\delta}; (S)_{\lambda\e_n} )=\sum_{S\in \mathcal{S}}  \sum_{x\in \eta_n \cap (S)_{\lambda\e_n}  }  \sum_{y\in \eta_n\cap B_{\lambda\e_n}(x)} |v_{\e,\delta}(x)-v_{\e,\delta}(y)|^2. 
\end{align*}
Recall that,  by construction we have that, for $x\in S$,  $v_{\e,\delta}(x)=v(x)$ on $ (S)_{\lambda\e_n}$ and, since it is continuous and piecewise affine, it is in particular a Lipschitz map. Then we have

\begin{align*}
\sum_{S\in \mathcal{S}}  \F_{\e}(v_{\e,\delta}; (S)_{\lambda\e_n} )&=\sum_{S\in \mathcal{S}}  \sum_{x\in \eta_n \cap (S)_{\lambda\e_n}  }  \sum_{y\in \eta_n\cap B_{\lambda\e_n}(x)} |v_{\e,\delta}(x)-v_{\e,\delta}(y)|^2\\
&\leq C\sum_{S\in \mathcal{S}}  \sum_{x\in \eta_n \cap (S)_{\lambda\e_n}  }  \sum_{y\in \eta_n\cap B_{\lambda\e_n}(x)} |x-y|^2\\
&\leq C\sum_{S\in \mathcal{S}}  \sum_{x\in \eta_n \cap (S)_{\delta}  } \eta_n ( B_{\lambda\e_n}(x)) \e_n^2\\
&\leq C\sum_{S\in \mathcal{S}}  P(S) \delta ,
\end{align*}
where in the last inequality we have used Proposition \ref{propo:count}. Thus
	\[
	\limsup_{\e\rightarrow 0} \F_{\e} (v_{\e,\delta}(x);Q)\leq \sum_{S\in \mathcal{S}}\Xi|S| |\xi|^2 + C\delta= \Xi \int_Q |\nabla v|^2 \d x+C\delta.
	\]
By now diagonalizing along $\delta$, with the aid of Lemma \ref{lem:diagonal},  we find $u_{\e}\rightarrow v$ such that 
\[
\limsup_{\e\rightarrow 0} \F_{\e} (v_{\e,\delta}(x);Q)\leq \Xi \int_Q |\nabla v|^2 \d x.
\]

\textbf{Step three:} \textit{we prove the existence of a recovery sequence for a generic $u\in W^{1,2}(Q)$}.  We just observe that for any $u\in W^{1,2}(Q)$ we can find a sequence of piecewise-affine functions $\{v_k\}_{k\in \N}$  with the structure as in Step two such that $v_k\rightarrow u$ in $L^2$ and 
\[
\int_{Q}|\nabla v_k|^2\d x\rightarrow \int_{Q} |\nabla u|^2\d x.
\]
The construction developed in Step two, for each $k$, and a further application of the diagonalizing procedure (Lemma \ref{lem:diagonal}) conclude the existence of the desired sequence.
\qed

\section{Appendix}\label{sct:app}

\subsection{Existence of regular grids}
We here focus on proving Theorem \ref{thm:pathConn}.  Let 
	\begin{align*}
	R^v_{T,\delta}(x_0)&:=\left[x_0 -\frac{T\delta}{2},x_0 +\frac{T\delta}{2}\right]\times\left [x_0 -\frac{T}{2}, x_0 +\frac{T}{2}\right]\\
	R^h_{T,\delta}(x_0)&:=\left [x_0 -\frac{T}{2}, x_0 +\frac{T}{2}\right]\times \left[x_0 -\frac{T\delta}{2},x_0 +\frac{T\delta}{2}\right]
	\end{align*}
\begin{definition}
Let $\{X_j\}_{j\in \Z^2}$ be a sequence of i.i.d random variable such that
	\begin{equation}
	X_j=\left\{ \begin{array}{lr}
	1 \ & \ \text{with probability $p$}\\
	0\ & \ \text{with probability $1-p$}\\
	\end{array}\right.
	\end{equation}
We say that $\{j_i\}_{i=1}^M$ is an \textit{open path} for the realization $\omega$ if $X_{j_i}(\omega)=1$ for all $i=1,\ldots, M$ and $j_{i}$, $j_{i+1}$ are neighboring squares. 
\end{definition}
We recall the following percolation property from \cite{kesten1982percolation}.
\begin{theorem}[Property of Bernoulli site percolation]\label{thm:perc}
There exists a probability $p_{\mathrm{cr}}$ such that the following holds.  Let $C$ be a compact set, for any $\delta>0$ there exists $K_{\delta}$ such that for almost all $\omega\in \Omega$ we can find $T_0(\omega)>0$  for which any rectangles $R^{v}_{T,\delta}(x_0), R^{h}_{T,\delta}(x_0)$ with $T>T_0$ and $x_0\in T C$ contains at least $K_{\delta}T$ disjoint paths that connects the two opposite sides of $R^v_{T,\delta}(x_0), R^h_{T,\delta}(x_0)$ respectively in the horizontal and in the vertical direction.
\end{theorem}
For the sake of brevity we introduce the following notation confined to this section. For $R^h_{T,\delta}(x_0),R^v_{T,\delta}(x_0)$ we denote by $\mathbf{h}^T_1(x_0),\ldots ,\mathbf{h}^T_M(x_0)$ and by $\mathbf{v}^T_1(x_0),\ldots ,\mathbf{v}^T_M(x_0)$, the families of horizontal (and respectively vertical) disjoint paths connecting the two opposite sides of $R^h_{T,\delta}(x_0)$ (and $R^v_{T,\delta}(x_0)$ respectively).
\begin{proposition}\label{prop:Ex}
Let $C\subset \R^2$ be a compact set.  There exists $K_{\delta}, \Upsilon_{\delta},\alpha_0,\lambda_0$ such that, provided 
\[
\alpha<\alpha_0 \ ,\lambda>\max\Bigl\{\frac{2}{\alpha}, \ \lambda_0\Bigr\},
\] for almost all $\omega\in \Omega$ we can find $T_0(\omega,C)>0$  for which any rectangles $R^{v}_{T,\delta}(x_0), R^{h}_{T,\delta}(x_0)$ with $T>T_0$ and $x_0\in T C $ contains at least $K_{\delta}T$ disjoint paths $\mathbf{h}^T_1(x_0),\ldots ,\mathbf{h}^T_{K_{\delta}T}(x_0)$, satisfying the following properties.
\begin{itemize}
\item[a.1)] $\mathbf{h}^T_i(x_0), \mathbf{v}^T_j(x_0)$ are paths in $\eta^{\alpha}(\lambda)\cap Q_T$ for all $i,j$;
	\item[b.1)] for any $m=1,\ldots, K_{\delta} T$, $\mathbf{h}^T_{m}$ connects the two opposite sides of $R^h_{T,\delta}(x_0)$ and is strictly contained in $R^h_{T,\delta}(x_0)$; 
		\item[c.1)] for any $m=1,\ldots, K_{\delta}T$, $\mathbf{v}^T_{m}$ connects the two opposite sides of $R^v_{T,\delta}(x_0)$ and is strictly contained in $R^v_{T,\delta}(x_0)$; 
	\item[d.1)] the following bounds hold for any $m \in\{1,\ldots,K_{\delta}T\}$;
	\begin{equation}\label{eqn:unifBndAPP}
	\begin{array}{c}
	\displaystyle\frac{T}{\Upsilon_{\delta}}\leq \ell(\mathbf{h}^T_m\cap R^h_{T,\delta}(x_0))\leq \Upsilon_{\delta} T \ \ \ 
	\displaystyle	\frac{T}{\Upsilon_{\delta}}\leq \ell(\mathbf{v}^T_m\cap R^v_{T,\delta}(x_0))\leq\Upsilon_{\delta}T
	\end{array}
	\end{equation}	
	\item[e.1)] $\mathrm{dist}(\mathbf{h}^T_i(x_0), \mathbf{h}^T_j(x_0))\geq 3\lambda$, $\mathrm{dist}(\mathbf{v}^T_i(x_0), \mathbf{v}^T_j(x_0))\geq 3\lambda$ for all $i,j= 1,\ldots,K_{\delta}T$,  $i\neq j$;
	\item[f.1)]  If $x \in (\mathbf{h}_j^T(x_0))_{3\lambda}\cap \eta$ then $\eta(B_{\lambda}(x))\leq \alpha^{-1} \lambda^2$;
	\item[g.1)] If $x,y\in \mathbf{h}_i^T(x_0)$, ($x,y\in \mathbf{h}_j^T(x_0)$) satisfies $\langle x,y\rangle$ then $|x-y|\leq \lambda$;
%
\end{itemize}
\end{proposition}
\begin{proof}
Fix $\Lambda\in \N $ and consider the division of $\R^2$ in the grid
\begin{align*}
	I&:=\{i\in \Lambda \lambda\Z^2 \ | \ (Q_{\Lambda \lambda}(i) \cap Q_T \neq \emptyset\}\\
	\mathcal{Q}&:=\{Q_{\Lambda \lambda}(i)\ | \ i\in I\}
	\end{align*} 
and, for any $Q_{\lambda\Lambda}(i)\in \mathcal{Q}$ consider the refinement
\begin{align*}
	J(i)&:=\{j\in \lambda\Z^2 \ | \ (Q_\lambda(j) \cap Q_{\Lambda\lambda}(i) \neq \emptyset\}\\
	\mathcal{Q}'(i)&:=\{Q_\lambda(j) \ | \ j\in J(i)\}\\
	N'(i)&:=\#(J'(i))=\Lambda^2.
	\end{align*} 
For any $i\in I$ we introduce the following events $\Omega^{\alpha,\lambda}_i$ of all the realizations $\omega$ with the following properties
	\begin{itemize}
	\item[I)] $1 \leq \eta(Q_\lambda(j) )\leq \frac{\alpha^{-1}}{8} \lambda^2 $ for all $j\in J(i)$;
	\item[II)] $\mathrm{dist}(x,y)\geq 2\alpha$ for all $x,y\in Q_{\Lambda \lambda}(i)\cap \eta$;
	\item[III)] $\mathrm{dist}(x,\partial Q_{\Lambda\lambda}(i))\geq 2\alpha$ for all $x\in Q_{\Lambda \lambda}(i)\cap \eta$;
	\end{itemize}
We also define
\begin{equation*}	\xi^{\alpha,\lambda}_i(\omega):=\ca_{\Omega^{\alpha,\lambda}_i}(\omega).
\end{equation*}		
Observe that the probability of the set of realizations satisfying properties (II) and (III) tends to $1$ if $\alpha\rightarrow 0$ (we refer to the same argument as in the proof of \cite[Lemma 4.1]{braides2020homogenization}). Moreover 
	\[
	\mathbb{P}\left(1\leq \eta(Q_\lambda(j) )\leq \alpha^{-1} \lambda^2\right)=e^{-\lambda^2}\sum_{m=1}^{\lfloor \frac{\alpha^{-1}}{8} \lambda^2\rfloor}\frac{\lambda^{2m}}{m!} =:p_\lambda(\alpha).
	\]
In particular, setting
	\begin{align*}
	A^{\alpha,\lambda}_i&:=\left\{\omega \in \Omega \ | \   1\leq \eta(Q_\lambda(j) )\leq \frac{\alpha^{-1}}{8} \lambda^2 \ \ \text{for all $j\in I(i)$}\right\}
	\end{align*}
we have that
$
	\mathbb{P}(A^{\alpha,\lambda}_i)= p_\lambda(\alpha)^{\Lambda^2}
$.

If $\alpha\rightarrow +\infty$ we have $p_\lambda(\alpha)\rightarrow 1-e^{-\lambda^2}$.  In particular, for any choice of $\Lambda>1$, $\gamma> 0$ we can find $\alpha_0(\Lambda), \lambda_0(\Lambda)>0$ such that 
	\[
	\mathbb{P}(\xi_{i}^{\alpha,\lambda}=1)=p(\alpha,\lambda)>1-\gamma \ \ \ \text{for all $\alpha<\alpha_0$,  $\lambda>\lambda_0 $} 
	\]
Thus for a suitably small $\alpha$ and big $\lambda$ we can invoke Theorem \ref{thm:perc} and find a $K_{\delta}$ independent of the realization and $T_0=T_0(\omega,C)$ such that, for any $R_{T,\delta}^v(x_0),R_{T,\delta}^h(x_0)$ (note that this is uniform as $x_0/T \in C$) contains at least $K_{\delta}T$ disjoint paths (connecting the two opposite sides of $ R_{T,\delta}^v(x_0),R_{T,\delta}^h(x_0) $ vertically and horizontally respectively) of neighboring squares from $\mathcal{Q}$. We now define the Voronoi paths as follows. Consider, for instance $\{i_j\}_{j=1}^{N}$ to be the first (from the bottom) horizontal path in $R^h_{T,\delta}(x_0)$ and let $s_{1},\ldots,s_{N}$ be the segment joining the centers of neighboring squares (for instance ($s_m$ joins the centers of $Q_{\Lambda\lambda}(i_m),Q_{\Lambda\lambda}(i_{m+1})$. We set
	\[
	\mathbf{h}_1^T(x_0):=\{x \in \eta : \ (C(x;\eta)\cap s_{j_m})\neq \emptyset \ \text{for some $m=1,\ldots,N$} \}.
	\] 
We define all the other horizontal paths accordingly, as well as the vertical paths. We refer to $\mathbf{h}_1^T(x_0)$ without loss of generality in proving the properties.  
By arguing as in the proof of Lemma \cite [Lemma 4.1]{braides2020homogenization} we can derive also that
	\[
	\mathrm{diam}(C(x;\eta))\leq \frac{1}{\alpha}, \ \ \ I(C(x;\eta))>\alpha.
	\] 
Moreover 
By choosing $\lambda>\max\left\{\sfrac{2}{\alpha},\lambda_0\right\}$ we can guarantee additionally that 
	\[
	\mathbf{h}^T_1(x_0)\subset \bigcup_{j=1}^{N} \bigcup_{\substack{m\in J(i_j):\\ Q_{\lambda}(m)\cap s_j\neq \emptyset}} Q_{\lambda}(m)
	\]
In particular if $x\in \mathbf{h}^T_1(x_0)$ then $x\in Q_{\lambda}(m)\subset \subset Q_{\Lambda \lambda}(i_j)$ for some $m\in J(i_j)$ intersecting $s_{j}$. In particular, $B_{\lambda}(x)$ is contained in the union of the eight squares whose boundary intersects in a non trivial way the boundary of $Q_{\lambda}(m)$ and this union, call it $O$, is still contained in $Q_{\Lambda\lambda}(i_j)$ and made by at most $8$ of such squares. Therefore
	\[
	\eta(B_{\lambda}(x))\leq  \eta(O)\leq  \alpha^{-1} \lambda^2.
	\]
This implies that $\mathbf{h}^T_1(x_0)$ is a path on $\eta^{\alpha}(\lambda)$ and we get property (a.1).  Properties (b.1) and  (c.1) are immediate from Bernoulli site percolation.  Also property (e.1) is a consequence of the fact that any path is contained in a square around the segment joining the centers and thus the distance between two paths is at least $3\lambda$.  By the same principle,  if $\Lambda$ was chosen big enough from the very beginning (say bigger than $10$) whenever $x\in (\mathbf{h}^T_1(x_0))_{3\lambda}$ then it belongs to a square of size $\lambda$ contained in $Q_{\Lambda\lambda}(i_j)$ and the same estimate applies on $\eta(B_{\lambda}(x))$, achieving property (f.1).  If instead $x,y$ are neighboring points in a path, then $|x-y|\leq \frac{2}{\alpha}<\lambda$ yielding property (g.1).  \smallskip

It remains to show property (d).  Due to the fact that $\mathrm{diam}(C(x;\eta))\geq \alpha^{-1}$ we have 
	\begin{align*}
	\ell(\mathbf{h}^T_m\cap R_{T,\delta}^h) \geq \alpha T
	\end{align*}
	 for any $m=1,\ldots, K_{\delta}T$.
Fix $L>0$ and consider 
	\[
	I(L):=\{m=1,\ldots K_{\delta} T \ | \ \ell(\mathbf{h}^T_m\cap R_{T,\delta}^h) > L T \}.
	\]
Then,  since the paths are disjoint, we have
$$
	\bigcup_{m\in I(L)} \mathcal{V}(\mathbf{h}^T_m) \subset R_{T,\delta}^h(x_0),\qquad
	\Bigl|	\bigcup_{m\in I(L)} \mathcal{V}(\mathbf{h}^T_m) \Bigr| \leq T^2\delta,\qquad
	\#(I(L)) \leq \frac{T\delta}{L\alpha^2 \pi}.
$$
If we now choose $L=L_{\delta}$ large enough we can ensure that
	\[
	K_{\delta} T -  T \frac{\delta}{L\alpha^2 \pi}=K'_{\delta} T
	\]
with $K_{\delta}'<K_{\delta}$.  Then,  up to discarding the paths labeled by $I(L)$ (which do not affect properties (a)--(g) ) we can reduce ourselves to $K_{\delta}'T $ paths satisfying also property (d) with 
$
	\Upsilon_{\delta}=\max\{ L_{\delta}, \alpha^{-1} \}
$.
\end{proof}
Observe that the above proposition is not sufficient to conclude the validity of Theorem \ref{thm:pathConn} yet,  since a straight application of Proposition \ref{prop:Ex} on $R_{T,t}^h(x)$ (that would be required to get the horizontal and vertical paths connecting the opposite side of $Q$) would yield the constants given by property (d) of  Proposition \ref{prop:Ex} dependent of $t$.  We instead require a uniform geometry. Moreover we need to localize the estimate given by property (d) of Proposition \ref{prop:Ex}  to each square $Q_{i,j}^t$. Therefore an additional construction is required.  \smallskip

The following corollary comes as an application of Proposition \ref{prop:Ex}.
\begin{corollary}\label{cor:onethird}
There exists $\Upsilon,\alpha_0,\lambda_0$ such that,  for any fixed $t>0$ and provided $\alpha,\lambda$ satisfies
\[
\alpha<\alpha_0 \ ,\lambda>\max\Bigl\{\frac{2}{\alpha}, \ \lambda_0\Bigr\},
\] 
then for almost all $\omega\in \Omega$ we can find $\e_0(\omega,t)>0$ for which,  if $\e<\e_0$ any rectangles $R^{h}_{t,1/2}(x), R^{v}_{t,1/2}(x)$ with $x\in Q_{t k_t} $ contains disjoint horizontal paths $\mathbf{h}^x_1,\ldots ,\mathbf{h}^x_{M_{\e,t}}$, (and respectively vertical paths $\mathbf{v}^x_1,\ldots ,\mathbf{v}^x_{M_{\e,t}}$) satisfying the following properties.
\begin{itemize}
	\item[a.2)] all the paths are in $\e \eta^{\alpha}(\lambda)$;
	\item[b.2)] for any $m=1,\ldots, M_{\e,t}$, $\mathbf{h}^x_{m}$ connects the two opposite sides of $R^{h}_{t,1/2}(x)$ of size $t$ and is strictly contained in $R^{h}_{t,1/2}(x)$; 
		\item[c.2)] for any $m=1,\ldots, M_{\e,t}$, $\mathbf{v}^x_{m}$ connects the two opposite sides of $R^{v}_{t,1/2}(x)$ of size $t$ and is strictly contained in $R^{v}_{t,1/2}(x)$; 
	\item[d.2)] the following bounds hold 
	\begin{equation}\label{eqn:unifBndas}
\frac{t}{\Upsilon  \e}\leq \ell(\mathbf{h}^x_m\cap R^{h}_{t,1/2}(x), )\leq \frac{\Upsilon t }{\e}, \qquad
\frac{t}{\Upsilon \e}\leq \ell(\mathbf{v}^x_m\cap R^{v}_{t,1/2}(x),)\leq \frac{\Upsilon t}{\e},\quad 	M_{\e,t}\geq \frac{t}{\Upsilon \e}
	\end{equation}
	
	\item[e.2)] $\mathrm{dist}(\mathbf{h}^x_m,\mathbf{h}^{x'}_s)\geq 3\lambda \e$ and $\mathrm{dist}(\mathbf{v}^j_m,\mathbf{v}^{j'}_s)\geq 3\lambda \e$ for all $i, i',j,  j'\in \{1,\ldots,k_t\}$, $m,s\in\{1,\ldots,M_{\e,t}\}$ (with $m\neq s$ for $i=i'$ or $j=j'$);
	\item[f.2)] If $y\in (\mathbf{h}^x_m)_{3\lambda\e}\cap \eta_{\e}$, $(y\in (\mathbf{v}^x_m)_{3\lambda\e}\cap \eta_{\e})$ it holds $\eta_{\e}(B_{\lambda\e}(y))\leq \frac{1}{\alpha} \lambda^2$;
	\item[g.2)] If $z,y\in \mathbf{h}^x_m$, ($z,y\in \mathbf{v}^x_m$) neighboring points then $|z-y|\leq \lambda\e$;
\end{itemize}
\end{corollary}
\begin{proof}
By invoking Proposition \ref{prop:Ex},  (applied with $\sfrac{t}{\e} $ in place of $T$) for almost all realizations we can find $K=K_{1/2}, \Upsilon'=\Upsilon_{1/2}$, $\alpha_0$, $\lambda_0$ such that, provided $\alpha,\lambda$ satisfies \eqref{eqn:unifBndAPP} then for almost all realizations $\omega$ we can find $\e_0=\e_0(\omega,t)$ for which any $R_{\sfrac{t}{\e},\sfrac{1}{2}}^h(x_0), R_{\sfrac{t}{\e},\sfrac{1}{2}}^v(x_0)$,  (provided $x_0\in Q_{k_t \sfrac{t}{\e}}$) contains at least $K\sfrac{t}{\e} $ disjoint paths satisfying properties (a)--(g) of Proposition \ref{prop:Ex}.  In particular by scaling back to $\eta_{\e}=\e\eta$ we have that,  for any $x\in Q_{t k_t}$, $\e<\e_0$ we can find $\mathbf{h}^x_1,\ldots, \mathbf{h}^x_{K \sfrac{t}{\e} }\in \e\eta^{\alpha}(\lambda)$ disjoint horizontal (and respectively vertical) paths contained in $ R_{t,\sfrac{1}{3}}^h(x)$ ($R_{t,\sfrac{1}{2}}^v(x)$) such that properties (a.2)--(g.2) of Corollary \ref{cor:onethird} are implied by Properties (a.1)--(g.1) of Proposition \ref{prop:Ex} by considering $\Upsilon=\max\{K^{-1},\Upsilon'\}$.
\end{proof}
We can now finally prove Theorem \ref{thm:pathConn}.  Let us briefly explain how we will proceed.  We will consider rectangles $R_{t,\sfrac{1}{2}}^h(x)$ whose edge proportion is fixed and for which Corollary \ref{cor:onethird} ensures the existence of the sought paths. These paths are not long enough to join the opposite sides of $R^h_i$, $R^v_j$ but the geometry of the paths is independent of $t,\e$ (since it depends only on the edge proportion $1/2$).  Thence we will perform a construction that will allow us to exploit such rectangles to build a grid on the whole square $Q_{tk_t}$ without affecting the constants and all the other properties.   We introduce some definitions in order to clarify the construction. We focus on vertical paths, since the construction will be performed in just one direction (the other following in the same way).

\begin{definition}\label{def:PJ}
Let $R$ be a rectangle and $\h,\v$ be a horizontal and vertical path in $R$ connecting opposite sides of $R$ in the respective directions.  Let $x\in \eta\cap R$.  We adopt the following terminology: 
\begin{itemize}
\item[-] We say that a point $x$ {\em lies below} $\h$ in $R$ (and we write $x\bigtriangledown_R \h$) if any path in $\eta\cap R$ that links $x$ to the \textit{top} of $R$ intersects $\h$;
\item[-] We say that a point $x$ {\em lies above} $\h$ in $R$ (and we write $x\bigtriangleup_R \h$) if any path in $\eta\cap R$ that links $x$ to the \textit{bottom} of $R$ intersects $\h$;
\item[-] We say that a point $x$ {\em lies on the left} of $\v$ in $R$ (and we write $x\lhd_R \v$) if any path in $\eta\cap R$ that links $x$ to the \textit{right} of $R$ intersects $\v$;
\item[-] We say that a point $x$ {\em lies on the right} of $\v$ in $R$ (and we write $x\rhd_R \v$) if any path in $\eta\cap R$ that links $x$ to the \textit{left} of $R$ intersects $\v$;
\end{itemize}
\end{definition}
\begin{proof}[Proof of Theorem \ref{thm:pathConn}]
 Let $\Upsilon, \alpha_0,\lambda_0$ be given by Corollary \ref{cor:onethird}. Fix $\alpha,\lambda$ satisfying \eqref{eqn:unifBnd},  let $t>0$.  Then for almost all realizations we can find $\e_0$ such that,  if $\e<\e_0$, the existence of paths in $R^h_{t,\sfrac{1}{2}}(x), R^v_{t,\sfrac{1}{2}}(x)$ is guaranteed for $x\in Q_{tk_t}$, according to Corollary \ref{cor:onethird}.  
\begin{figure}
\includegraphics[scale=1]{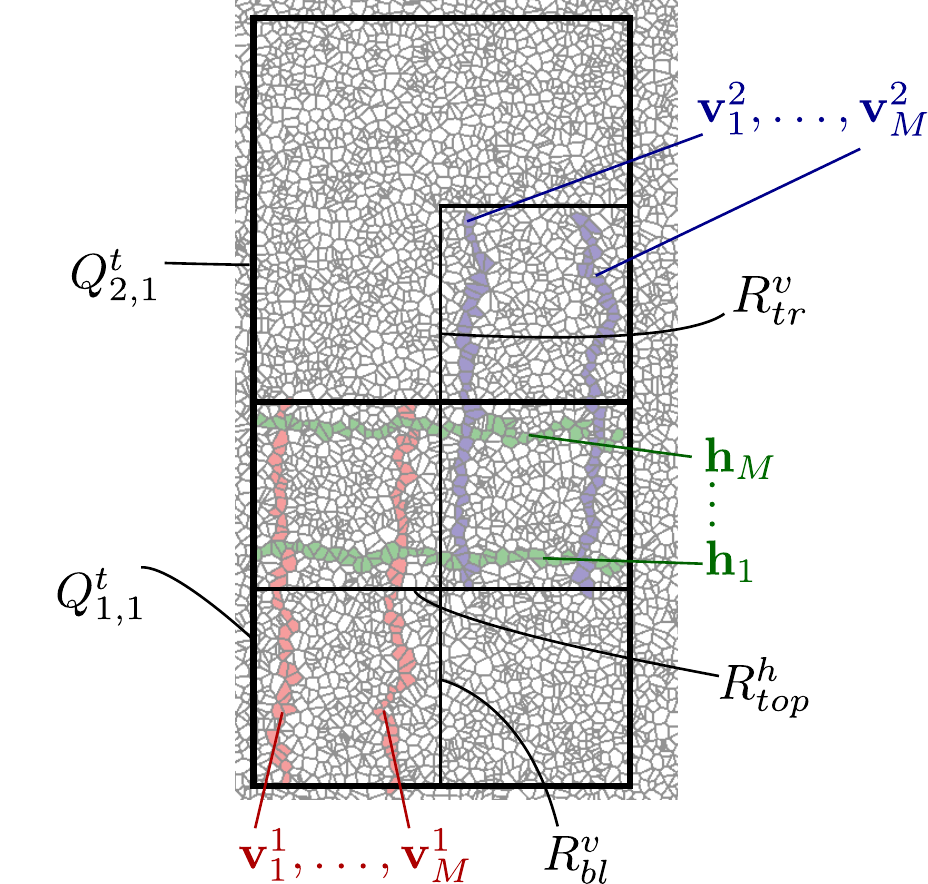}
\caption{The notation adopted in the proof of Theorem \ref{thm:pathConn}.  Here are depicted the sets $\mathcal{V}_{\e}(\v_j^i)$, and $\mathcal{V}_{\e}(\h_j)$ for $i=1,2$ $j=1,\ldots,M$.}
\label{fig:Construction}
\end{figure}
We set the notation with the aid of Figure \ref{fig:Construction}.  Consider, for instance,  the first square $Q^t_{1,1}$ on the bottom-left corner corner.  Consider the three rectangles depicted in Figure \ref{fig:Construction} with edges proportion of $1/2$,  namely the vertical one on the bottom left $R_{bl}^v\subset Q_{1,1}^t$,  the horizontal one on the top $R^h_{top}\subset Q_{1,1}^t$ and the vertical one on the top right $R_{tr}^v$ intersecting $R^h_{top}$ and the adjacent square $Q^t_{2,1}$.  Each of them contains at least $M_{\e,t}\geq \frac{\Upsilon t}{\e}$ paths satisfying properties (a.2)--(g.2) of Corollary \ref{cor:onethird} relatively to their own rectangle.  Now we label these paths from the $$
	\mathbf{h}_1,\ldots \mathbf{h}_{M} \in R^h_{top},\qquad
	\mathbf{v}^1_1,\ldots \mathbf{v}^1_{M}\in R_{bl}^v, \qquad
	\mathbf{v}^2_1,\ldots \mathbf{v}^2_{M} \in R^v_{tr},
$$
with the shorthand $M=M_{\e,t}$. Now the strategy is to join them and then refine the family in a way that all the properties (a)--(g) are ensured.  

%
%
\begin{figure}
\includegraphics[scale=0.8]{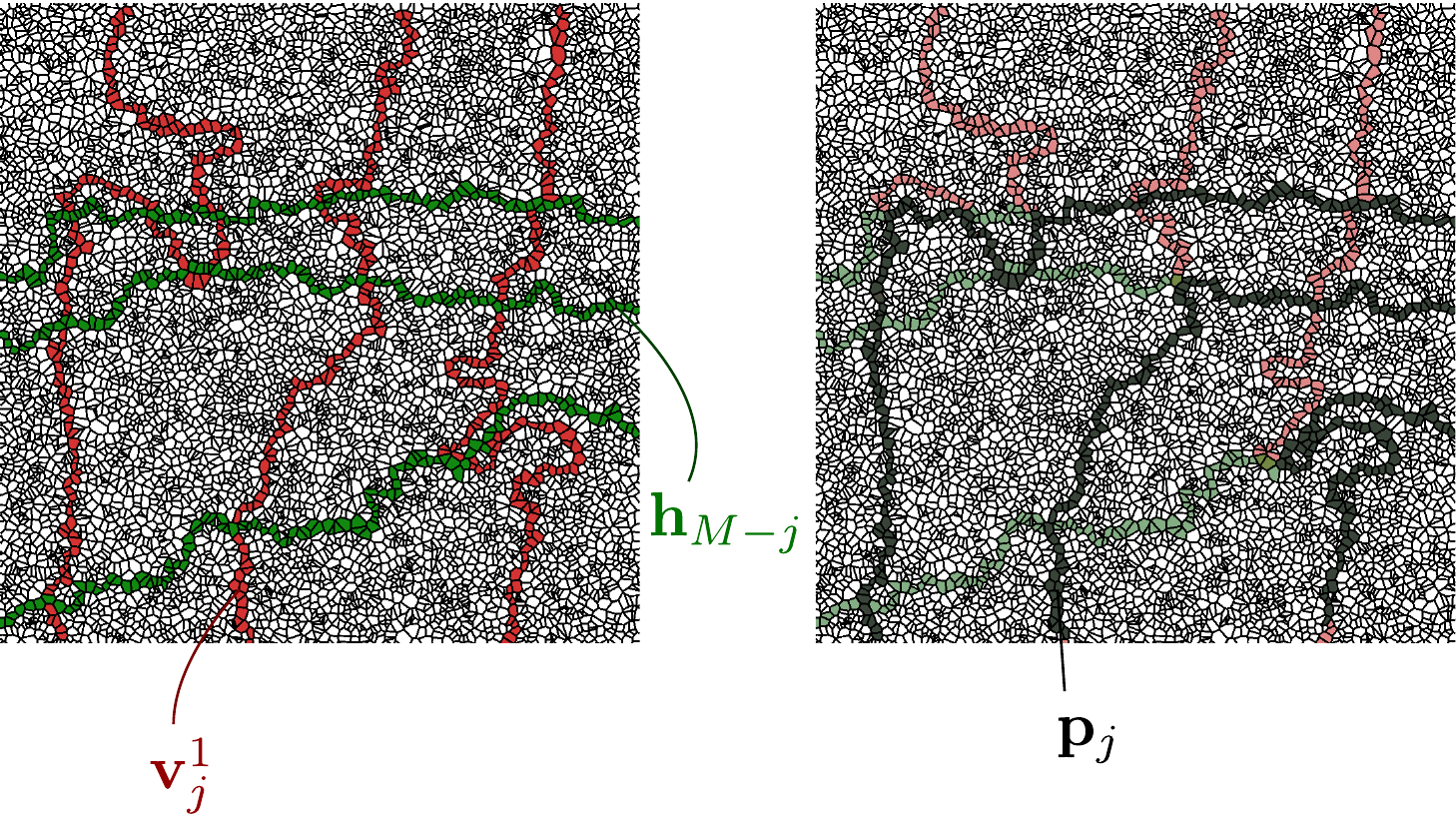}
\caption{The paths $\p_j$ in the box $R$ built in the proof of Theorem \ref{thm:pathConn}.  Again, we are depicting the Voronoi cells of the paths $\mathcal{V}_{\e}(\cdot)$. }
\label{fig:Construction2}
\end{figure}
Consider a generic $\mathbf{v}_j^1$.  Let $R=R_{bl}^v\cap R_{top}^h$,  $\tilde{R}=R_{tr}^v\cap R_{top}^h$.  We define the following path (see for instance Figure \ref{fig:Construction2}).
	\begin{align*}
	\p_j:=\{x\in \v_j^1 : x\bigtriangledown_R \h_{M-j}\} \cup \{y\in \h_{M-j} : \v_j^1\triangleleft_R y\triangleleft_{\tilde{R}} \v_j^2\} \cup\{x \in \v_j^2 : x \bigtriangleup_{\tilde{R}} \h_{M-j} \} 
	\end{align*}
Note that $\p_j$ is a path which has the same starting point than $\v_j^1$ and the same ending point of $\v_j^2$.  Moreover with this definition we note that
	\[
	\mathrm{dist}(\p_j,\p_{j+1})> 3\lambda\e.
	\]
Indeed, suppose that for some $(y_j,y_{j+1})\in \p_j\times \p_{j+1}$ we have
	\[
	|y_j-y_{j+1}|\leq 3\lambda\e.
	\]
Then one of the following is necessarily in force
	\begin{itemize}
	\item[(I)]	 $y_j\in \h_{M-j}$ and $y_{j+1}\in \v_{j+1}^1$;
	\item[(II)] $y_j \in\h_{M-j}$ and $y_{j+1}\in \v_{j+1}^2$;
	\item[(III)] $y_j\in \v_j^1$  and  $y_j\in \h_{M-j-1}$;
	\item[(IV)] $y_j\in \v_j^2$  and  $y_j\in \h_{M-j-1}$.
\end{itemize}
The other possibilities are ruled out by the fact that 
	\[
	\mathrm{dist}(\v_j^1,\v_{j+1}^1)>3\lambda\e , \ \ \ \ \ \mathrm{dist}(\v_j^2,\v_{j+1}^2)>3\lambda\e, \ \ \ \ \ \mathrm{dist}(\h_{M-j}, \h_{M-j-1}^1)>3\lambda\e.
	\]
Now case (I) cannot be attained since it would imply $y_{j+1} \bigtriangledown_R \h_{M-j-1} $ and thus $\h_{M-j}$ would get too close to $\h_{M-j-1}$.  Analogously for case (II), since we would have $y_{j} \triangleleft_{\tilde{R}} \v_{j}^2$ and thus $\v_j^2$ would get close to $\v_{j+1}^2$.  The other cases follow the same line. In particular none of them can be achieved and thus we must get that
	\[
	\mathrm{dist}(\p_j,\p_{j+1})>3\lambda\e.
	\]

By applying this construction we can prolong $\mathbf{v}_j^1$ a bit outside $Q_{1,1}^t$.  If we shift this construction and we repeat it, we can extend each path further until we reach $Q_{k_t,1}^t$.   In this way we are able to obtain a family of $M_{\e,t}$ paths in each rectangle $R_j^v(t)$.  By exploiting the same argument, with the required modification we also obtain the horizontal paths.  This produces a family of vertical and horizontal paths $G_{\e,t}$ which satisfies properties (a), (b), (c), (f) and (g) (from the validity of properties (a.2), (b.2), (c.2), (f.2) and (g.2) of Corollary \ref{cor:onethird}). Property (e) instead follows by the previous argument and the care adopted to junction the paths.  It remains to show that they also meet the requests of property (d).  Up to discarding some paths (operation that never affects the other properties) we can ensure that $M_{\e,t}\leq \Upsilon \frac{t}{\e}$.  Moreover, each $\mathbf{v}^i_m\cap Q_{i,j}^t$ (as well as the vertical ones) is the junction of a finite number (independent of $t,\e$) of paths of length satisfying \eqref{eqn:unifBndas}.  Therefore,  up to increase a bit $\Upsilon$ (but independently of $t,\e$) we can guarantee that property (d) is in force.  This concludes the proof. 
\end{proof}

\subsection*{Acknowledgments}
This work has been supported by PRIN 2017  `Variational methods for stationary and evolution problems with singularities and interfaces'. The authors are members of GNAMPA, INdAM. The authors acknowledge the MIUR Excellence Department Project awarded to the Department of Mathematics, University of Rome Tor Vergata, CUP E83C18000100006.

\bibliography{references}
\bibliographystyle{plain}

\end{document}